\documentclass[11pt, reqno]{amsart}
\usepackage{amsfonts,latexsym,enumerate}
\usepackage{amsmath}
\usepackage{mathtools}
\usepackage{amscd}
\usepackage{float,amsmath,amssymb,mathrsfs,bm,multirow,graphics}
\usepackage[dvips]{graphicx}
\usepackage[percent]{overpic}
\usepackage[foot]{amsaddr}
\usepackage[numbers,sort&compress]{natbib}
\usepackage{geometry}
\usepackage{hyperref}\hypersetup{hidelinks}
\usepackage{enumerate}		
\usepackage[percent]{overpic}
\usepackage{subcaption}
\usepackage{appendix}
\usepackage{tikz}
\usetikzlibrary{calc}
\newcommand{\R}{{\mathbb R}}

\newcommand{\re}{\text{\upshape Re\,}}
\newcommand{\im}{\text{\upshape Im\,}}

\newtheorem{theorem}{Theorem}[section]
\newtheorem*{theorem*}{Theorem}
\newtheorem*{lemma*}{Lemma}
\newtheorem{proposition}[theorem]{Proposition}
\newtheorem{lemma}[theorem]{Lemma}

\theoremstyle{definition}
\newtheorem{definition}[theorem]{Definition}
\newtheorem*{definition*}{Definition}
\newtheorem{remark}[theorem]{Remark}
\newtheorem*{remark*}{Remark}
\newtheorem{RHP}{RH problem}[section]
\newtheorem{Assumption}{Assumption}[section]

\numberwithin{equation}{section}
\newcommand*{\dif}{\mathop{}\!\mathrm{d}}

\title[The defocusing mKdV equation with step-like initial data]
{On the Large-Time Asymptotics of the Defocusing mKdV Equation with Step-like Initial Data}

\author[T.-Y.\@ Xu]{Taiyang Xu$^a$}
\author[Y.-D. \@ Zhang]{Yidan Zhang$^a$}

\address{$^a$School of Mathematical Sciences, Fudan University, 200433 Shanghai, China}
\email{\href{mailto: tyxu19@fudan.edu.cn}{tyxu19@fudan.edu.cn}}
\email{\href{mailto: ydzhang23@m.fudan.edu.cn}{23110840018@m.fudan.edu.cn}}

\begin{document}
\UseRawInputEncoding

\begin{abstract}
We study the Cauchy problem for the defocusing modified Korteweg-de Vries (mKdV) equation with step-like initial data approaching nonzero constants $c_l$ and $c_r$ as $x \to -\infty$ and $x\to+\infty$, respectively.
Assuming $c_l>c_r>0$, the solution exhibits a rarefaction wave structure. We first develop the inverse scattering transform for the solution satisfying these step-like boundary conditions. 
Using the associated scattering data, we prove that there exists a unique global solution of the Cauchy problem and characterize it in terms of a Riemann-Hilbert (RH) problem. 
By applying the nonlinear steepest descent method to this RH problem, we rigorously obtain large-time asymptotics of rarefaction wave solution in three distinct space-time regions, 
each characterized by a different leading order behavior. They are: (I) a left-field region where the solution approaches the left background constant, modulo a small oscillatory correction, 
(II) a central region where the solution exhibits a slowly varying profile that transitions between the two constants, and 
(III) a right-field region where the solution tends to the right background constant, up to an algebraically small correction. 
Rigorous derivations of the leading terms, sub-leading terms as well as the error bounds are presented.
\end{abstract}
\maketitle
\noindent


\setcounter{tocdepth}{1} \tableofcontents

%%%%%%%%%%%%%%%%%%%%%%%%%%%%%%%%%%%

%%%%%%%%%%%%%%%%%%%%%%%%%%%%%%%%%%%
\section{Introduction}
\noindent
We study the Cauchy problem of the defocusing mKdV equation 
\begin{alignat}{2}
&q_t(x,t)-6q^2(x,t)q_{x}(x,t)+q_{xxx}(x,t)=0, &\qquad&x\in\R,\quad t> 0, \label{equ:mkdv} \\
&q(x,0)=q_0(x),&&x\in\R, \label{Initial data}
\end{alignat}
under the assumption that the solution $q(x,t)$ approaches to some nonzero real constants as $x \to \pm \infty$ respectively, 
i.e.,
\begin{equation} \label{boundaryconditions}
q(x,t)\to 
\begin{cases}
c_{l}, & x\to -\infty, \\
c_{r}, & x\to +\infty,
\end{cases}
\end{equation}
where $c_{l}>c_{r}>0$. To ensure compatibility of the boundary conditions \eqref{boundaryconditions} with the evolution
\eqref{equ:mkdv}--\eqref{Initial data}, we require that the initial datum $q_0(x)$ must satisfy
\begin{equation}\label{condition:q0limits}
q_0(x)\to
\begin{cases}
c_{l}, &x\to -\infty,\\
c_{r}, &x\to +\infty.
\end{cases}
\end{equation}
The precise rates of convergence in \eqref{boundaryconditions} and \eqref{condition:q0limits} will be specified in the main theorems.
Since the symmetries $q\rightarrow-q$, $x\rightarrow -x$ as well as $t\rightarrow-t$ keep the form of the defocusing mKdV equation \eqref{equ:mkdv} unchanged, it is sufficient to focus on the scenario where 
$c_l\geqslant \vert c_{r} \vert$ (or $c_r\geqslant \vert c_{l} \vert$). The cases $c_l>c_r\geqslant 0$ and $c_l> 0>c_{r}>-c_l$ lead to quite different asymptotic analysis that a kink soliton region
appears in the later case; see \cite[Figure 1.1]{Wang mkdv}. In the present work, we restrict $c_l>c_{r}>0$ and analyze the large-time dynamics of rarefaction wave solutions. 
For the complementary case $c_r>c_l>0$, the large-time asymptotics of dispersive shock wave solution has been obtained in \cite{Wang mkdv}. 
We also refer the readers to \cite{xzf, Zhang&Xu mkdv} for results on the large-time asymptotics with the symmetric nonzero boundaries 
$\lvert c_l\rvert=\lvert c_r\rvert\neq0$.
% \begin{figure}[!h]
% 	\centering \subfigure[Four saddle points on $\mathbb{R}$\qquad\quad]{ \includegraphics[width=0.25\linewidth]{"case1"}\hspace{1cm}
% 	\label{case 1}} \subfigure[Two saddle points on $\mathbb{R}$\qquad\quad]{\includegraphics[width=0.25\linewidth]{"case2"}\hspace{1cm}
% 	\label{case 2}}	
    
% 	\caption{ \footnotesize {The classification  of the $\text{sign of Im}\theta$ for cases I-III.  In the blue regions, $\text{Im}\theta>0$, which implies that $|e^{2it\theta}|\to 0$ as $t\to\infty$.
% While  in the white regions,   $\text{Im}\theta<0$, which means  $|e^{-2it\theta}|\to 0$ as $t\to\infty$.   The blue curves  $\text{Im}\theta=0$ are the critical lines between the decay and growth regions. } }
% 	\label{figtheta-1}
% \end{figure}
% \FloatBarrier
The defocusing mKdV equation \eqref{equ:mkdv} is a canonical model in mathematical physics, 
describing various nonlinear phenomena such as acoustic waves and phonons in certain 
anharmonic lattices \cite{Zab1967,Ono1992}, and Alfv\'en waves in cold, collisionless 
plasmas \cite{kaku1969,Kha1998}. Early works on the large-time asymptotics of the mKdV equation
\eqref{equ:mkdv} concentrated on initial data vanishing as $x\to\pm\infty$; 
see the works \cite{AblowitzSegur, Zakharov-Manakov-1976} based on the inverse scattering method.
A major advance was made by Deift and Zhou \cite{DZAnn}, who introduced the nonlinear steepest descent 
method for an oscillatory RH problem and used it to obtain detailed large-time asymptotics of the defocusing mKdV equation
with Schwartz class initial data. Such a technique has been widely applied to a variety of integrable PDEs \cite{DIZSurvey}.
We also refer the readers to \cite{CandLdmkdv,lenellsmkdv} for large-time asymptotics of the defocusing mKdV equation with 
vanishing initial data in lower regularity spaces.

The study of large-time asymptotics for Cauchy problems of nonlinear integrable systems with step-like initial data has a long history 
and can be traced back to the pioneering work \cite{GuPitae} of Gurevich and Pitaevskii on the Korteweg-de Vries (KdV) equation. 
Working within the framework of Whitham modulation theory \cite{Whitham1974}, they predicted the emergence of 
highly oscillatory structures, now referred to as dispersive shock waves in the large-time dynamics.
The rigorous mathematical justification of this phenomenon was subsequently done by Khruslov \cite{Khruslov1976} 
via the inverse scattering transform in terms of the Marchenko integral equations and the so-called asymptotic soliton.
Rigorous asymptotic analysis of step-like Cauchy problems for integrable PDEs began with the papers \cite{MonItsKot2009, BuckVena2007}.
Since then, the large-time asymptotics of integrable PDEs with step-like initial data has been extensively 
studied for KdV equation \cite{EgorovaKDVstep-a, EgorovaKDVstep-b}, 
for focusing nonlinear Schr\"odinger equation (NLS) \cite{MonKotSheIMRN,MonLenSheCMP,MonLenSheCMP3},
for focusing mKdV equation \cite{KotMinakovJMP,MinakovJPAMaTheor,KotMinakovJMPAG1,KotMinakovJMPAG2,GraMinakovSIAM}
as well as for some nonlocal integrable PDEs \cite{RDJDE2021,RDCMP,RDJMPAG,xnonlocalmkdv}. 

The aforementioned studies are mostly concerned with Cauchy problems of focusing integrable equations, that is, 
the corresponding Lax operators are non-self-adjoint. Of particular relevance are the recent works \cite{JenkinsNon,LenellsDNLS} on 
the defocusing NLS equation and \cite{Wang mkdv} on the defocusing mKdV equation, whose Lax operators are self-adjoint. 
For instance, in \cite{Wang mkdv} the authors derived the large-time asymptotics for the defocusing mKdV equation with step-like initial data 
of the form
\begin{equation}\label{condition: pure step initial}
q_0(x)=
\begin{cases}
c_{l}, & x<0,\\[4pt]
c_{r}, & x>0,
\end{cases}
\end{equation}
but where $c_r>c_l>0$, corresponding to the dispersive shock wave problem.
In the setting of the {\sl pure} step initial functions as stated 
in \eqref{condition: pure step initial}, the Jost functions and associated scattering data
in the inverse scattering transform could be explicitly constructed, 
which hence greatly simplifies the asymptotic analysis. 
Initial data of this type was also studied in many works, such as 
\cite{MonKotSheIMRN, JenkinsNon, MinakovJPAMaTheor}.
Compared with the dispersive shock regimes studied for KdV, NLS, and mKdV in the above references, the present rarefaction wave regime exhibits a rather different asymptotic profile. In the shock type case, one typically encounters a genuinely oscillatory region bounded by modulation boundaries. However, when $c_l>c_r>0$, the solution is dominated by a slowly varying transition between the two backgrounds.
In the present work, we consider initial data satisfying the boundary conditions
only {\sl asymptotically} as $x\to\pm \infty$. We analyze the large-time asymptotics 
of rarefaction wave solution for $c_l>c_r>0$ (cf. Figure \ref{Fig: RFW}), in contrast to the case of dispersive wave solution for 
$c_r>c_l>0$ (cf. Figure \ref{Fig: DSW}). Indeed, one primary motivation for this work is to understand 
how nonzero boundary conditions imposed only asymptotically 
can be handled for an integrable PDE. Since the boundary limits hold only in the far field ($x\to\pm\infty$), 
one does not have an explicit closed form for the associated scattering
data, particularly the reflection coefficient $r(k)$. Consequently, the analysis 
should find out the necessary properties of $r(k)$ from its definition via the Volterra integral equations 
and then construct suitable analytic approximations developed in \cite{DZAnn} or $\bar\partial$-extensions \cite{Dieng} 
of the scattering data to implement the nonlinear steepest descent analysis.

\begin{figure}[t] 
\centering %图片居中
\includegraphics[width=1.0\textwidth]{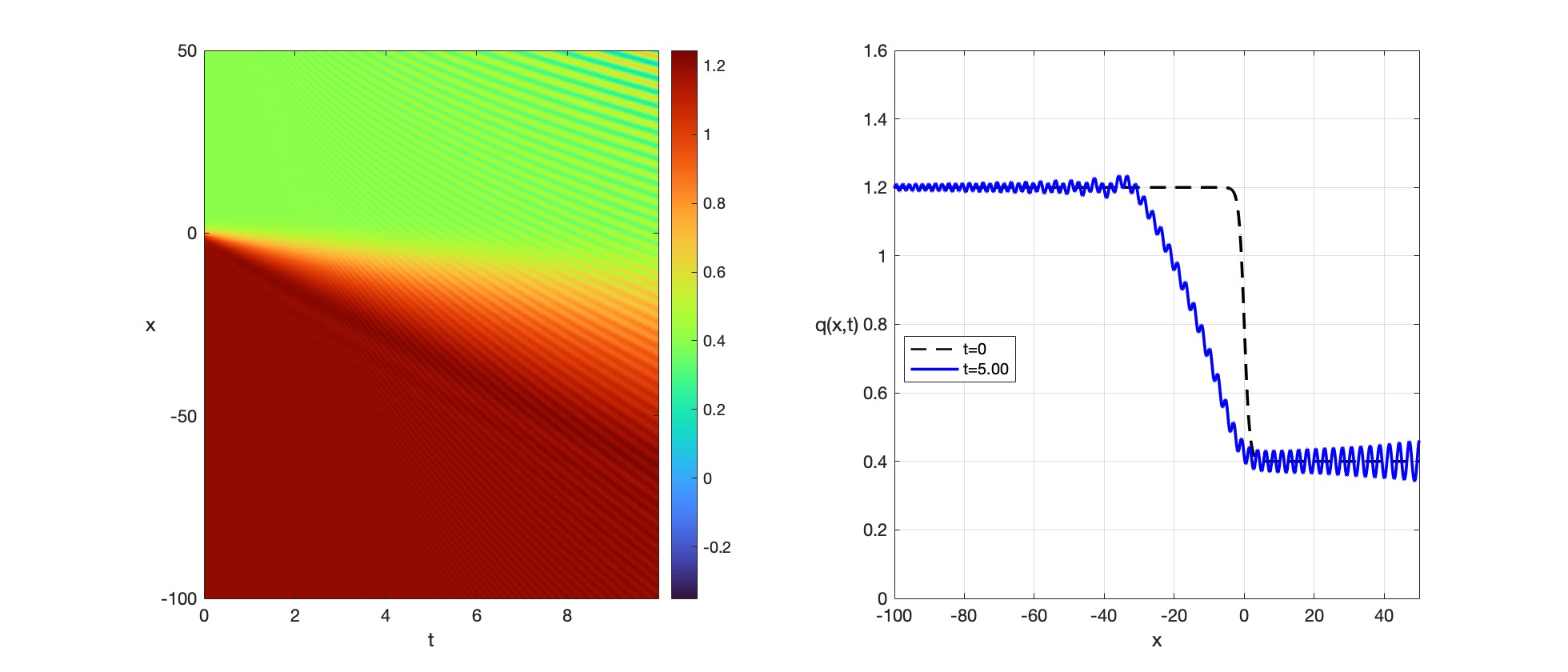} %插入图片，[]中设置图片大小，{}中是图片文件名
\caption{The evolution of $q(x,t)$ for the step-like initial data \eqref{Initial data} with $c_l=1.2 > c_r=0.4$.
(left) Heatmap of $q(x,t)$. (right) The black dashed curve shows the initial profile $q(x,0)$, while the blue curve displays $q(x,t)$ at $t=5$.
The solution develops a rarefaction wave, exhibiting a slowly varied profile between the two constant states.}
\label{Fig: RFW} %用于文内引用的标签
\end{figure}

\begin{figure}[htbp] 
\centering %图片居中
\includegraphics[width=1.0\textwidth]{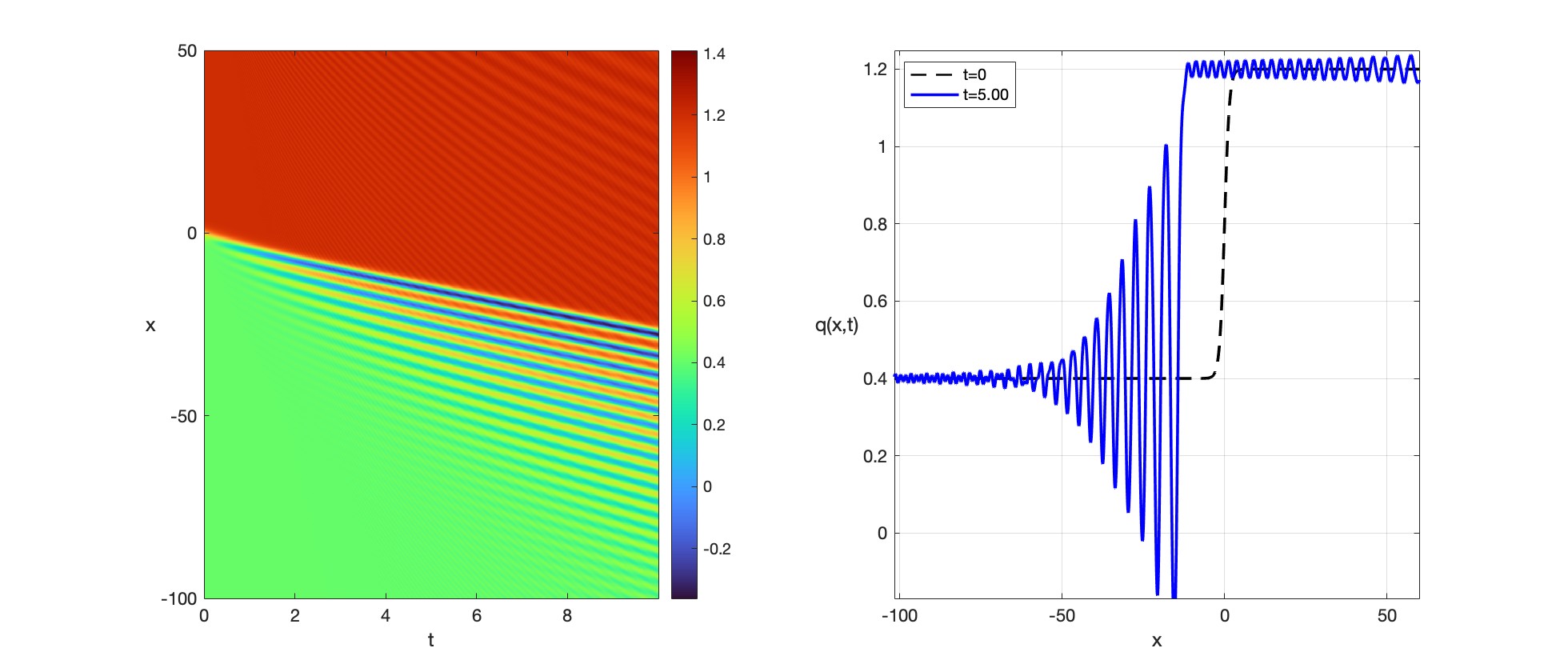} %插入图片，[]中设置图片大小，{}中是图片文件名
\caption{Evolution of $q(x,t)$ for the step-like initial data \eqref{Initial data} with $c_l=0.4 < c_r=1.2$.
(left) Heatmap of $q(x,t)$. (right) The black dashed curve shows the initial profile $q(x,0)$, while the blue curve displays $q(x,t)$ at $t=5$.
The solution develops a dispersive shock wave, featuring a region of rapid oscillations between the two constant states.}
\label{Fig: DSW} %用于文内引用的标签
\end{figure}

\subsection*{Outline of this paper}
Our main theorems are stated in Section \ref{sec: main results}.
In Section \ref{sec:IST}, we present the spectral analysis and inverse scattering transform to 
characterize the solution of the Cauchy problem \eqref{equ:mkdv}--\eqref{Initial data} in terms of an RH problem.
Sections \ref{sec:asymptotic analysis in RI}--\ref{sec:asymptotic analysis in RIV} are devoted to the asymptotic analysis of the obtained 
RH problem in different space-time regions, from which the main results stated in Theorem \ref{thm:mainthm} below
are proved.  

\subsection*{Notations}
Throughout this paper, we adopt the following notations.
\begin{itemize}
    \item As usual, the classical Pauli matrices $\{\sigma_j\}_{j=1,2,3}$ are defined by
	\begin{equation*}
		\sigma_1:=\begin{pmatrix}0 & 1 \\ 1 & 0\end{pmatrix}, \quad
		\sigma_2:=\begin{pmatrix}0 & -i \\ i & 0\end{pmatrix}, \quad
		\sigma_3:=\begin{pmatrix}1 & 0 \\ 0 & -1\end{pmatrix}.
	\end{equation*}
    For a $2\times 2$ matrix $A$, we also define
	$$e^{\hat{\sigma}_j}A:=e^{\sigma_j}Ae^{-\sigma_j}, \quad j=1,2,3.$$ 
	\item For a complex-valued function $f(z)$, we use $f^{*}(z):=\overline{f(\bar{z})}$ for $z\in\mathbb{C}$ to denote its Schwartz conjugation.
    \item For a region $U\subseteq \mathbb{C}$, we use $U^*$ to denote the conjugated region of $U$.
    \item For $1 \leqslant p < \infty$, the $L^p$-space is defined as:
                \begin{equation*}
                L^p(\mathbb{R}) = \left\{ f: \mathbb{R} \to \mathbb{C} \ \middle| \ f \text{ is measurable and } \|f\|_p < \infty \right\},
                \end{equation*}
                where the $L^p$-norm is given by $\|f\|_p = \left( \int_{\mathbb{R}} |f(x)|^p  \dif x \right)^{1/p}$.
                For $p = \infty$, the $L^\infty$-space is defined by
                \begin{equation*}
                L^\infty(\mathbb{R}) = \left\{ f: \mathbb{R} \to \mathbb{C} \ \middle| \ f \text{ is measurable and } \|f\|_\infty < \infty \right\},
                \end{equation*}
                where the essential supremum norm is given by $\|f\|_\infty = \inf \left\{ M \geqslant 0 \ | \ |f(x)| \leqslant M \text{ a.e.} \right\}.$
                The $\mathcal{C}^{N}(\Omega)$, $N=1,2,\dots,\infty$ is defined as the space of $N$-times continuously differentiable functions on $\Omega$.
    \item For all smooth oriented curve $\Sigma$, the Cauchy operator $C$ on $\Sigma$ is defined  by
	\begin{align*}
		Cf(z)=\frac{1}{2\pi i}\int_{\Sigma}\frac{f(\zeta)}{\zeta-z}\dif\zeta, \qquad  z\in\mathbb{C}\setminus \Sigma.
	\end{align*}
	Given a function $f \in L^p(\Sigma)$, $1\leqslant p<\infty$,
	\begin{align*}
		C_\pm f(z):=\lim_{z'\to z\in\Sigma}\frac{1}{2\pi i}\int_{\Sigma}\frac{f(\zeta)}{\zeta-z'}\dif\zeta
	\end{align*}
   stands for the positive / negative (according to the orientation of $\Sigma$) non-tangential boundary value of $Cf$.
   It is also used to adopting $f_{\pm}(z)$ to represent the non-tangential limits from the positive / negative side respectively, 
   i.e., $f_{\pm}(z)=\lim_{{\rm positive/negative \ side}\ni \zeta\to z}f(\zeta)$. 
   \item If $A$ is a matrix, then $(A)_{ij}$ stands for its $(i,j)$-th entry, and $[A]_j$ represents the $j$-th column. We use $A^{\rm H}$ to denote its conjugate transpose, which means $A^{\rm H}=\bar A^{\rm T}$.
   \item We use $\mathbb{C}^+:=\{k| \im k>0\}$ and $\mathbb{C}^-:=\{k| \im k<0\}$ to represent the  upper and lower half plane respectively. Similarly, we introduce $\mathbb{R}^+:=(0,+\infty)$ and $\mathbb{R}^-:=(-\infty,0)$.
   \item We use the notation $a\lesssim b$ (resp.\ $a \gtrsim b$) to indicate that $a\leqslant Cb$ (resp.\ $a\geqslant Cb$) for some generic positive constant $C$.
   \item The matrix-valued functions arising from RH transformations, such as $M^{(\infty)}$ and $M^{(err)}$, are defined locally in the subsection where they are used. 
\end{itemize}
For the reader's convenience, we summarize the principal spectral functions, phase functions, and contour notations used in the RH deformation throughout the paper; see Table \ref{tab:notation-summary}.
\begin{table}[htbp]
\centering
\renewcommand{\arraystretch}{1.15}
\small
\begin{tabular}{@{}p{0.19\textwidth}p{0.56\textwidth}p{0.19\textwidth}@{}}
\hline
Notation & Role in the analysis & First definition / illustration \\
\hline
$a(k), b(k), r(k)$ & Scattering coefficients and reflection coefficient; they determine the jump matrix of the basic RH problem. & \eqref{def:S(k)}--\eqref{equ:jump V(k)} \\
$X_j(k), \Delta_j(k), \chi_j(k)$ & Background spectral quantities associated with the branch cuts at $\pm c_j$ and the model factors in the Jost solutions. & \eqref{def:X_j(k)}--\eqref{def:chi} \\
$\theta(\xi;k)$ & Original oscillatory phase in the RH problem for $M$. & Below \eqref{eq: def of M(x,t;k)} \\
$r_2(k),\; r_a,\; r_{2,a},\; r_r$ & Auxiliary reflection data used in jump factorizations and analytic approximations during contour deformation. & \eqref{factor11}--\eqref{factor21} and the corresponding approximation propositions \\
$g_{\textup{\uppercase\expandafter{\romannumeral1}}},\; g_{\textup{\uppercase\expandafter{\romannumeral2}}},\; g_{\textup{\uppercase\expandafter{\romannumeral3}}}$ & Deformed phase functions introduced in Regions I--III to describe the sign distribution of the phase and guide the contour deformation. & \eqref{equ:g function RI}, \eqref{equ:g function RII}, \eqref{equ:g function RIII} \\
$D_{\textup{\uppercase\expandafter{\romannumeral1}}},\; D_{\textup{\uppercase\expandafter{\romannumeral2}}}$ & Szeg\H{o} type functions used in Region I and II, respectively. & \eqref{equ:def D function RI}, \eqref{equ:def D function RII} \\
$\Gamma^{(3)}$ & Jump contour after opening lenses; its precise geometry depends on the space-time region under consideration. & Figures \ref{fig:U_j domain RI} and \ref{fig:U_j domain RII} \\
$\Gamma^{(r)},\; \Gamma^{(l)}$ & Local contours near the stationary points, used in the construction of local parametrices. & \eqref{def:Gamma^(ell) RI} and its Region II analogue \\
$\Gamma^{(err)}$ & Contour for the final small-norm RH problem. & Figures \ref{fig:jump contour Merr RI} and \ref{fig:jump contour Merr RII} \\
$\Gamma_1,\; \Gamma_1^*$ & Parallel contours used in the deformation for $\mathcal{R}_{\textup{\uppercase\expandafter{\romannumeral3}}_b}$. & Figure \ref{fig:deformtosaddle} \\
\hline
\end{tabular}\caption{Summary of the principal spectral functions and contour notations used in the RH deformation and phase analysis.}
\label{tab:notation-summary}
\end{table}

\section{Main results}\label{sec: main results}
Besides the boundary condition \eqref{condition:q0limits}, we assume that the initial
data $q_0(x)$ satisfies the following conditions.
\begin{Assumption}\label{assumption on q_0}
    \hfill
   \begin{itemize}
       \item [\rm(a)] $x^m(q_0-c_l)|_{\mathbb{R}^-}\in L^1(\mathbb{R}^-)$, $x^m\left(q_0-c_r\right)|_{\mathbb{R}^+}\in L^1(\mathbb{R}^+)$, $m=0,\dots,10$.
       \item [\rm(b)] $q_0\in\mathcal{C}^4(\mathbb{R})$, $\partial_x^n q_0\in L^\infty(\mathbb{R})$, $n=0,\dots,4$.
   \end{itemize}
\end{Assumption}

As a preliminary step, the first theorem establishes the existence of a global solution to the Cauchy problem \eqref{equ:mkdv}--\eqref{Initial data} under the Assumption \ref{assumption on q_0}. 
To present this result, we firstly introduce the following definition that comes from \cite{lenellsmkdvfourier} with a minor modification.
\begin{definition}[\cite{lenellsmkdvfourier}]\label{def: global solution}
Suppose $q_0: \mathbb{R} \rightarrow \mathbb{R}$ satisfies Assumption \ref{assumption on q_0}. 
We say that a function $q: \mathbb{R} \times[0, \infty) \rightarrow \mathbb{R}$ is a global solution of the 
Cauchy problem \eqref{equ:mkdv}--\eqref{Initial data} with initial data $q_0$ if the following conditions hold:
\begin{itemize}
 \item[(a)] $q: \mathbb{R}\times[0,+\infty)$ is $\mathcal{C}^3$ in $x$ and $\mathcal{C}^1$ in $t$.
 \item[(b)] $q$ satisfies the defocusing mKdV equation \eqref{equ:mkdv} for all $(x, t) \in \mathbb{R} \times[0, \infty)$.
 \item[(c)] $q(x, 0)=q_0(x)$ for all $x \in \mathbb{R}$.
 \item[(d)] For each $t \in[0,+\infty)$, $q(x,t)$ satisfies the following step-like boundary conditions
  $$
    q(x, t)= \begin{cases}c_l+o(1), & x \rightarrow-\infty, \\ c_r+o(1), & x \rightarrow+\infty .\end{cases}
  $$
\end{itemize}
\end{definition}

With the above definition, we have the following existence result.
\begin{theorem}\label{thm: global solution existence}
    Suppose $q_0$ satisfies the Assumption \textup{\ref{assumption on q_0}}. Then the Cauchy problem \eqref{equ:mkdv}--\eqref{Initial data}
    with initial data $q_0$ has a unique global solution $q(x,t)$ shown in the Definition \textup{\ref{def: global solution}}, which could be expressed in terms of the solution of 
    the formulated RH problem \textup{\ref{RHP:basic RHP}} via the reconstruction formula \eqref{equ:recovering formula}.
\end{theorem}
\begin{proof}
    The proof of this theorem follows a similar approach to that in \cite{lenellsmkdvfourier} for the defocusing mKdV equation in the quarter plane. 
    However, the main difference is to ensure that $q(x,t)$ satisfy the step-like initial data instead; 
    a sketch of the proof is provided in Section \ref{sec:IST}.
\end{proof}

We next give the precise definition of the different asymptotic regions in the $(x,t)$-half plane.
\begin{definition}\label{def:divide regions R_I-R_IV}
    For any constants $c_l$ and $c_r$ satisfying $c_l>c_{r}>0$, 
    we define
    \begin{itemize}
        \item Left field: $\mathcal{R}_{\textup{\uppercase\expandafter{\romannumeral1}}}:=\left\{(x,t)| \xi<-\frac{c_{l}^2}{2}\right\}$.
        \item Slowly varied region: $\mathcal{R}_{\textup{\uppercase\expandafter{\romannumeral2}}}:=\left\{(x,t)|-\frac{c^2_l}{2}<\xi<-\frac{c_{r}^2}{2}\right\}$.
        \item Right field: $\mathcal{R}_{\textup{\uppercase\expandafter{\romannumeral3}}}:=\mathcal{R}_{\textup{\uppercase\expandafter{\romannumeral3}}_a}\cup\mathcal{R}_{\textup{\uppercase\expandafter{\romannumeral3}}_b}$,
        where
        $$
            \mathcal{R}_{\textup{\uppercase\expandafter{\romannumeral3}}_{a}}:=\left\{(x,t)|-\frac{c_{r}^2}{2}<\xi<-\frac{c_r^2}{6} \right\}
            \quad {\rm and} \quad \mathcal{R}_{\textup{\uppercase\expandafter{\romannumeral3}}_{b}}:=\left\{(x,t)|\xi>-\frac{c_{r}^2}{6}\right\}.
        $$
    \end{itemize}
Here $\xi=x/(12t)$; see Figure \ref{fig:cone} for an illustration.
\end{definition}

\begin{figure}[H]
    \begin{center}
    \begin{tikzpicture}[node distance=2cm]
    %\filldraw[yellow!20,line width=2] (0,0)--(6,1)--(7,0);
    %\filldraw[purple!20,line width=2] (0,0)--(-6,1)--(-7,0);
    %\draw[red, dashed](0,0)--(6,1)node[above,black]{$x=6t$};
    \draw[yellow!20, fill=yellow!20] (0,0)--(-5,2)--(-5,0)--(0, 0);
    \draw[green!20, fill=blue!20] (0,0 )--(-4,4)--(-5,4)--(-5,2)--(0,0);
    \draw[blue!20, fill=green!20] (0,0 )--(4,4)--(-4,4);
    \draw[green!20, fill=green!20] (0,0 )--(4,4)--(4,0)--(0,0);
    \draw[->](-5.5,0)--(5.5,0)node[right]{$x$};
    \draw[->](0,0)--(0,4)node[above]{$t$};
    \draw[red, dashed](0,0)--(-5,2)node[above,black]{$\xi=-\frac{c_{l}^2}{2}$};
    \draw[red, dashed](0,0)--(-4,4)node[above,black]{$\xi=-\frac{c_{r}^2}{2}$};
    \draw[red, dashed](0,0)--(-1,4)node[above,black]{$\xi=-\frac{c_{r}^2}{6}$};
    \node[below]{$0$};
    \coordinate (A) at (-4.5, 0.8);
	\fill (A) node[right] {$\mathcal{R}_{\textup{\uppercase\expandafter{\romannumeral1}}}$};
    \coordinate (B) at (-3.5, 2);
	\fill (B) node[right] {$\mathcal{R}_{\textup{\uppercase\expandafter{\romannumeral2}}}$};
    \coordinate (C) at (-2, 2.5);
	\fill (C) node[right] {$\mathcal{R}_{\textup{\uppercase\expandafter{\romannumeral3}}_a}$};
    \coordinate (D) at (2, 2);
	\fill (D) node[left] {$\mathcal{R}_{\textup{\uppercase\expandafter{\romannumeral3}}_b}$};
    \end{tikzpicture}
    %\flushleft{\footnotesize {\bf Figure $\ref{cone}$}  }
    \caption{The different asymptotic regions of the $(x,t)$-half plane.} \label{fig:cone}
    \end{center}
\end{figure}

Physically, the three asymptotic regions reflect how the step-like data propagate at different speeds, moving from the left background, through a rarefaction wave, to the right background.
Large-time asymptotics of the solution $q(x,t)$ in each of the regions given in Definition \ref{def:divide regions R_I-R_IV} 
are main results of the present work.
\begin{theorem}\label{thm:mainthm}
    Let $q(x,t)$ be the global solution of the Cauchy problem \eqref{equ:mkdv}--\eqref{Initial data}
    for the defocusing mKdV equation over the real line under Assumption \textup{\ref{assumption on q_0}}, and denote
    by $r(k)$ the reflection coefficient. As $t\rightarrow+\infty$, we have the following asymptotics of $q(x,t)$ 
    in the regions $\mathcal{R}_\textup{\uppercase\expandafter{\romannumeral1}}-\mathcal{R}_{\textup{\uppercase\expandafter{\romannumeral3}}}$ 
    given in Definition \textup{\ref{def:divide regions R_I-R_IV}}.
    \begin{itemize}
    \item[{\rm (I)}] For $\xi\in\mathcal{R}_{\textup{\uppercase\expandafter{\romannumeral1}}}$, we have
    \begin{align}\label{asy formula:RI}        
        q(x,t)=c_l+t^{-1/2}f_{\textup{\uppercase\expandafter{\romannumeral1}}}(\xi)+\mathcal{O}(t^{-1}\log t),
    \end{align}
    where
    \begin{align*}
        f_\textup{\uppercase\expandafter{\romannumeral1}}(\xi)=\frac{\nu(\eta_l)^{1/2}\left(-\xi-\frac{c_l^2}{2}\right)^{1/4}}{\left[3(-\xi+\frac{c_l^2}{2})\right]^{1/2}}\cos\left[16t\left(-\xi-\frac{c_l^2}{2}\right)^{3/2}-\nu(\eta_l)\log \left(\frac{48t\left(-\xi+\frac{c_l^2}{2}\right)}{\left(-\xi-\frac{c_l^2}{2}\right)^{1/2}}\right)+\Theta(\xi)\right].
    \end{align*}
    Here,
\begin{align*}
    &\nu(\eta_l)=-\frac{1}{2\pi}\log(1-|r(\eta_l)|^2),\quad \eta_l(\xi)=\sqrt{-\xi+\frac{c_l^2}{2}},\\
    &\Theta(\xi)=\frac{\pi}{4}-\arg r(\eta_l)-\arg \Gamma(-i\nu(\eta_l))+2d_{\textup{I},\eta},\quad X_{l}(k)=\sqrt{k^2-c_l^2},\\
    &d_{\textup{I},\eta}=-\frac{(\eta_l^2-c_l^2)^{\frac{1}{2}}}{2\pi }\left[\left(\int_{-c_l}^{-c_{r}}+\int_{c_{r}}^{c_l}\right)\frac{\log r_{+}(s)}{X_{l+}(s)(s-\eta_l)}\dif s+\int_{-\eta_l}^{- c_l}\frac{\log(1-r(s)r^{*}(s))}{X_{l}(s)(s- \eta_l)}\dif s\right]-\nu(\eta_l)\log (\eta_l-c_l).
\end{align*}

    \item[{\rm (II)}] For $\xi\in\mathcal{R}_{\textup{\uppercase\expandafter{\romannumeral2}}}$, we have
    \begin{align}\label{asy formula:RII}
        q(x,t)=\sqrt{-2\xi}+t^{-1}f_{\textup{\uppercase\expandafter{\romannumeral2}}}(\xi)+\mathcal{O}(t^{-2}),
    \end{align}
    where
    \begin{equation*}
    f_{\textup{\uppercase\expandafter{\romannumeral2}}}(\xi)=\frac{\left[7+5(2\eta)^{\frac{1}{2}}\right]\hat C_\eta}{288\eta^2},
    \end{equation*}
    with 
    \begin{align*}
         &\eta(\xi)=\sqrt{-2\xi},\\
         &\hat C_\eta=-\frac{\sqrt{2\eta}}{2\pi }\left[\int_{c_r}^\eta\frac{\frac{\arg r_+(s)}{\sqrt{\eta+s}}-\frac{\arg r_+(\eta)}{\sqrt{2\eta}}}{(s-\eta)\sqrt{\eta-s}}\dif s+\frac{2\arg r_+(\eta)}{\sqrt{2\eta(\eta-c_r)}}-\int_{-\eta}^{-c_r}\frac{\arg r_+(s)}{\sqrt{\eta^2-s^2}(s-\eta)}\dif s\right].
    \end{align*}

    \item[\text{(III\textsubscript{a})}] For $\xi\in\mathcal{R}_{{\textup{\uppercase\expandafter{\romannumeral3}}}_a}$, we have
    \begin{align}\label{asy formula:RIII} q(x,t)=c_{r}+t^{-2}f_{\textup{\uppercase\expandafter{\romannumeral3}}}(\xi)+\mathcal{O}(t^{-3}),
    \end{align}
    where
    \begin{align*}      &f_{\textup{\uppercase\expandafter{\romannumeral3}}}(\xi)=-\frac{(\mu-1)(\mu-9)}{4608c_r(c_r^2+2\xi)^2},
    \end{align*}
    with $\mu={4\left[\arg r(c_r)\right]^2}/{\pi^2}$.

    \item[\text{(III\textsubscript{b})}] For $\xi\in\mathcal{R}_{{\textup{\uppercase\expandafter{\romannumeral3}}}_b}$, we have
    \begin{align}\label{asy formula:RIV}
        q(x,t)=c_{r}+\mathcal{O}(t^{-2}).
    \end{align}
\end{itemize}
\end{theorem}
\begin{proof}
    See sections \ref{sec:asymptotic analysis in RI}--\ref{sec:asymptotic analysis in RIV} for the detailed proofs of the asymptotic formulas \eqref{asy formula:RI}--\eqref{asy formula:RIV} in each of the regions $\mathcal{R}_{\textup{\uppercase\expandafter{\romannumeral1}}}$, $\mathcal{R}_{\textup{\uppercase\expandafter{\romannumeral2}}}$, $\mathcal{R}_{\textup{\uppercase\expandafter{\romannumeral3}}_a}$ and $\mathcal{R}_{\textup{\uppercase\expandafter{\romannumeral3}}_{b}}$ respectively.
\end{proof}
We clarify that no soliton structures emerge under the considered step-like initial data satisfying Assumption \ref{assumption on q_0}; see part (b) of Proposition \ref{prop:a,b,r} below.
In the left-most region $\mathcal{R}_{\textup{\uppercase\expandafter{\romannumeral1}}}$, the leading term in the formula \eqref{asy formula:RI} is given by the constant $c_l$, which is compatible with the boundary condition \eqref{boundaryconditions} for $x<0$. 
The asymptotic solution in $\mathcal{R}_{\textup{\uppercase\expandafter{\romannumeral1}}}$ admits a sub-leading term of the form $t^{-1/2}f_{\textup{\uppercase\expandafter{\romannumeral1}}}$, constructed via a parabolic cylinder parametrix. 
For $\xi\in\mathcal{R}_{\textup{\uppercase\expandafter{\romannumeral2}}}$, a slowly varying region emerges, bridging the left and right fields. Here, the leading term is given by the slowly varying factor $\sqrt{-\frac{x}{6t}}$, while the sub-leading term $t^{-1}f_{\textup{\uppercase\expandafter{\romannumeral2}}}$ is derived from an Airy parametrix. 
As expected, when $\xi$ approaches $-c_l^2/2$, the leading term of \eqref{asy formula:RII} matches the leading asymptotics in \eqref{asy formula:RI}; similarly, when $\xi$ approaches $-c_{r}^2/2$, the leading term of \eqref{asy formula:RII} matches that of \eqref{asy formula:RIII}. 
The asymptotics in the regions $\mathcal{R}_{\textup{\uppercase\expandafter{\romannumeral3}}_a}$ and $\mathcal{R}_{\textup{\uppercase\expandafter{\romannumeral3}}_b}$ shares the same leading term. However, their sub-leading terms behave differently: in region $\mathcal{R}_{\textup{\uppercase\expandafter{\romannumeral3}}_a}$, the sub-leading term $t^{-2}f_{\textup{\uppercase\expandafter{\romannumeral3}}}$ arises from a Bessel parametrix, whereas in region $\mathcal{R}_{\textup{\uppercase\expandafter{\romannumeral3}}_b}$, the error term $\mathcal{O}(t^{-2})$ 
is from the non-analytic part of the reflection coefficient on the real line.
We also note that the asymptotic formulas given above do not apply to the critical lines $\xi=-c_l^2/2$ and $\xi=-c_r^2/2$, 
where the transient asymptotics are recently established in \cite{FanWangZhang2026}.

\section{Spectral analysis and inverse scattering transform}\label{sec:IST}
\subsection{Direct scattering transform}
As the member of AKNS hierarchy \cite{AKNS}, the Lax pair of
the defocusing mKdV equation \eqref{equ:mkdv} is given by
\begin{align}\label{equ:lax pair}
    \left\{
        \begin{aligned}
        &\Phi_x+ik\sigma_3\Phi=Q\Phi, \\
        &\Phi_t+4ik^3\sigma_3\Phi=V\Phi,
        \end{aligned}
    \right.
\end{align}
where $\Phi(x,t;k)$ is a $2\times 2$ matrix-valued function with
the spectral parameter $k\in\mathbb{C}$. Here, $Q(x,t)$ and $V(x,t)$
are some matrices associated with the potential function $q(x,t)$, defined by
\begin{align}
    &Q(x,t)=\begin{pmatrix}0 & q(x,t) \\ q(x,t) & 0\end{pmatrix}, \label{equ:Q(x,t)}\\
    &V(x,t)=4k^2Q(x,t)+2ik\sigma_3\left(Q_x\left(x,t\right)-Q^2\left(x,t\right)\right)+2Q^3(x,t)-Q_{xx}(x,t).\label{equ:V(x,t)}
\end{align}
The compatible condition $\Phi_{xt} = \Phi_{tx}$ of the Lax pair \eqref{equ:lax pair} is equivalent to the mKdV equation \eqref{equ:mkdv}.

For $j\in\{l,r\}$, substituting $q(x,t)=c_{j}$ into \eqref{equ:lax pair}, we obtain
the corresponding ``background'' Lax pair
\begin{equation}\label{laxpair:background}
     \left\{
        \begin{aligned}   &\Phi^b_{j,x}+ik\sigma_3\Phi^b_{j}=Q^b_j\Phi^b_{j}, \\
        &\Phi^b_{j,t}+4ik^3\sigma_3\Phi^b_{j}=V^b_j\Phi^b_{j},
        \end{aligned}
    \right.
\end{equation}
    where
    \begin{equation}\label{equ:Q_j^b,V_j^b}
        Q^b_j=\begin{pmatrix}0 & c_{j} \\ c_j & 0\end{pmatrix},\quad V^b_j=\begin{pmatrix}
            -2ic_j^2k& 4c_j k^2+2c_j^3\\4c_j k^2+2c_j^3&2ic_j^2k
        \end{pmatrix}.
    \end{equation}
The explicit solutions to the ``background'' Lax pair \eqref{laxpair:background} are given by
\begin{equation}\label{equ:Phi_j^p}
    \Phi_{j}^{b}(x,t;k)=\Delta_{j}(k)e^{-i(X_{j}(k)x+\Omega_{j}(k)t)\sigma_3},\quad j\in\{l,r\},
\end{equation}
where
\begin{align}\label{def:X_j(k)}
    X_{j}(k)=\sqrt{k^2-c_{j}^2}, \quad \Omega_{j}(k)=2(2k^2+c_{j}^2)X_{j}(k),\quad j\in\{l,r\},
\end{align}
and
\begin{equation}\label{equ:Delta_j}
    \Delta_{j}(k)=\frac{1}{2}
    \left(
        \begin{array}{cc}
        \chi_{j}(k)+\chi_{j}^{-1}(k) & i\left(\chi_{j}(k)-\chi_{j}^{-1}(k)\right)\\
        -i\left(\chi_{j}(k)-\chi_{j}^{-1}(k)\right) & \chi_{j}(k)+\chi_{j}^{-1}(k)
        \end{array}
    \right), \quad j\in\{l,r\}.
\end{equation}
Here, the function $\chi_j$ is defined by 
\begin{equation}\label{def:chi}
    \chi_{j}: \mathbb{C}\backslash[-c_{j},c_{j}]\rightarrow\mathbb{C}, \quad \chi_{j}(k)=\left(\frac{k-c_{j}}{k+c_{j}}\right)^{\frac{1}{4}}, \quad j\in\{l,r\},
\end{equation}
with the branch cut being chosen such that $\chi_j(k)=1+\mathcal{O}(k^{-1})$ as $k\rightarrow\infty$.

For $j\in\{l,r\}$, we denote $\Phi^{b}_j(x;k):=\Phi^{b}_{j}(x,0;k),\;Q_0(x):=Q(x,0)$, where $\Phi^{b}_{j}(x,0;k)$ is
defined in \eqref{equ:Phi_j^p} for $t=0$. To proceed, we consider the Lax pair \eqref{equ:lax pair}
for $t=0$ and define the Jost functions $\Phi_j(x;k):=\Phi_j(x,0;k)$, which satisfy the $x$-part of \eqref{equ:lax pair} and admit the following asymptotic conditions
\begin{align}
    &\Phi_{l}(x;k)=\Phi_{l}^{b}(x;k)\left(I+o(1)\right), \quad x\rightarrow-\infty, \quad k\in\mathbb{R}\setminus\{-c_l,c_l\},\label{equ:jostfunc-l}\\
    &\Phi_{r}(x;k)=\Phi_{r}^{b}(x;k)\left(I+o(1)\right), \quad x\rightarrow+\infty, \quad k\in\mathbb{R}\setminus\{-c_r,c_r\}.\label{equ:jostfunc-r}
\end{align}
By Lax pairs \eqref{equ:lax pair} and \eqref{laxpair:background}, for $j\in\{l,r\}$, the Jost functions $\Phi_j(x;k)$ are defined via the Volterra integral equations
\begin{align}\label{equ:VolterraforPhi}
    \Phi_{j}(x;k)=\Phi^{b}_{j}(x;k)+\int_{\infty_j}^{x}\Phi^b_j(x;k)(\Phi^b_j)^{-1}(y;k)\left[Q_0(y)-Q^b_j\right]\Phi_{j}(y;k)\dif y,
\end{align}
where $\infty_r:=+\infty$, $\infty_l:=-\infty$ and $Q^b_j$ is defined in \eqref{equ:Q_j^b,V_j^b}.

Some basic properties of the Jost functions $\Phi_{l}$ and $\Phi_{r}$ are summarized in the following proposition,
which can be established in a standard manner by analyzing the integral equation \eqref{equ:VolterraforPhi} and 
the spatial part of Lax pair \eqref{equ:lax pair}.
\begin{proposition}\label{prop: properties of Phi(r,j)}
    Under the Assumption \textup{\ref{assumption on q_0}} on the initial data, the Jost functions $\Phi_{l}$ and $\Phi_{r}$ defined by \eqref{equ:VolterraforPhi} have the following properties for $j\in\{l,r\}$:
    \begin{itemize}
        \item[(a)] For each $x\in\mathbb{R}$, we have\\
        $[\Phi_r]_1(x; k)$ is holomorphic for $k\in\mathbb{C}^{-}$ and has continuous extension to $\overline{\mathbb{C}^{-}}\backslash\{-c_{r}, c_{r}\}$,\\
        $[\Phi_r]_2(x; k)$ is holomorphic for $k\in\mathbb{C}^{+}$ and has continuous extension to $\overline{\mathbb{C}^{+}}\backslash\{-c_{r}, c_{r}\}$,\\
        $[\Phi_l]_1(x; k)$ is holomorphic for $k\in\mathbb{C}^{+}$ and has continuous extension to $\overline{\mathbb{C}^{+}}\backslash\{-c_{l}, c_{l}\}$,\\
        $[\Phi_l]_2(x; k)$ is holomorphic for $k\in\mathbb{C}^{-}$ and has continuous extension to $\overline{\mathbb{C}^{-}}\backslash\{-c_{l}, c_{l}\}$.
        \item[(b)] $\det\Phi_j(x;k)=1$, \quad$k\in\mathbb{R}\setminus[-c_j,c_j]$.
        \item[(c)] As $k\rightarrow\infty$, we have 
        \begin{align*}
                \left([\Phi_l(x;k)]_1, [\Phi_r(x;k)]_2\right)e^{ikx\sigma_3}=I+\mathcal{O}(k^{-1}), \quad k\in \mathbb{C}^{+},\\
                \left([\Phi_r(x;k)]_1, [\Phi_l(x;k)]_2\right)e^{ikx\sigma_3}=I+\mathcal{O}(k^{-1}), \quad k\in \mathbb{C}^{-}.
        \end{align*}
        \item[(d)] $\Phi_j(x;k)$ admits the following symmetries
        \begin{equation*}
            \sigma_1\overline{\Phi_{j}(x;\overline{k})}\sigma_1=\Phi_{j}(x;k), \quad
            [\Phi_j]_1(x;k)=\sigma_1[\Phi_j]_2(x;-k),\quad
            \overline{\Phi_j(x;-\bar{k})}=\Phi_j(x;k).
        \end{equation*}
        \item[(e)]For $k\in (-c_j, c_j)$, we have
        \begin{equation*}
           [\Phi_j]_1 (x;k)= [\Phi_j]_2 (x;k).
        %\quad[\Phi_r]_1 (x;k)=- [\Phi_r]_2 (x;k),
        \end{equation*}
        %\item[(f)] As $k\rightarrow\pm c_j$, we have $\Phi_j(x; k)=\mathcal{O}((k\mp c_j)^{-\frac{1}{4}})$.
    \end{itemize}
\end{proposition}
\begin{remark}
    When $k$ lies on the branch cut $(-c_l, c_l)$, $[\Phi_l]_1(x; k)$ denotes the boundary value of $[\Phi_l]_1(x; \tilde{k})$ as $\tilde{k}$ approaches $k$ from $\mathbb{C}^+$, while $[\Phi_l]_2(x; k)$ denotes the boundary value as $\tilde{k}$ approaches $k$ from $\mathbb{C}^-$. The definition for $\Phi_r(x; k)$, $k\in(-c_r,c_r)$ 
    is analogous.
\end{remark}

Since the matrices $\Phi_{l}(x;k)$ and $\Phi_{r}(x;k)$ are both solutions to the $x$-part of the Lax pair \eqref{equ:lax pair} for $k\in\mathbb{R}\setminus\{\pm c_l,\pm c_r\}$ and $(x,t)\in\mathbb{R}\times[0,+\infty)$, they must be linearly dependent. 
Consequently, there exists a scattering matrix $S(k)$, independent of $x$, such that
\begin{equation}\label{equ:scattering matrix}
    S(k)=\Phi_{r}^{-1}(x;k)\Phi_{l}(x;k),\quad x\in\mathbb{R},\;k\in\mathbb{R}\setminus\{\pm c_l,\pm c_r\}.\end{equation}
Due to the symmetries of the $x$-part in \eqref{equ:lax pair}, the scattering matrix $S(k)$
admits the following structure
\begin{equation}\label{def:S(k)}
    S(k)=\begin{pmatrix}a(k) & b^*(k) \\ b(k) & a^*(k) \end{pmatrix},\quad k\in\mathbb{R}\setminus\{\pm c_l,\pm c_r\},
\end{equation}
where the scattering coefficients $a(k)$, $b(k)$ are given by
\begin{align}
    &a(k)=\det\left([\Phi_{l}]_1, [\Phi_{r}]_2\right), \label{equ:a,b det-1}\\
    &b(k)=\det\left([\Phi_{r}]_1, [\Phi_{l}]_1\right).\label{equ:a,b det-2}
\end{align}
To proceed, we define the reflection coefficient
\begin{equation}
    r(k):=\frac{b(k)}{a(k)},\quad k\in\mathbb{R}\setminus\{\pm c_l,\pm c_r\}.
\end{equation}

We now introduce the modified Jost functions $\phi_j$ for $j\in\{l,r\}$:
\begin{equation}\label{def:phij}
    \begin{aligned}
        &\phi_l(x;k)=\Phi_l(x;k)e^{iX_l(k)x\sigma_3}, && k\in(\overline{\mathbb{C}^{+}},\overline{\mathbb{C}^{-}})\setminus\{-c_{l}, c_{l}\},\\
        &\phi_r(x;k)=\Phi_r(x;k)e^{iX_r(k)x\sigma_3}, && k\in(\overline{\mathbb{C}^{-}},\overline{\mathbb{C}^{+}})\setminus\{-c_{r}, c_{r}\}.
    \end{aligned}
\end{equation} 

Some basic properties of the function $\phi_j$ is summarized in the following proposition.
\begin{proposition}\label{phi properties}
    For $j\in\{l,r\}$, the function $\phi_j$ has the following properties:
\begin{itemize}
        \item[\rm (a)]For each $x\in\mathbb{R}$, we have\\
        $[\phi_r]_1(x; k)$ is holomorphic for $k\in\mathbb{C}^{-}$ and has continuous extension to $\overline{\mathbb{C}^{-}}\backslash\{-c_{r}, c_{r}\}$,\\
        $[\phi_r]_2(x; k)$ is holomorphic for $k\in\mathbb{C}^{+}$ and has continuous extension to $\overline{\mathbb{C}^{+}}\backslash\{-c_{r}, c_{r}\}$,\\
        $[\phi_l]_1(x; k)$ is holomorphic for $k\in\mathbb{C}^{+}$ and has continuous extension to $\overline{\mathbb{C}^{+}}\backslash\{-c_{l}, c_{l}\}$,\\
        $[\phi_l]_2(x; k)$ is holomorphic for $k\in\mathbb{C}^{-}$ and has continuous extension to $\overline{\mathbb{C}^{-}}\backslash\{-c_{l}, c_{l}\}$.\\
        Moreover, for each $x\in\mathbb{R}$, we have $\phi_j(x,\cdot)\in\mathcal{C}^{10}(\mathbb{R}\setminus\{\pm c_j\})$.
        \item[(b)] $\det\phi_j(x;k)=1$, \quad$k\in\mathbb{R}\setminus[-c_j,c_j]$.
        \item[(c)] It holds that
        \begin{align*}
    &\phi_{l}(x;k)=\Delta_{l}(k)\left(I+o(1)\right), \quad x\rightarrow-\infty, \quad k\in\mathbb{R}\setminus\{-c_l,c_l\},\\
    &\phi_{r}(x;k)=\Delta_{r}(k)\left(I+o(1)\right), \quad x\rightarrow+\infty, \quad k\in\mathbb{R}\setminus\{-c_r,c_r\}.
\end{align*}
        \item[(d)] $\phi_j(x;k)$ admits the following symmetries
        \begin{equation*}
            \sigma_1\overline{\phi_{j}(x;\overline{k})}\sigma_1=\phi_{j}(x;k), \quad
            [\phi_j]_1(x;k)=\sigma_1[\phi_j]_2(x;-k),\quad
            \overline{\phi_j(x;-\bar{k})}=\phi_j(x;k).
        \end{equation*}
   \item[(e)] As $k\rightarrow\pm c_j$, we have $\phi_j(x; k)=\mathcal{O}((k\mp c_j)^{-\frac{1}{4}})$.
\end{itemize}
\end{proposition}
\begin{proof}
    The items (a)--(d) follow from the properties stated in 
    Proposition \ref{prop: properties of Phi(r,j)} directly. 

    We now prove (e). Taking $\phi_l(x;k)$ as an example, from the definition in \eqref{def:phij}, we obtain
    \begin{equation*}
        \phi_{l}(x;k)=\Delta_l(k)+\int_{-\infty}^{x}\Phi^b_l(x;k)(\Phi^b_l)^{-1}(y;k)\left[Q_0(y)-Q^b_l\right]\phi_{l}(y;k)e^{iX_l(k)(x-y)\sigma_3}\dif y,
    \end{equation*}
   and then $[\phi_l]_1$ can be represented as 
\begin{equation}\label{equ:formula for phil}
    [\phi_l]_1=\Lambda_l[\Delta_l]_1,\quad k\in\overline{\mathbb{C}^{+}}\backslash\{-c_{l}, c_{l}\},
\end{equation}
where
\begin{equation}\label{def:Lambda}
    \Lambda_l(x;k)=I+\int_{-\infty}^{x}e^{iX_l(k)(x-y)}\Phi^b_l(x;k)(\Phi^b_l)^{-1}(y;k)\left[Q_0(y)-Q^b_l\right]\Lambda_l(y;k)\dif y.
\end{equation}
Using similar manners stated in \cite[Lemma 3.4]{lenellsmkdv}, 
we can show that for every $x\in\mathbb{R}$, 
$\Lambda(x,k)$ has a continuous extension for $k\in\overline{\mathbb{C}^{+}}$, and 
$\Lambda(x,\cdot)\in\mathcal{C}^{10}(\mathbb{R}\setminus\{\pm c_l\})\cap\mathcal{C}(\mathbb{R})$. 
Furthermore, applying standard claims of Neumann series to $\Lambda_l(x;k)$, 
it follows that as $k\to\pm c_l$,
\begin{equation}\label{equ:expansion for lambda}
\Lambda_l(x;k)=\Lambda_l\left(x,\pm c_l\right)+\sum_{m=1}^{9} \Lambda^{(m)}_{l,\pm}(x)\left(k\mp c_l\right)^{m / 2}+o\left(\left(k\mp c_l\right)^{\frac{9}{2}}\right),
\end{equation}
where $\Lambda^{(m)}_{l,\pm}(x)\in\mathbb{C}$. 
Combining this with the fact that $\Delta_l(k)=\mathcal{O}((k\mp c_l)^{-1/4})$ as $k\to\pm c_l$, we conclude that $[\phi_l]_1=\mathcal{O}((k\mp c_l)^{-1/4})$. Furthermore, by the symmetry property stated in item (d), $[\phi_l]_2$ admits a similar expansion. The analysis for $\phi_r$ proceeds analogously by defining $\Lambda_r(x;k)$, which completes the proof of (e).
\end{proof}

With the aid of Proposition \ref{prop: properties of Phi(r,j)} and \ref{phi properties}, we derive the following properties for the spectral functions $a(k)$, $b(k)$ and the reflection coefficient $r(k)$.
\begin{proposition}\label{prop:a,b,r}
    Spectral functions $a(k)$, $b(k)$ and reflection coefficient $r(k)$ satisfy the following properties:
\begin{itemize}
    \item[\rm (a)] Spectral function $a(k)$ is holomorphic for $k\in\mathbb{C}^{+}$, and it could be continuously
    extended up to the boundary $\mathbb{R}\setminus\{\pm c_l, \pm c_r\}$. As $k\to \pm c_j,$ for $j\in\{l,r\}$,
    we have $a(k)=\mathcal{O}((k\mp c_j)^{-1/4})$. For $k\in\mathbb{C}^+$, as $k\to\infty$,
    we have $a(k)=1+\mathcal{O}(k^{-1})$. The spectral function $b(k)$ is defined continuously for $k\in\mathbb{R}\setminus\{\pm c_l, \pm c_r\}$.
    In particular, $a$, $b$ belong to $\mathcal{C}^{10}(\mathbb{R}\setminus\{\pm c_l, \pm c_r\})$.
    \item[\rm (b)] $a(k)$ has no zeros on the complex plane $\mathbb{C}$.
    \item[\rm (c)] In their domains of definition, we have
        \begin{align}\label{symmetry a,b,r}
            a(k)=\overline{a(-\bar{k})}, \quad b(k)=\overline{b(-\bar{k})}, \quad  r(k)=\overline{r(-\bar{k})}.
        \end{align}
    \item[\rm (d)] The functions $a(k)$, $b(k)$, and $r(k)$ obey the following jump relations:
    \begin{align*}
        &a_{+}(k)=a^*_{-}(k), \; b_{+}(k)=-b^*_{-}(k),\; r_{+}(k)=-r^*_{-}(k), && k\in (-c_{r}, c_{r}), \\
        &a_{+}(k)=b_{-}^{*}(k), \; b_{+}(k)=a^*_{-}(k), \; r_{+}(k)=r_{-}^{*}(k)^{-1},  && k\in (-c_l, -c_{r})\cup(c_{r}, c_l), \\
        &a_{+}(k)=a_{-}(k), \; b_{+}(k)=b_{-}(k), \; r_{+}(k)=r_{-}(k),&& k\in (-\infty,-c_l)\cup(c_l,+\infty). 
    \end{align*}    
    It can be readily seen that $\vert r(k) \vert=1$ for $k\in(-c_l, -c_{r})\cup(c_{r}, c_l)$
    and $|r(k)|<1$ for $k$ belongs to the other intervals on $\mathbb{R}$.
    \item[\rm (e)] The reflection coefficient $r(k)\in\mathcal{C}^{10}(\mathbb{R}\setminus\{\pm c_l,\pm c_r\})$ has the following asymptotics near the branch cut points:
    \begin{equation}\label{expansionrk}
        r(k)=\begin{cases}
            \sum_{m=0}^{9} b_{l, m}\left(k-c_l\right)^{m / 2}+o(\left(k-c_l\right)^{\frac{9}{2}}), &  k \searrow c_l, \\
            \sum_{m=0}^{9} i^m b_{l,m}\left(c_l-k\right)^{m / 2}+o(\left(c_l-k\right)^{\frac{9}{2}}), & k \nearrow c_l, \\ 
            \sum_{m=0}^{9} i^m b_{r,m}\left(k-c_r\right)^{m / 2}+o(\left(k-c_r\right)^{\frac{9}{2}}), &  k \searrow c_r, \\ 
            \sum_{m=0}^{9}(-1)^m b_{r,m}\left(c_r-k\right)^{m / 2}+o(\left(c_r-k\right)^{\frac{9}{2}}), &  k \nearrow c_r,\\
            \sum_{m=0}^{9} (-i)^m\overline{b_{ l,m}}\left(k+c_l\right)^{m / 2}+o(\left(k+c_l\right)^{\frac{9}{2}}), & k \searrow -c_l, \\
            \sum_{m=0}^{9}  \overline{b_{l,m}}\left(-c_l-k\right)^{m / 2}+o(\left(-c_l-k\right)^{\frac{9}{2}}) ,& k \nearrow -c_l, \\ 
            \sum_{m=0}^{9} (-1)^m \overline{b_{r,m}}\left(k+c_r\right)^{m / 2}+o(\left(k+c_l\right)^{\frac{9}{2}}), & k \searrow -c_r, \\ 
            \sum_{m=0}^{9} (-i)^m \overline{b_{r,m}}\left(-c_r-k\right)^{m / 2}+o(\left(-c_r-k\right)^{\frac{9}{2}}), &  k \nearrow -c_r,
            \end{cases}
    \end{equation}
    where the coefficients $b_{j,m} \in \mathbb{C}$ satisfies
\begin{equation}\label{expansionrk coefficient}
    \left|b_{j, 0}\right|=1, \quad b_{j,1} \neq 0, \quad \sum_{m=0}^n i^{n-m}(-i)^m b_{j,n-m}\overline{b_{j,m}}=0, \quad j\in\{l,r\}, \;n=1, \dots, 9.
\end{equation}
Moreover, $r(k)$ has the following asymptotics as $k\to\pm\infty$:
\begin{equation}\label{dacay of rk}
    \partial_k^mr(k)=\mathcal{O}(k^{-5}),\quad m=0,\dots,10.
\end{equation}
\end{itemize}
\end{proposition}
\begin{proof}
\hfill
\begin{itemize}
\item[\rm (a)] The analytic properties of $a(k)$ and $b(k)$ follow directly from those of $\Phi_j$ stated in the item (a) of Proposition \ref{prop: properties of Phi(r,j)}
via \eqref{equ:a,b det-1}--\eqref{equ:a,b det-2}. The expansions near the branch points $\pm c_l$ and $\pm c_r$ are consequences of item (e) of Proposition \ref{phi properties}. 
The large-$k$ asymptotics of $a(k)$ follow from the corresponding behavior of the Jost functions.

\item[\rm (b)] By the standard claim, we see that $\det S(k)=1$ for $k\in\mathbb{R}\setminus[-c_l,c_l]$. Thus, we 
can obtain that $|a(k)|\geqslant 1$ for $k\in\mathbb{R}\setminus[-c_l,c_l]$. 
Noticing the symmetry $a(k)=\overline{a(-\bar{k})}$ in (c), it remains to show that
$a(k)\neq 0$ for $k\in\mathbb{C}^{+}\cup(-c_l,c_l)$. We deal with the cases $k\in\mathbb{C}^{+}$
and $k\in(-c_l,c_l)$ separately.

At first, we assume that there exists a zero $k_0\in\mathbb{C}^{+}$ such that $a(k_0)=0$. 
Define a vector complex-valued function $f:=(f_1,f_2)$ which belongs to the Hilbert space $L^{2}(\mathbb{R})$,
that is, both $f_1$ and $f_2$ belong to $L^{2}(\mathbb{R})$. In this space, we equip the inner product by
\begin{equation*}
    \langle f, g\rangle=\int_{\mathbb{R}}\left(f_1(x)\overline{g_1(x)}+f_2(x)\overline{g_2(x)}\right)\dif x.
\end{equation*}
It's noticed that mKdV equation admits the self-adjoint Lax 
operator $L:=i\sigma_3\partial_x-i\sigma_3Q_{0}$, which satisfies
\begin{equation*}
    \langle Lf, g\rangle=\langle f, Lg\rangle,  
\end{equation*}
for any vector complex-valued functions $f,g\in H^{1}(\mathbb{R})\subset L^1(\mathbb{R})$.
Now, we define $p:\mathbb{R}\to\mathbb{C}^{2}$ by
\begin{align}
    p(x):=[\Phi_r]_2(x;k_0)=[\phi_{r}(x;k_0)]_{2}e^{iX_r(k_0)x}.
\end{align}
Since $\im X_{r}(k_0)>0$ for $k_0\in\mathbb{C}^{+}$, we employ the item (c)
of Proposition \ref{phi properties} to yield exponential decay of $p$ as $x\to+\infty$.
Meanwhile, from the determinant representation of $a(k)$ in  \eqref{equ:a,b det-1},
the assumption $a(k_0)=0$ implies that $[\Phi_l]_1(x;k_0)$ and $[\Phi_r]_2(x;k_0)$ are linearly dependent. 
Thus, there exists a nonzero constant $\gamma\in\mathbb{C}$ such that
\begin{align}
    p(x)=\gamma[\Phi_l]_1(x;k_0)=\gamma[\phi_{l}(x;k_0)]_{1}e^{iX_l(k_0)x}.
\end{align}
Also noticing $\im X_{l}(k_0)>0$ for $k_0\in\mathbb{C}^{+}$, and again using the item (c)
of Proposition \ref{phi properties}, we see that $p(x)$ decays exponentially as $x\to-\infty$.
Therefore, we conclude that $p\in H^{1}(\mathbb{R})$.
Furthermore, from the Lax pair \eqref{equ:lax pair}, we have at time $t=0$:
\begin{equation*}                   
    Lp(x)=k_0 p(x).
\end{equation*}
Taking the inner product with $p$ on both sides of the above equation, we obtain
\begin{equation*}                
    k_0\langle p, p\rangle=\langle Lp, p\rangle=\langle p, Lp\rangle=\overline{k_0}\langle p, p\rangle,
\end{equation*}
that is, $(k_0-\overline{k_0})\langle p, p\rangle=0$. Since $\im k_0>0$, we have $\langle p, p\rangle>0$,
which leads to a contradiction. Thus, $a(k)$ has no zeros in $\mathbb{C}^{+}$.

It only remains to show that $a(k)\neq 0$ for $k\in(-c_l,c_l)$.
Utilizing the part (e) of Proposition \ref{prop: properties of Phi(r,j)}
as well as the definition of $S(k)$ in \eqref{equ:scattering matrix}, we have
for $k\in(-c_l,c_l)$
\begin{align*}
    a(k)=\det\left([\Phi_{l}]_1, [\Phi_{r}]_2\right)=\det\left([\Phi_{l}]_2, [\Phi_{r}]_2\right)=\overline{b(k)}. 
\end{align*}
Suppose that there exists a zero $k_0\in(-c_l,c_l)$ such that $a(k_0)=0$. 
Then, we have $b(k_0)=0$. 
From the scattering relation $\Phi_l(x;k)=\Phi_r(x;k)S(k)$, we have
\begin{equation*}
    [\Phi_l]_1(x;k)=a(k)[\Phi_r]_1(x;k)+b(k)[\Phi_r]_2(x;k).
\end{equation*}
Substituting $k=k_0$, it follows that $[\Phi_l]_1(x;k_0)\equiv 0$ for all $x\in\mathbb{R}$.
Then, 
\begin{align*}
    \phi_{l}(x;k_0)=[\Phi_{l}]_{1}(x;k_0)e^{iX_{l+}(k_0)x\sigma_3}
\end{align*}
also vanishes everywhere for $x\in\mathbb{R}$. However, this contradicts the 
item (c) in Proposition \ref{phi properties} since $\Delta_l(k)\neq 0$ for $k\in(-c_l,c_l)$.
Thus, $a(k)$ has no zeros in $(-c_l,c_l)$.

Summarizing the above analysis, we conclude that $a(k)$ has no zeros on $\mathbb{C}$.

\item[\rm (c)] It follows from the symmetries of Jost functions given in item (d) of Proposition \ref{prop: properties of Phi(r,j)}.

\item[\rm (d)] It follows from item (e) of Proposition \ref{prop: properties of Phi(r,j)}
and the scattering matrix \eqref{def:S(k)} that
\begin{align*}
    S_{+}(k)=\begin{pmatrix}
        0 & 1 \\
        -1 & 0
    \end{pmatrix}S_{-}(k)\begin{pmatrix}
        0 & -1 \\
        1 & 0
    \end{pmatrix}, \quad k\in(-c_{r},c_{r})
\end{align*}
as well as
\begin{align*}
    S_{+}(k)=S_{-}(k)\begin{pmatrix}
        0 & -1 \\
        1 & 0
    \end{pmatrix}, \quad k\in (-c_l,-c_{r})\cup(c_{r},c_l).
\end{align*}
Therefore, the jump relations of $a$, $b$ and $r$ are obtained via the entries of the above relations.

\item[\rm (e)] 
The proof of this item is similar to that of \cite[Theorem 2.1(a)]{LenellsDNLS}, so we only outline the derivation here. Using \eqref{equ:a,b det-1}--\eqref{equ:a,b det-2} and the definition of $\phi_j$ in \eqref{def:phij}, we can rewrite $r(k)$ as 
\begin{equation*}
    r(k)=\frac{\phi_{l,11}(0;k)\phi_{r,21}(0;k)-\phi_{l,21}(0;k)\phi_{r,11}(0;k)}{\phi_{l,11}(0;k)\phi_{r,22}(0;k)-\phi_{l,21}(0;k)\phi_{r,12}(0;k)},\quad k\in\mathbb{R}\setminus\{\pm c_l,\pm c_r\}.
\end{equation*}
Since each entry $\phi_{j,mn}$ for $j \in \{l, r\}$ and $m,n \in \{1,2\}$ can be expressed in terms of the entries of $\Lambda_j$ in \eqref{equ:formula for phil}, the expansion \eqref{equ:expansion for lambda} directly implies the asymptotics of $r(k)$ given in \eqref{expansionrk}. Additionally, the relations in \eqref{expansionrk coefficient} follow from the fact that $|r(k)|=1$ for $k\in(-c_l, -c_{r})\cup(c_{r}, c_l)$. 
Finally, the decay estimate \eqref{dacay of rk} is obtained from the 
large-$k$ asymptotics of $\phi_j(x;k)$, $j\in\{l,r\}$, as well as their regularity (cf. \cite[Proposition 3.6 and Corollary 3.8]{LenellsDNLS}).
\end{itemize}
\end{proof}

\subsection{RH characterization of the mKdV equation}
In order to construct a basic RH problem for the long-time asymptotic analysis, it is required to operate
the time evolution of the scattering data. Assuming that the solution $q(x,t)$ of Cauchy problem
\eqref{equ:mkdv}--\eqref{Initial data} exists for $t\geqslant 0$, it is followed that
\begin{align}
    &\Phi_{l}(x,t;k)=\Phi_{l}^{b}(x,t;k)\left(I+o(1)\right), \quad x\rightarrow-\infty, \quad k\in\mathbb{R},\label{equ:Jostsolwithtime-1}\\
    &\Phi_{r}(x,t;k)=\Phi_{r}^{b}(x,t;k)\left(I+o(1)\right), \quad x\rightarrow+\infty, \quad k\in\mathbb{R},\label{equ:Jostsolwithtime-2}
\end{align}
where $\Phi_{j}^{b}(x,t;k)$ for $j\in\{l,r\}$  is defined in \eqref{equ:Phi_j^p}.

Since $\Phi_l$ and $\Phi_r$ are defined as simultaneous solutions of the Lax pair \eqref{equ:lax pair}, they must be linearly dependent. Consequently, from \eqref{equ:Jostsolwithtime-1} and \eqref{equ:Jostsolwithtime-2}, we obtain the following relation:
\begin{equation}
    \Phi_{l}(x,t;k)=\Phi_{r}(x,t;k)S(t;k), \quad k\in\mathbb{R}\setminus\{\pm c_{l}, \pm c_r\}.
\end{equation}
On account of the Lax pair \eqref{equ:lax pair}, it is obtained that
\begin{align*}
 \frac{\partial S(t;k)}{\partial t}\equiv 0,
\end{align*}
which implies that $\partial_t a(t;k)=0$ and $\partial_t b(t;k)=0$. This claim shows that the scattering data $a$ and $b$ are constant with respect to time.

Let us define the following matrix-valued function $M^{(0)}(k):=M(x,0;k)$ by
\begin{equation}\label{def:M0(k)}
   M^{(0)}(k):=M(x,0;k)=\left\{
        \begin{aligned}
        &\left(\frac{[\Phi_l(x,0;k)]_1}{a(k)}, [\Phi_r(x,0;k)]_2\right)e^{ikx\sigma_3}, \quad k\in \mathbb{C}^{+}, \\
        &\left([\Phi_r(x,0;k)]_1, \frac{[\Phi_l(x,0;k)]_2}{a^{*}(k)}\right)e^{ikx\sigma_3}, \quad k\in \mathbb{C}^{-},
        \end{aligned}
        \right.
\end{equation}
The matrix-valued function $M^{(0)}$ satisfies the RH problem below.
\begin{RHP}\label{RHP:basic 0RHP}
    \hfill
    \begin{itemize}
        \item $M^{(0)}(k)$ is holomorphic for $k\in\mathbb{C}\backslash\mathbb{R}$.
        \item $M^{(0)}(k)$ has continuous boundary values $M^{(0)}_{\pm}(k)$ on $\mathbb{R}$ with the jump condition
        \begin{align*}
            M^{(0)}_{+}(k)=M^{(0)}_{-}(k)V_0(k),
        \end{align*}
        where
        \begin{equation}\label{equ:jump V^{(0)}(k)}
            V^{(0)}(k)=
                \begin{cases}
                \begin{pmatrix}1-rr^* & - r^*e^{-2ikx}\\ re^{2ikx} & 1\end{pmatrix}, &k\in(-\infty,-c_l)\cup(c_l,+\infty),  \\
                \begin{pmatrix}0 & - r^*_{-}e^{-2ikx}\\ r_{+}e^{2ikx} & 1\end{pmatrix}, &k\in (-c_l,-c_{r})\cup(c_{r},c_l),\\
                \begin{pmatrix}0 & -e^{-2ikx}\\ e^{2ikx} & 0\end{pmatrix},  &k\in (-c_{r}, c_{r}).
                \end{cases}       
        \end{equation}
         
        \item As $k\rightarrow \infty$ in $\mathbb{C}\setminus\mathbb{R}$, we have $M(k)=I+\mathcal{O}(k^{-1})$.
    \end{itemize}
\end{RHP}     

Due to the time dependence, we are motivated to construct a piecewise analytic matrix-valued function as follows:
\begin{equation}\label{eq: def of M(x,t;k)}
    M(x,t;k):=\left\{
        \begin{aligned}
        &\left(\frac{[\Phi_l(x,t;k)]_1}{a(k)}, [\Phi_r(x,t;k)]_2\right)e^{it\theta(\xi;k)\sigma_3}, \quad k\in \mathbb{C}^{+}, \\
        &\left([\Phi_r(x,t;k)]_1, \frac{[\Phi_l(x,t;k)]_2}{a^{*}(k)}\right)e^{it\theta(\xi;k)\sigma_3}, \quad k\in \mathbb{C}^{-},
        \end{aligned}
        \right.
\end{equation}
where $\theta(\xi;k)=4k^3+12k\xi$ and $\xi=x/(12t)$. The matrix-valued function $M(x,t;\cdot):\mathbb{C}\backslash\mathbb{R}\to GL(2,\mathbb{C})$ satisfies the RH problem below.
\begin{RHP}\label{RHP:basic RHP}
    \hfill
    \begin{itemize}
        \item $M(k)$ is holomorphic for $k\in\mathbb{C}\backslash\mathbb{R}$.
        \item $M(k)$ has continuous boundary values $M_{\pm}(k)$ on $\mathbb{R}$ with the jump condition
        \begin{align*}
            M_{+}(k)=M_{-}(k)V(k),
        \end{align*}
        where
        \begin{equation}\label{equ:jump V(k)}
            V(k)=
                \begin{cases}
                \begin{pmatrix}1-rr^* & - r^*e^{-2it\theta}\\ re^{2it\theta} & 1\end{pmatrix}, &k\in(-\infty,-c_l)\cup(c_l,+\infty),  \\
                \begin{pmatrix}0 & - r^*_{-}e^{-2it\theta}\\ r_{+}e^{2it\theta} & 1\end{pmatrix}, &k\in (-c_l,-c_{r})\cup(c_{r},c_l),\\
                \begin{pmatrix}0 & -e^{-2it\theta}\\ e^{2it\theta} & 0\end{pmatrix},  &k\in (-c_{r}, c_{r}).
                \end{cases}       
        \end{equation}
         
        \item As $k\rightarrow \infty$ in $\mathbb{C}\setminus\mathbb{R}$, we have $M(k)=I+\mathcal{O}(k^{-1})$.
        \item As $k\rightarrow\pm c_{l}$, we have $M(k)=\mathcal{O}(1)$.
    
        \item As $k\rightarrow\pm c_{r}$, we have
        \begin{align*}
            &M(k)=\mathcal{O}\begin{pmatrix}(k\mp c_{r})^{\frac{1}{4}} & (k\mp c_{r})^{-\frac{1}{4}} \\ (k\mp c_{r})^{\frac{1}{4}} & (k\mp c_{r})^{-\frac{1}{4}}\end{pmatrix}, \quad k\in\mathbb{C}^{+}, \\
            &M(k)=\mathcal{O}\begin{pmatrix}(k\mp c_{r})^{-\frac{1}{4}} & (k\mp c_{r})^{\frac{1}{4}} \\ (k\mp c_{r})^{-\frac{1}{4}} & (k\mp c_{r})^{\frac{1}{4}}\end{pmatrix}, \quad k\in\mathbb{C}^{-}.
        \end{align*}
    \end{itemize}
\end{RHP}                                                                      It follows from item (d) of Proposition \ref{prop: properties of Phi(r,j)} and
item (c) of Proposition \ref{prop:a,b,r} that the solution $M(k)$ to the RH problem \ref{RHP:basic RHP}
automatically satisfies the following symmetries:
\begin{align*}
M(k)=\sigma_1M^{*}(k)\sigma_1=\overline{M(-\bar{k})}=\sigma_1M(-k)\sigma_1.
\end{align*}

\subsubsection{Proof of Theorem \ref{thm: global solution existence}}
Since $\det V(k)=1$, it follows from standard proof that the solution of the 
RH problem \ref{RHP:basic RHP} for $M$ is unique if it exists.
Furthermore, the existence of such a solution to RH problem \ref{RHP:basic RHP}
is ensured by H\"older continuity of the jump matrix $V(k)$ and the following vanishing lemma (cf. \cite{ZhouVanishLemma}).
\begin{proposition}[Vanishing lemma]
    Suppose that $M(x,t;k)$ is a solution to the RH problem \textup{\ref{RHP:basic RHP}} with 
    the homogeneous large-$k$ condition for $(x,t)\in\mathbb{R}\times[0,+\infty)$, namely 
    \begin{equation}\label{new normalization}
        M(k)=\mathcal{O}(k^{-1}),\quad k\to\infty. 
    \end{equation}
    Then $M(k)$ vanishes everywhere.
\end{proposition}
\begin{proof}
    Define $H(k)=M(k)M(\bar{k})^{\mathrm{H}}$. Using the decay of $r(k)$ as $k\to\infty$ in item (e) of Proposition \ref{prop:a,b,r} and the homogeneous condition 
    \eqref{new normalization}, we obtain
        \begin{align*}
            &\int_{\mathbb{R}}H_+(k)\dif k=\int_{\mathbb{R}}M_-(k)V(k)M_-(k)^{\mathrm{H}}\dif k=0,\\
            &\int_{\mathbb{R}}H_-(k)\dif k=\int_{\mathbb{R}}M_-(k)V(k)^{\mathrm{H}}M_-(k)^{\mathrm{H}}\dif k=0, 
        \end{align*}
which implies that
\begin{equation*}
    \int_{\mathbb{R}}M_-(k)Y(k)M_-(k)^{\mathrm{H}}\;\mathrm{d}k=0,
\end{equation*}
where $Y(k)=\frac{1}{2}(V(k)+V(k)^{\mathrm{H}})$. 
Since $|r(k)|<1$ on $\mathbb{R}\setminus[-c_l,c_l]$, we see that $V(k)$ is positive definite on $\mathbb{R}\setminus[-c_l, c_l]$. 
It then follows that $M_-(k)=0$ on $\mathbb{R}\setminus[-c_l,c_l]$. 
Applying the Privalov's uniqueness theorem on $M(k)$ for $k\in\mathbb{C}^{-}$, 
we indeed have $M_-(k)=0$ on $\mathbb{R}$. which implies that $M_+(k)=0$ via $M_+(k)=M_-(k)V(k)$. The decay of $M(k)$ implies that $M(k)$ 
vanishes everywhere.
\end{proof}
Furthermore, under Assumption \ref{assumption on q_0} for the initial data $q_0$, 
it follows from the proof of \cite[Theorem 7]{lenellsmkdvfourier} that
a classical solution $q(x,t)$ of the mKdV equation \eqref{equ:mkdv} exists and 
can be reconstructed via the following limit for all $(x,t)\in\mathbb{R}\times[0,+\infty)$:
\begin{equation}\label{equ:recovering formula}
    q(x,t)=2i\lim_{k\rightarrow\infty} \left(kM(x,t;k)\right)_{12}.
\end{equation}
The crucial step in the proof of \cite[Theorem 7]{lenellsmkdvfourier} is to 
show that the solution $M$ of the RH problem \ref{RHP:basic RHP} satisfies the Lax pair equations:
\begin{align*}
    \frac{\partial M}{\partial x}+ik[\sigma_3,M]=QM,\quad 
    \frac{\partial M}{\partial t}+4ik^3[\sigma_3,M]=VM,
\end{align*}
for all $(x,t)\in\mathbb{R}\times[0,+\infty)$. 
To address this, it suffices to verify that 
$r\in L^{2}(\mathbb{R})\cap L^{\infty}(\mathbb{R})$ and $r(k)=\mathcal{O}(k^{-5})$.
Both requirements are satisfied due to items (d) and (e) of Proposition \ref{prop:a,b,r}.
Furthermore, by the definition of \eqref{def:M0(k)} and the uniqueness, 
the solution $M^{(0)}(k)$ of the RH problem \ref{RHP:basic 0RHP} satisfies the RH problem \ref{RHP:basic RHP} 
at $t=0$. Consequently, it follows from the reconstruction formula \eqref{equ:recovering formula} that 
$q(x,0)$ is indeed equal to the initial data $q_0(x)$.
It remains to show that $q(x,t)$ satisfies the boundary stated in
item (d) of Definition \ref{def: global solution}. This will be accomplished by
finishing the proof of Theorem \ref{thm:mainthm}.

These facts lead to the Theorem \ref{thm: global solution existence}.

\subsubsection{Factorizations of jump matrix}

A crucial step of performing the steepest descent analysis is to open lenses, which is aided by two well-known
factorizations of the jump matrix $V(k)$ defined in \eqref{equ:jump V(k)}. Denote $r_2(k)={r^*(k)}/{(1-r(k)r^*(k))}$ for 
$k\in\mathbb{R}\setminus[-c_l,c_l]$, then the factorizations
utilized throughout the context can be listed as follows.

For $k\in(-\infty,-c_l)\cup(c_l,+\infty)$,
\begin{align}
    \begin{pmatrix} 1-rr^* & -r^*e^{-2it\theta} \\ re^{2it\theta} & 1 \end{pmatrix}&=\begin{pmatrix} 1 & -r^*e^{-2it\theta} \\ 0 & 1\end{pmatrix}\begin{pmatrix} 1 & 0 \\ re^{2it\theta} & 1\end{pmatrix}\label{factor11}\\
        &=\begin{pmatrix} 1 & 0 \\r_2^*e^{2it\theta} & 1 \end{pmatrix}\left(1-rr^*\right)^{\sigma_3}\begin{pmatrix} 1 & -r_2e^{-2it\theta}\\ 0 & 1\end{pmatrix}.
\end{align}

For $k\in (-c_l,-c_{r})\cup(c_{r},c_l)$, where $r_+=1/r_-^*$,
\begin{align}
    \begin{pmatrix} 0 & -r_-^*e^{-2it\theta} \\ r_+e^{2it\theta} & 1 \end{pmatrix}&=\begin{pmatrix} 1 & -r_-^*e^{-2it\theta} \\ 0 & 1\end{pmatrix}\begin{pmatrix} 1 & 0 \\ r_+e^{2it\theta} & 1\end{pmatrix}\label{factor21}\\
        &=\begin{pmatrix} 1 & 0 \\ r_{2,-}^*e^{2it\theta}& 1 \end{pmatrix}\begin{pmatrix} 0 & -r_-^*e^{-2it\theta} \\ r_{+}e^{2it\theta} & 0\end{pmatrix}\begin{pmatrix} 1 & -r_{2,+}e^{-2it\theta}\\ 0 & 1\end{pmatrix}.
\end{align}

For $k\in (-c_{r}, c_{r})$, where $r_{+}=-r_-^*$,
\begin{align}
    \begin{pmatrix}0 & -e^{-2it\theta}\\ e^{2it\theta} & 0\end{pmatrix}&=\begin{pmatrix} 1 & -r_-^*e^{-2it\theta} \\ 0 & 1\end{pmatrix}\begin{pmatrix}0 & -e^{-2it\theta}\\ e^{2it\theta} & 0\end{pmatrix}\begin{pmatrix} 1 & 0 \\ re^{2it\theta} & 1\end{pmatrix}\\
    &=\begin{pmatrix} 1 & 0 \\ r_{2,-}^*e^{2it\theta} & 1 \end{pmatrix}\begin{pmatrix} 0 & -e^{-2it\theta} \\ e^{2it\theta} & 0\end{pmatrix}\begin{pmatrix} 1 & -r_{2,+}e^{-2it\theta}\\ 0 & 1\end{pmatrix}.
\end{align}

\section{Asymptotic analysis of the RH problem for $M$ in $\mathcal{R}_{\textup{\uppercase\expandafter{\romannumeral1}}}$}\label{sec:asymptotic analysis in RI}
This section is devoted to the large-time asymptotic analysis 
of the RH problem \ref{RHP:basic RHP} in the region $\mathcal{R}_{\textup{\uppercase\expandafter{\romannumeral1}}}$. We refer to Table \ref{tab:notation-summary} for a compact guide to the principal spectral functions and contour notations used throughout the deformations.

\subsection{First transformation: $M\rightarrow M^{(1)}$}
Let us start with the construction of the $g$-function, which is crucial to deform
the RH problem \ref{RHP:basic RHP} into a form suitable for asymptotic analysis.
\subsubsection{The $g$-function}
For $\xi\in\mathcal{R}_{\textup{\uppercase\expandafter{\romannumeral1}}}$, we introduce
\begin{equation}\label{equ:g function RI}
g_{\textup{\uppercase\expandafter{\romannumeral1}}}(\xi;k):=(4k^2+12\xi+2c_{l}^2)X_{l}(k)
\end{equation}
with $X_l(k)=\sqrt{k^2-c_l^2}$ given in \eqref{def:X_j(k)}. 
The branch of the square root is chosen such that $X_l(k)=k+\mathcal{O}(k^{-1})$ as $k\rightarrow\infty$.
%It's readily verified that $g_{\textup{\uppercase\expandafter{\romannumeral1}}}$ defined in \eqref{equ:g function RI} satisfies the following properties:
\begin{proposition}The  function $g_{\textup{\uppercase\expandafter{\romannumeral1}}}$ defined in \eqref{equ:g function RI} satisfies the following properties:
    \begin{itemize}
        \item $g_{\textup{\uppercase\expandafter{\romannumeral1}}}(\xi;k)$ is holomorphic for $k\in\mathbb{C}\backslash[-c_l,c_l]$.
        \item As $k\rightarrow\infty$ in $\mathbb{C}\setminus[-c_l,c_l]$,
        we have $g_{\textup{\uppercase\expandafter{\romannumeral1}}}(\xi;k)=\theta(\xi;k)+\mathcal{O}(k^{-1})$.
        \item For $k\in(-c_l,c_l)$, $g_{\textup{\uppercase\expandafter{\romannumeral1}},+}(\xi;k)+g_{\textup{\uppercase\expandafter{\romannumeral1}}, -}(\xi;k)=0$.
    \end{itemize}
\end{proposition}
It's readily seen that the $k$-derivative of $g(k)$ is given by
\begin{equation}\label{g_1'}   g_{\textup{\uppercase\expandafter{\romannumeral1}}}'(k)=\frac{12k\left(k-\eta_l(\xi)\right)\left(k+\eta_l(\xi)\right)}{X_l(k)}, \quad \eta_l(\xi)=\sqrt{-\xi+\frac{c_l^2}{2}}\in(c_l,+\infty).
\end{equation}
The signature table for $\im g_{\textup{\uppercase\expandafter{\romannumeral1}}}$ is illustrated in Figure \ref{fig:signs img RI}, where ``$+$'' represents
that $\im g_{\textup{\uppercase\expandafter{\romannumeral1}}}>0$, and ``$-$'' represents that $\im g_{\textup{\uppercase\expandafter{\romannumeral1}}}<0$.
\begin{figure}[H]
\begin{center}
    \tikzset{every picture/.style={line width=0.75pt}} %set default line width to 0.75pt
    \begin{tikzpicture}[x=0.75pt,y=0.75pt,yscale=-1,xscale=1]
    %uncomment if require: \path (0,300); %set diagram left start at 0, and has height of 300
    %Curve Lines [id:da11389389910090686]
    \draw    (413,61) .. controls (387,61) and (386,242) .. (419,242) ;
    %Straight Lines [id:da8850126286527125]
    \draw    (159,152) -- (475,152) ;
    %Curve Lines [id:da9567273527008515]
    \draw    (203,61) .. controls (228,62) and (228,242) .. (200,243) ;
    % Text Node
    \draw (203,155.4) node [anchor=north west][inner sep=0.75pt]  [font=\tiny]  {$-\eta_l$};
       \fill (221,152) circle (1.2pt);
    % Text Node
    \draw (397,154.4) node [anchor=north west][inner sep=0.75pt]  [font=\tiny]  {$\eta_l$};
        \fill (394,152) circle (1.2pt);
    % Text Node
    \draw (232,156.4) node [anchor=north west][inner sep=0.75pt]  [font=\tiny]  {$-c_{l}$};
    % Text Node
    \draw (368,156.4) node [anchor=north west][inner sep=0.75pt]  [font=\tiny]  {$c_{l}$};
    % Text Node
    \draw (255.18,157.13) node [anchor=north west][inner sep=0.75pt]  [font=\tiny,rotate=-1.56]  {$-c_{r}$};
    % Text Node
    \draw (346,156.4) node [anchor=north west][inner sep=0.75pt]  [font=\tiny]  {$c_{r}$};
    % Text Node
    \draw (441,102.4) node [anchor=north west][inner sep=0.75pt]    {$+$};
    % Text Node
    \draw (447,183.4) node [anchor=north west][inner sep=0.75pt]    {$-$};
    % Text Node
    \draw (294,90.4) node [anchor=north west][inner sep=0.75pt]    {$-$};
    % Text Node
    \draw (299,202.4) node [anchor=north west][inner sep=0.75pt]    {$+$};
    % Text Node
    \draw (170,185.4) node [anchor=north west][inner sep=0.75pt]    {$-$};
    % Text Node
    \draw (173,103.4) node [anchor=north west][inner sep=0.75pt]    {$+$};
    % Text Node

    \fill (237,152) circle (1.2pt);
    % Text Node
   % \draw (261,147.4) node [anchor=north west][inner sep=0.75pt]  [font=\tiny]  {$($};
   \fill (261,152) circle (1.2pt);
    % Text Node
    %\draw (369,146.4) node [anchor=north west][inner sep=0.75pt]  [font=\tiny]  {$)$};
    \fill (369,152) circle (1.2pt);
    % Text Node
    %\draw (347,147.4) node [anchor=north west][inner sep=0.75pt]  [font=\tiny]  {$)$};
    \fill (347,152) circle (1.2pt);
    \end{tikzpicture}
\caption{Signature table of $\im g_{\textup{\uppercase\expandafter{\romannumeral1}}}(\xi;k)$ for $\xi\in\mathcal{R}_{\textup{\uppercase\expandafter{\romannumeral1}}}$.}\label{fig:signs img RI}
\end{center}
\end{figure}
\subsubsection{RH problem for $M^{(1)}$}
By the function $g_{\textup{\uppercase\expandafter{\romannumeral1}}}$ defined in \eqref{equ:g function RI}, we introduce a new matrix-valued function $M^{(1)}$ by
\begin{equation}\label{def:M1 RI}
    M^{(1)}(x,t;k):=M(x,t;k)e^{it\left(g_{\textup{\uppercase\expandafter{\romannumeral1}}}(\xi;k)-\theta(\xi;k)\right)\sigma_3}.
\end{equation}
Then the RH problem for $M^{(1)}$ reads as follows:
\begin{RHP}\label{RHP:M1 RI}
\hfill
\begin{itemize}
    \item $M^{(1)}(k)$ is holomorphic for $k\in\mathbb{C}\backslash\mathbb{R}$.
    \item $M^{(1)}(k)$ has continuous boundary values $M_{\pm}^{(1)}(k)$ on $\mathbb{R}$ with the jump condition
    \begin{equation*}
        M^{(1)}_{+}(k)=M^{(1)}_{-}(k)V^{(1)}(k), \quad k\in\mathbb{R},
    \end{equation*}
    where
    \begin{equation}\label{equ:jump V1 RI}
        V^{(1)}(k)=
            \begin{cases}
            \begin{pmatrix}1-rr^* & -r^*e^{-2itg_{\textup{\uppercase\expandafter{\romannumeral1}}}}\\ re^{2itg_{\textup{\uppercase\expandafter{\romannumeral1}}}} & 1\end{pmatrix}, &k\in(-\infty,-c_l)\cup(c_l,+\infty),  \\
            \begin{pmatrix}0 & -r_{-}^*\\ r_{+} & e^{-2itg_{\textup{\uppercase\expandafter{\romannumeral1}},+}} \end{pmatrix}, &k\in (-c_l,-c_{r})\cup(c_{r},c_l),\\
            \begin{pmatrix}0 & -1 \\ 1 & 0\end{pmatrix},  &k\in (-c_{r}, c_{r}).
            \end{cases}
    \end{equation}
    \item As $k\rightarrow\infty$ in $\mathbb{C}\setminus\mathbb{R}$, we have $M^{(1)}(k)=I+\mathcal{O}(k^{-1})$.
    \item $M^{(1)}(k)$ admits the same singular behavior as $M(k)$ at branch points $\pm c_{l},\pm c_{r}$.
\end{itemize}
\end{RHP}
\subsection{Second transformation: $M^{(1)}\rightarrow M^{(2)}$}
In this section, we introduce an auxiliary function 
$D_{\textup{\uppercase\expandafter{\romannumeral1}}}$ to 
pave the way for the subsequent contour deformation 
along the rays $(-\infty, -\eta_{l})$ and $(\eta_{l}, +\infty)$.
\subsubsection{The $D_{\textup{\uppercase\expandafter{\romannumeral1}}}$ function}
Define $D_{\textup{\uppercase\expandafter{\romannumeral1}}}:\mathcal{R}_{\textup{\uppercase\expandafter{\romannumeral1}}}\times\mathbb{C}\setminus([-\eta_l,-c_r]\cup[c_r,\eta_l])\to\mathbb{C}$ by
\begin{equation}\label{equ:def D function RI}
\begin{aligned}
D_{\mathrm{I}}(\xi;k):=\exp\bigg\{\frac{X_{l}(k)}{2\pi i}
\bigg[&\left(\int_{-\eta_l(\xi)}^{-c_l}+\int_{c_l}^{\eta_l(\xi)}\right)
\frac{\log(1-r(s)r^{*}(s))}{X_{l}(s)(s-k)}\dif s \\
&+ \left(\int_{-c_l}^{-c_{r}}+\int_{c_{r}}^{c_l}\right)
\frac{\log r_{+}(s)}{X_{l+}(s)(s-k)}\dif s\bigg]\bigg\}.
\end{aligned}
\end{equation}

The necessary properties of the function 
$D_{\textup{\uppercase\expandafter{\romannumeral1}}}$ are given as follows.
\begin{proposition}\label{prop: D function RI}
The function $D_{\textup{\uppercase\expandafter{\romannumeral1}}}$ defined in \eqref{equ:def D function RI} satisfies the following properties for $\xi\in \mathcal{R}_{\textup{\uppercase\expandafter{\romannumeral1}}}$:
\begin{itemize}
    \item[\rm (a)]$D_{\textup{\uppercase\expandafter{\romannumeral1}}}(k)$ is holomorphic for $k\in \mathbb{C}\backslash([-\eta_l,-c_r]\cup[c_r,\eta_l])$.
    \item[\rm (b)]$D_{\textup{\uppercase\expandafter{\romannumeral1}}}(k)$ satisfies the following jump relations:
    \begin{equation*}
          \begin{aligned}
       & D_{\textup{\uppercase\expandafter{\romannumeral1}}, +}(k)=D_{\textup{\uppercase\expandafter{\romannumeral1}}, -}(k)(1-r(k)r^*(k)), &&k\in(-\eta_l,-c_l)\cup(c_l,\eta_l), \\
       & D_{\textup{\uppercase\expandafter{\romannumeral1}}, +}(k)D_{\textup{\uppercase\expandafter{\romannumeral1}}, -}(k)=r_{+}(k), &&k\in (-c_l,-c_{r})\cup(c_{r},c_l),\\
       & D_{\textup{\uppercase\expandafter{\romannumeral1}}, +}(k)D_{\textup{\uppercase\expandafter{\romannumeral1}}, -}(k)=1, &&k\in (-c_{r}, c_{r}).
        \end{aligned} 
    \end{equation*}
    \item[\rm (c)]$D_{\textup{\uppercase\expandafter{\romannumeral1}}}(k)$ admits the symmetry $D_{\textup{\uppercase\expandafter{\romannumeral1}}}(-k)=D_{\textup{\uppercase\expandafter{\romannumeral1}}}(k)^{-1}$ for $k\in\mathbb{C}\backslash([-\eta_l,-c_r]\cup[c_r,\eta_l])$.
    \item[\rm (d)]As $k\rightarrow \infty$, $D_{\textup{\uppercase\expandafter{\romannumeral1}}}(\xi;k)=1+\mathcal{O}(k^{-1})$ for all $\xi\in\mathcal{R}_{\textup{\uppercase\expandafter{\romannumeral1}}}$.
    \item[\rm (e)]At the branch points $\pm c_l$ and $\pm c_r$, we have:
    \begin{equation}\label{equ:singular behavior of D}
         \begin{aligned}
       & D_{\textup{\uppercase\expandafter{\romannumeral1}}}(k)=(c_{l}\mp k )^{\frac{1}{4}}e^{d_{\textup{\uppercase\expandafter{\romannumeral1}},0}}\left(1+\mathcal{O}\left((c_l\mp k)^{1/2}\right)\right), &&k\rightarrow\pm c_{l},\; k\in\overline{\mathbb{C}^+ },  \\
       & D_{\textup{\uppercase\expandafter{\romannumeral1}}}(k)=(\pm k- c_r)^{-\frac{\arg b_{r,0}}{2\pi}}e^{ d_{\textup{\uppercase\expandafter{\romannumeral1}},r}}\left(1+\mathcal{O}\left((\pm k- c_r)^{1/2}\right)\right), &&k\rightarrow\pm c_{r},\;k\in\overline{\mathbb{C}^+ },  
    \end{aligned}
    \end{equation}
   
where
\begin{align*}
    d_{\textup{\uppercase\expandafter{\romannumeral1}},0}&=-\frac{1}{2}\log\left(-b_{l,0}\overline{b_{l,1}}-\overline{b_{l,0}}b_{l,1}\right)+\frac{1}{2}\log b_{l,0},\\
    d_{\textup{\uppercase\expandafter{\romannumeral1}},r}&=\frac{\sqrt{c_l^2-c_r^2}}{2\pi}\left[\left(\int_{-\eta_l}^{-c_l}+\int_{c_l}^{\eta_l}\right)\frac{\log(1-r(s)r^{*}(s))}{X_{l}(s)(s\mp c_r)}\dif s+\int_{- c_{l}}^{- c_r}\frac{\log r_{+}(s)}{X_{l+}(s)(s- c_r)}\dif s\right]\\
    &\hspace*{1em}+\frac{\arg b_{r,0}}{2\pi }\log (c_r-c_l)
\end{align*}
  with the principal branch being chosen for the complex power functions as well as the logarithms.
   \item[\rm (f)] As $k$ approaches $\pm\eta_l$ non-tangentially, we have
   \begin{equation}\label{dbr}
        \begin{aligned}
      & D_{\textup{\uppercase\expandafter{\romannumeral1}}}(k)=(k-\eta_l)^{i\nu(\eta_l)}D_{b,r}(k),\quad &&k\to\eta_l,\\
       & D_{\textup{\uppercase\expandafter{\romannumeral1}}}(k)=(-k-\eta_l)^{-i\nu(\eta_l)}D_{b,l}(k),\quad &&k\to-\eta_l,
   \end{aligned}
   \end{equation}
    where $\nu(k)=-\frac{1}{2\pi}\log(1-r(k)r^*(k))$. Here, $D_{b,r}(k)$ and $D_{b,l}(k)$ are uniformly bounded near $\pm\eta_l$ with the following estimates:
    \begin{align}\label{estimate D1}
        \left|\frac{D_{b,r}( k)}{D_{b,r}\left(\eta_l\right)}-1\right| \lesssim \left|k-\eta_l\right|\left(1+\log | k-\eta_l|\right),\;\left|\frac{D_{b,l}( k)}{D_{b,l}\left(-\eta_l\right)}-1\right| \lesssim \left|k+\eta_l\right|\left(1+\log | k+\eta_l|\right).
    \end{align}
    Here $D_{b,r}\left(\eta_l\right)=e^{id_{\textup{\uppercase\expandafter{\romannumeral1}},\eta}},D_{b,l}\left(-\eta_l\right)=e^{-id_{\textup{\uppercase\expandafter{\romannumeral1}},\eta}}$ with
    \begin{align*}
        d_{\textup{\uppercase\expandafter{\romannumeral1}},\eta}&=-\frac{\sqrt{\eta_l^2-c_l^2}}{2\pi }\left[\left(\int_{-c_l}^{-c_{r}}+\int_{c_{r}}^{c_l}\right)\frac{\log r_{+}(s)}{X_{l+}(s)(s-\eta_l)}\dif s+\int_{-\eta_l}^{- c_l}\frac{\log(1-r(s)r^{*}(s))}{X_{l}(s)(s- \eta_l)}\dif s\right]\\
    &\hspace*{1em}-\nu(\eta_l)\log (\eta_l-c_l),
    \end{align*}
   and the principal branch being used for the logarithms.
\end{itemize}
\end{proposition}
\begin{proof}
\hfill
\begin{itemize}
   \item[(a)] It follows from the definition of $D_{\textup{\uppercase\expandafter{\romannumeral1}}}$ by considering basic ideas
   of Cauchy transformation.
   \item[(b)] The jump condition is an immediate consequence by 
   using the well-known Sokhotski-Plemelj formula.
   \item[(c)] It can be straightforward verified.
   \item[(d)] With the aid of the fact $|r(k)|=1$ for $k\in (-c_l,-c_{r})\cup(c_{r},c_l)$, 
   the asymptotics of $D_{\textup{\uppercase\expandafter{\romannumeral1}}}(\xi;k)$ at 
   infinity denoted by $D_{\textup{\uppercase\expandafter{\romannumeral1}}, \infty}(\xi)$ 
   is given by 
    \begin{equation*}
        D_{\textup{\uppercase\expandafter{\romannumeral1}}, \infty}(\xi)=\exp\left\{-\frac{1}{2\pi i}\left[\left(\int_{-\eta_l}^{-c_l}+\int_{c_l}^{\eta_l}\right)\frac{\log(1-r(s)r^{*}(s))}{X_{l}(s)}\,\mathrm{d}s+\left(\int_{-c_l}^{-c_{r}}+\int_{c_{r}}^{c_l}\right)\frac{i\arg r_{+}(s)}{X_{l+}(s)}\,\mathrm{d}s\right]\right\}.
        \end{equation*}
        By the symmetry relation \eqref{symmetry a,b,r} for $r(k)$ and the fact that $X_{l}(-k)=-X_{l}(k)$ for $k\in(c_l,\infty)$,
        it can be directly verified that $D_{\textup{\uppercase\expandafter{\romannumeral1}}, \infty}(\xi)=1$ 
        for all $\xi\in\mathcal{R}_{\textup{\uppercase\expandafter{\romannumeral1}}}$.
       
        \item[(e)] This follows from the definition of 
        $D_{\textup{\uppercase\expandafter{\romannumeral1}}}$ and 
        the asymptotic expansion of $r(k)$ near the endpoints, as given in 
        \eqref{expansionrk}.
        
        \item[(f)] The singularity at $k = \pm\eta_l$ arise 
        from the logarithmic term $\log(1 - r(s) r^*(s))$ in the integrand. 
        Expanding near $s = \eta_l$, we obtain the leading behavior 
        shown in \eqref{dbr}. The explicit formulae for $d_{\textup{\uppercase\expandafter{\romannumeral1}},\eta}$ 
        and the estimates for $D_{b,r}(k)$ and $D_{b,l}(k)$ are direct consequences 
        derived by Taylor expansion and properties of Cauchy integrals.
\end{itemize}
\end{proof}
% \begin{remark}
%     Item (d) of Proposition \ref{prop: D function RI} is necessary for the solution to the mKdV equation to be consistent with boundary conditions \eqref{boundaryconditions}.
% \end{remark}
\subsubsection{RH problem for $M^{(2)}$}
With the help of $D_{\textup{\uppercase\expandafter{\romannumeral1}}}$ function,
we define
\begin{equation}\label{def:M2 RI}
    M^{(2)}(x,t;k)=M^{(1)}(x,t;k)D_{\textup{\uppercase\expandafter{\romannumeral1}}}^{-\sigma_3}(\xi;k).
\end{equation}
Then RH problem for $M^{(2)}$ reads as follows:
\begin{RHP}
\hfill
\begin{itemize}
    \item $M^{(2)}(k)$ is holomorphic for $k\in\mathbb{C}\backslash\mathbb{R}$.
    \item $M^{(2)}(k)$ has continuous boundary values $M_{\pm}^{(2)}(k)$ on $\mathbb{R}$ with the jump condition
    \begin{equation*}
        M^{(2)}_{+}(k)=M^{(2)}_{-}(k)V^{(2)}(k),
    \end{equation*}
    where
    \begin{equation*}
        V^{(2)}(k)=
            \begin{cases}
            \begin{pmatrix}1-rr^* & -D_{\textup{\uppercase\expandafter{\romannumeral1}}}^2r^*e^{-2itg_\textup{\uppercase\expandafter{\romannumeral1}}}\\ D_{\textup{\uppercase\expandafter{\romannumeral1}}}^{-2}re^{2itg_\textup{\uppercase\expandafter{\romannumeral1}}} & 1\end{pmatrix}, &k\in(-\infty,-\eta_l)\cup(\eta_l,+\infty),  \\
            \begin{pmatrix}1 & -D_{\textup{ \uppercase\expandafter{\romannumeral1}}, +}^{2}r_2e^{-2itg_\textup{\uppercase\expandafter{\romannumeral1}}} \\ D_{\textup{\uppercase\expandafter{\romannumeral1}}, -}^{-2}r_2^*e^{2itg_\textup{\uppercase\expandafter{\romannumeral1}}} & 1-rr^* \end{pmatrix}, &k\in(-\eta_l, -c_l)\cup(c_l,\eta_l),\\
            \begin{pmatrix}0 & -1\\ 1 & D_{\textup{ \uppercase\expandafter{\romannumeral1}}, +}D_{\textup{\uppercase\expandafter{\romannumeral1}}, -}^{-1}e^{-2itg_{\textup{\uppercase\expandafter{\romannumeral1}},+}} \end{pmatrix}, &k\in (-c_l,-c_{r})\cup(c_{r},c_l),\\
            \begin{pmatrix}0 & -1 \\ 1 & 0\end{pmatrix},  &k\in (-c_{r}, c_{r}).
            \end{cases}
    \end{equation*}
    \item As $k\rightarrow\infty$ in $\mathbb{C}\setminus\mathbb{R}$, we have $M^{(2)}(k)=I+\mathcal{O}(k^{-1})$.
    \item As $k\rightarrow\pm c_{l}$, we have $M^{(2)}(k)=\mathcal{O}\left(\left(k\mp c_{l}\right)^{-\frac{1}{4}}\right)$.
    %\item As $k\rightarrow\pm c_{r}$, we have $M^{(2)}(k)=\mathcal{O}(1)$.
\end{itemize}
\end{RHP}
\subsection{Third transformation: $M^{(2)}\rightarrow M^{(3)}$}
The aim of the third transformation is to open lenses in 
regions $U^{(3)}_j,U^{(3)*}_j,\;j=1,2,3$, which are 
illustrated in Figure \ref{fig:U_j domain RI}. 
To this end, we need to construct analytic approximations 
for the spectral functions $r$ and $r_2$.

\subsubsection{Analytic approximations} 
The following two propositions establish the 
analytic approximations for $r$ and $r_2$ in different regions
as illustrated in Figure \ref{fig:U_j domain RI}. Their proofs 
are analogous to those of Lemma 5.2 and Lemma 5.3 in \cite{LenellsDNLS}, respectively, 
and thus are omitted here.
\begin{proposition}[Analytic approximation of $r$]\label{analytic extension of r RI}There exist continuous functions
\begin{equation*}
    r_a: \mathcal{R}_\textup{\uppercase\expandafter{\romannumeral1}}  \times \left(\overline{U_1^{(3)}}\cup\overline{U_3^{(3)}}\right) \rightarrow \mathbb{C} \; \text { and } \; r_r: \mathcal{R}_\textup{\uppercase\expandafter{\romannumeral1}} \times\left((-\infty,-\eta_l)\cup(\eta_l,+\infty)\right) \rightarrow \mathbb{C},
\end{equation*}
which satisfy the following properties:
\begin{itemize}
    \item [\rm (a)] $r(k)=r_a(\xi; k)+r_r(\xi; k)$ for all $(\xi; k) \in \mathcal{R}_\textup{\uppercase\expandafter{\romannumeral1}}  \times\{(-\infty,-\eta_l)\cup(\eta_l,+\infty)\}$.
    \item [\rm (b)] $r_a(-k)=r_a^*(k)$ for $k\in\overline{U_1^{(3)}}\cup\overline{U_3^{(3)}}.$
    \item [\rm (c)] For all $\xi \in \mathcal{R}_\textup{\uppercase\expandafter{\romannumeral1}}$, the function $r_a: U_1^{(3)}\cup U_3^{(3)} \rightarrow \mathbb{C}$ is holomorphic. 
    Moreover, for $(\xi; k)\in \mathcal{R}_\textup{\uppercase\expandafter{\romannumeral1}}\times \left(\overline{U_1^{(3)}}\cup\overline{U_3^{(3)}}\right)$, we have 
\begin{equation*}
    \left|r_a(\xi; k)-r\left(\pm \eta_l\right)-r^{\prime}\left(\pm \eta_l\right)\left(k\mp \eta_l\right)\right| \lesssim \left|k\mp \eta_l\right|^2 \mathrm{e}^{t|\operatorname{Im} g_\textup{\uppercase\expandafter{\romannumeral1}}(\xi;k)|},
\end{equation*}
% for all $\xi\in\mathcal{R}_\textup{\uppercase\expandafter{\romannumeral1}}$ uniformly, 
and
\begin{equation*}
    \left|r_a(\xi; k)\right| \lesssim \frac{\mathrm{e}^{t|\operatorname{Im} g_\textup{\uppercase\expandafter{\romannumeral1}}(\xi; k)|}}{1+|k|^2}.
\end{equation*}
\item [\rm (d)] For all $\xi \in \mathcal{R}_\textup{\uppercase\expandafter{\romannumeral1}}$, the function $r_r \in L^p\left((-\infty,-\eta_l)\cup(\eta_l,+\infty)\right),\; p\in[1,+\infty]$, and as $t \rightarrow +\infty$,
\begin{equation*}
    \left\|r_r(\xi;k)\right\|_{L^p\left((-\infty,-\eta_l)\cup(\eta_l,+\infty)\right)} = \mathcal{O}\left(t^{-1}\right).
\end{equation*}
\end{itemize} 
\end{proposition}                        
\begin{proposition}[Analytic approximation of $r_2$]\label{analytic extension for r_2 RI} There exist continuous functions
    \begin{equation*}
    r_{2, a}: \mathcal{R}_\textup{\uppercase\expandafter{\romannumeral1}}  \times \left(\overline{U_2^{(3)}} \backslash\left\{\pm c_l\right\} \right)\rightarrow \mathbb{C} \; \text { and } \; r_{2, r}: \mathcal{R}_\textup{\uppercase\expandafter{\romannumeral1}} \times\left(\left(-\eta_l, -c_l\right)\cup(c_l,\eta_l) \right)\rightarrow \mathbb{C},
\end{equation*}
which satisfy the following properties:
\begin{itemize}
    \item [\rm (a)]  $r_2(k)=r_{2, a}(\xi; k)+r_{2,r}(\xi;k)$ for all $(\xi; k) \in \mathcal{R}_\textup{\uppercase\expandafter{\romannumeral1}} \times\left(-\eta_l, -c_l\right)\cup(c_l,\eta_l)$.
    \item [\rm (b)] $r_{2,a}(-k)=r_{2,a}^*(k)$ for $k\in\overline{U_2^{(3)}} \backslash\{\pm c_l\}$.
    \item [\rm (c)]  For all $\xi\in \mathcal{R}_\textup{\uppercase\expandafter{\romannumeral1}}$, the function $r_{2, a}: U_2^{(3)} \rightarrow \mathbb{C}$ is holomorphic, and $r_{2, a}(\xi;k)= \mathcal{O}\left(\left|k\mp c_l\right|^{-1 / 2}\right)$ as $k \rightarrow \pm c_l$ respectively. Moreover, for every $\varepsilon>0$, there exists a constant $C(\varepsilon)>0$ such that for $(\xi;k)\in \mathcal{R}_\textup{\uppercase\expandafter{\romannumeral1}}\times\left(\overline{U^{(3)}_2}\setminus D_{\varepsilon}(\pm c_l)\right)$, we have
\begin{equation*}
    \left|r_{2, a}(\xi;k)\right| \leqslant C(\varepsilon )\frac{\mathrm{e}^{t\left|\operatorname{Im} g_{\textup{\uppercase\expandafter{\romannumeral1}},+}(\xi;k)\right|}}{1+|k|^2}.
\end{equation*}
\item [\rm (d)] For all $\xi\in \mathcal{R}_\textup{\uppercase\expandafter{\romannumeral1}}$, the function $r_{2, r} \in L^p \left(\left(-\eta_l, -c_l\right)\cup(c_l,\eta_l)\right),\;p\in[1,+\infty]$, and as $t \rightarrow +\infty$,
\begin{equation*}
    \left\|r_{2, r}(\xi;k)\right\|_{L^p\left(\left(-\eta_l, -c_l\right)\cup(c_l,\eta_l)\right)} = \mathcal{O}\left(t^{-1}\right).                        
  \end{equation*}
\item [\rm (e)] For all $\xi \in \mathcal{R}_\textup{\uppercase\expandafter{\romannumeral1}}$, the function $\left(r r_{2, a}+r^* r_{2, a}^*+1\right) \mathrm{e}^{-2 i t g_{\textup{\uppercase\expandafter{\romannumeral1}},+}}\in L^p\left((-c_l,-c_r)\cup(c_r,c_l)\right),\;p\in[1,+\infty]$, and as $t \rightarrow +\infty$,
\begin{equation*}
    \left\|\left(r r_{2, a}+r ^*r^*_{2, a}+1\right)(\xi;k) \mathrm{e}^{-2 i t g_{\textup{\uppercase\expandafter{\romannumeral1}},+}(\xi;k)}\right\|_{L^p\left((-c_l,-c_r)\cup(c_r,c_l)\right)}=\mathcal{O}\left(t^{-1}\right).
\end{equation*}
\item [\rm (f)] For all $\xi \in \mathcal{R}_\textup{\uppercase\expandafter{\romannumeral1}}$, the function $\left(r_{2,a}+r_{2,a}^*\right) \mathrm{e}^{-2 i t g_{\textup{\uppercase\expandafter{\romannumeral1}},+}}\in L^p\left(-c_r,c_r\right),\;p\in[1,+\infty]$, and as $t \rightarrow +\infty$,
\begin{equation*}
    \left\|\left(r_{2,a}+r_{2,a}^*\right)(\xi;k) \mathrm{e}^{-2 i t g_{\textup{\uppercase\expandafter{\romannumeral1}},+}(\xi;k)}\right\|_{L^p\left(-c_r,c_r\right)}=\mathcal{O}\left(t^{-1}\right).
\end{equation*}
\end{itemize}

\end{proposition}
%\begin{remark}
 %   Items (e) and (f) of Proposition \ref{analytic extension for r_2 RI} ensure that the (2,2) entry of $V^{(3)}$ in \eqref{equ:jump V3 RI} decays to $0$ on intervals $(-c_l,-c_r)$, $(c_r,c_l)$ and $(-c_l,c_l)$.
%\end{remark}
\begin{figure}[H]
\begin{center}
\tikzset{every picture/.style={line width=0.75pt}} %set default line width to 0.75pt
\begin{tikzpicture}[x=0.75pt,y=0.75pt,yscale=-1,xscale=1]
%uncomment if require: \path (0,300); %set diagram left start at 0, and has height of 300
%Straight Lines [id:da804783244241797]
 \newcommand{\gap}{12}
\draw    (240,151) -- (412,151) ;
\draw [shift={(332,151)}, rotate = 180] [color={rgb, 255:red, 0; green, 0; blue, 0 }  ][line width=0.75]    (10.93,-3.29) .. controls (6.95,-1.4) and (3.31,-0.3) .. (0,0) .. controls (3.31,0.3) and (6.95,1.4) .. (10.93,3.29)   ;
\draw [shift={(332,151)}, rotate = 180] [color={rgb, 255:red, 0; green, 0; blue, 0 }  ][line width=0.75]    (4.93-5*\gap,-3.29) .. controls (0.95-5*\gap,-1.4) and (-3.31-5*\gap,-0.3) .. (-6-5*\gap,0) .. controls (-3.31-5*\gap,0.3) and (0.95-5*\gap,1.4) .. (4.93-5*\gap,3.29)   ;
\draw [shift={(332,151)}, rotate = 180] [color={rgb, 255:red, 0; green, 0; blue, 0 }  ][line width=0.75]    (4.93+6*\gap,-3.29) .. controls (0.95+6*\gap,-1.4) and (-3.31+6*\gap,-0.3) .. (-6+6*\gap,0) .. controls (-3.31+6*\gap,0.3) and (0.95+6*\gap,1.4) .. (4.93+6*\gap,3.29)   ;
\draw [shift={(332,151)}, rotate = 180] [color={rgb, 255:red, 0; green, 0; blue, 0 }  ][line width=0.75]    (4.93-10*\gap,-3.29) .. controls (0.95-10*\gap,-1.4) and (-3.31-10*\gap,-0.3) .. (-6-10*\gap,0) .. controls (-3.31-10*\gap,0.3) and (0.95-10*\gap,1.4) .. (4.93-10*\gap,3.29)   ;
\draw [shift={(332,151)}, rotate = 180] [color={rgb, 255:red, 0; green, 0; blue, 0 }  ][line width=0.75]    (4.93+11*\gap,-3.29) .. controls (0.95+11*\gap,-1.4) and (-3.31+11*\gap,-0.3) .. (-6+11*\gap,0) .. controls (-3.31+11*\gap,0.3) and (0.95+11*\gap,1.4) .. (4.93+11*\gap,3.29)   ;
\draw [shift={(332,151)}, rotate = 180] [color={rgb, 255:red, 0; green, 0; blue, 0 }  ][line width=0.75]    (4.93-18*\gap,-3.29) .. controls (0.95-18*\gap,-1.4) and (-3.31-18*\gap,-0.3) .. (-6-18*\gap,0) .. controls (-3.31-18*\gap,0.3) and (0.95-18*\gap,1.4) .. (4.93-18*\gap,3.29)   ;
\draw [shift={(332,151)}, rotate = 180] [color={rgb, 255:red, 0; green, 0; blue, 0 }  ][line width=0.75]    (4.93+19*\gap,-3.29) .. controls (0.95+19*\gap,-1.4) and (-3.31+19*\gap,-0.3) .. (-6+19*\gap,0) .. controls (-3.31+19*\gap,0.3) and (0.95+19*\gap,1.4) .. (4.93+19*\gap,3.29)   ;
%Straight Lines [id:da5569742615021391]
\draw    (323,51) -- (481,152) ;
\draw [shift={(407.06,104.73)}, rotate = 212.59] [color={rgb, 255:red, 0; green, 0; blue, 0 }  ][line width=0.75]    (10.93,-3.29) .. controls (6.95,-1.4) and (3.31,-0.3) .. (0,0) .. controls (3.31,0.3) and (6.95,1.4) .. (10.93,3.29)   ;
%Straight Lines [id:da4789393691693813]
\draw    (481,152) -- (588,222) ;
\draw [shift={(539.52,190.28)}, rotate = 213.19] [color={rgb, 255:red, 0; green, 0; blue, 0 }  ][line width=0.75]    (10.93,-3.29) .. controls (6.95,-1.4) and (3.31,-0.3) .. (0,0) .. controls (3.31,0.3) and (6.95,1.4) .. (10.93,3.29)   ;
%Straight Lines [id:da6299994774957993]
\draw    (481,152) -- (332,249) ;
\draw [shift={(412.37,196.68)}, rotate = 146.94] [color={rgb, 255:red, 0; green, 0; blue, 0 }  ][line width=0.75]    (10.93,-3.29) .. controls (6.95,-1.4) and (3.31,-0.3) .. (0,0) .. controls (3.31,0.3) and (6.95,1.4) .. (10.93,3.29)   ;
%Straight Lines [id:da7282944804854083]
\draw    (582,86) -- (481,152) ;
\draw [shift={(537.36,115.17)}, rotate = 146.84] [color={rgb, 255:red, 0; green, 0; blue, 0 }  ][line width=0.75]    (10.93,-3.29) .. controls (6.95,-1.4) and (3.31,-0.3) .. (0,0) .. controls (3.31,0.3) and (6.95,1.4) .. (10.93,3.29)   ;
%Straight Lines [id:da8230642158560675]
\draw    (171,150) -- (64.26,79.61) ;
\draw [shift={(123.47,118.66)}, rotate = 213.4] [color={rgb, 255:red, 0; green, 0; blue, 0 }  ][line width=0.75]    (10.93,-3.29) .. controls (6.95,-1.4) and (3.31,-0.3) .. (0,0) .. controls (3.31,0.3) and (6.95,1.4) .. (10.93,3.29)   ;
%Straight Lines [id:da7178793747548173]
\draw    (171,150) -- (323,51) ;
\draw [shift={(252.03,97.23)}, rotate = 146.92] [color={rgb, 255:red, 0; green, 0; blue, 0 }  ][line width=0.75]    (10.93,-3.29) .. controls (6.95,-1.4) and (3.31,-0.3) .. (0,0) .. controls (3.31,0.3) and (6.95,1.4) .. (10.93,3.29)   ;
%Straight Lines [id:da5987324696981353]
\draw    (69.76,215.63) -- (171,150) ;
\draw [shift={(125.41,179.55)}, rotate = 147.05] [color={rgb, 255:red, 0; green, 0; blue, 0 }  ][line width=0.75]    (10.93,-3.29) .. controls (6.95,-1.4) and (3.31,-0.3) .. (0,0) .. controls (3.31,0.3) and (6.95,1.4) .. (10.93,3.29)   ;
%Straight Lines [id:da1508450918134565]
\draw    (412,151) -- (591,151) ;
%Straight Lines [id:da32770708129160253]
\draw    (61,151) -- (240,151) ;
%Straight Lines [id:da5129961687643008]
\draw    (171,150) -- (332,249) ;
\draw [shift={(256.61,202.64)}, rotate = 211.59] [color={rgb, 255:red, 0; green, 0; blue, 0 }  ][line width=0.75]    (10.93,-3.29) .. controls (6.95,-1.4) and (3.31,-0.3) .. (0,0) .. controls (3.31,0.3) and (6.95,1.4) .. (10.93,3.29)   ;
% Text Node
\draw (477,156.4) node [anchor=north west][inner sep=0.75pt]  [font=\tiny]  {$\eta_l$};
\fill (481,151) circle (1.2pt);
% Text Node
\draw (162,156.4) node [anchor=north west][inner sep=0.75pt]  [font=\tiny]  {$-\eta_l$};
\fill (170,151) circle (1.2pt);
% Text Node
\draw (227,153.4) node [anchor=north west][inner sep=0.75pt]  [font=\tiny]  {$-c_{l}$};
\fill (241,151) circle (1.2pt);
% Text Node
\draw (408,153.4) node [anchor=north west][inner sep=0.75pt]  [font=\tiny]  {$c_{l}$};
\fill (412,151) circle (1.2pt);
\draw (358,153.4) node [anchor=north west][inner sep=0.75pt]  [font=\tiny]  {$c_{r}$};
\fill (362,151) circle (1.2pt);
\draw (277,153.4) node [anchor=north west][inner sep=0.75pt]  [font=\tiny]  {$-c_{r}$};
\fill (284,151) circle (1.2pt);
% Text Node
\draw (541,72.4) node [anchor=north west][inner sep=0.75pt]  [font=\tiny]  {$\Gamma _{1}^{( 3)}$};
% Text Node
\draw (549,215.4) node [anchor=north west][inner sep=0.75pt]  [font=\tiny]  {$\Gamma _{1}^{( 3) *}$};
% Text Node
\draw (416,77.4) node [anchor=north west][inner sep=0.75pt]  [font=\tiny]  {$\Gamma _{2}^{( 3)}$};
% Text Node
\draw (216,71.4) node [anchor=north west][inner sep=0.75pt]  [font=\tiny]  {$\Gamma _{3}^{( 3)}$};
% Text Node
\draw (106,79.4) node [anchor=north west][inner sep=0.75pt]  [font=\tiny]  {$\Gamma _{4}^{( 3)}$};
% Text Node
\draw (416,205.4) node [anchor=north west][inner sep=0.75pt]  [font=\tiny]  {$\Gamma _{2}^{( 3) *}$};
% Text Node
\draw (228,209.4) node [anchor=north west][inner sep=0.75pt]  [font=\tiny]  {$\Gamma _{3}^{( 3) *}$};
% Text Node
\draw (107,191.4) node [anchor=north west][inner sep=0.75pt]  [font=\tiny]  {$\Gamma _{4}^{( 3) *}$};
% Text Node
\draw (543,114.4) node [anchor=north west][inner sep=0.75pt]  [font=\tiny]  {$U_{1}^{( 3)}$};
% Text Node
\draw (546,162.4) node [anchor=north west][inner sep=0.75pt]  [font=\tiny]  {$U_{1}^{( 3) *}$};
% Text Node
\draw (314,98.4) node [anchor=north west][inner sep=0.75pt]  [font=\tiny]  {$U_{2}^{( 3)}$};
% Text Node
\draw (85,112.4) node [anchor=north west][inner sep=0.75pt]  [font=\tiny]  {$U_{3}^{( 3)}$};
% Text Node
\draw (317,172.4) node [anchor=north west][inner sep=0.75pt]  [font=\tiny]  {$U_{2}^{( 3) *}$};
% Text Node
\draw (85,163.4) node [anchor=north west][inner sep=0.75pt]  [font=\tiny]  {$U_{3}^{( 3) *}$};
\end{tikzpicture}
\caption{The jump contours $ \Gamma^{(3)}$ and regions $U^{(3)}_j,U^{(3)*}_j,\;j=1,2,3$ of RH problem for $ M^{(3)}$ when $\xi\in\mathcal{R}_\textup{\uppercase\expandafter{\romannumeral1}}$.}\label{fig:U_j domain RI}
\end{center}
\end{figure}

\subsubsection{RH problem for $ M^{(3)}$} 
Now we are ready to define $ M^{(3)}:= M^{(3)}(x,t;k)$ by
\begin{equation}\label{def:M3 RI}
     M^{(3)}(k)=M^{(2)}(k)D_{\textup{\uppercase\expandafter{\romannumeral1}}}^{\sigma_3}(k)G(k)D_{\textup{\uppercase\expandafter{\romannumeral1}}}^{-\sigma_3}(k),
\end{equation}
where
\begin{equation}
    G(k):=
        \begin{cases}
        \begin{pmatrix}1 & 0 \\ -r_ae^{2itg_{\textup{\uppercase\expandafter{\romannumeral1}}}} & 1\end{pmatrix}, &k\in U^{(3)}_1\cup U^{(3)}_3, \\
        \begin{pmatrix}1 & -r_a^{*}e^{-2itg_\textup{\uppercase\expandafter{\romannumeral1}}} \\ 0 & 1\end{pmatrix}, &k\in U_1^{(3)*}\cup U_3^{(3)*}, \\
        \begin{pmatrix}1 & r_{2,a}e^{-2itg_\textup{\uppercase\expandafter{\romannumeral1}}} \\ 0 & 1\end{pmatrix}, &k\in U_2^{(3)},\\
        \begin{pmatrix}1 & 0 \\ r^*_{2,a}e^{2itg_\textup{\uppercase\expandafter{\romannumeral1}}} & 1\end{pmatrix}, & k\in U_2^{(3)*},\\
        I,  &\textnormal{elsewhere}.
        \end{cases}
\end{equation}
Here, the domains $U^{(3)}_j$, $j=1,2,3$ are illustrated in Figure \ref{fig:U_j domain RI}. 

Then RH problem for $ M^{(3)}$ reads as follows:
\begin{RHP}
\hfill
\begin{itemize}
    \item $M^{(3)}(k)$ is holomorphic for $k\in\mathbb{C}\backslash\Gamma^{(3)}$, where $ \Gamma^{(3)}:=\cup_{j=1}^{4}(\Gamma^{(3)}_j\cup\Gamma^{(3)*}_j)\cup\mathbb{R}$; see Figure \ref{fig:U_j domain RI} for an illustration.
    \item $M^{(3)}(k)$ has continuous boundary values $M_{\pm}^{(3)}(k)$ on $k\in \Gamma^{(3)}$ with the jump condition
    \begin{equation*}
        M^{(3)}_{+}(k)= M^{(3)}_{-}(k)V^{(3)}(k),
    \end{equation*}
    where
    \begin{equation}\label{equ:jump V3 RI}
        V^{(3)}(k)=
            \begin{cases}
            \begin{pmatrix} 1 & 0\\ D_{\textup{\uppercase\expandafter{\romannumeral1}}}^{-2}r_ae^{2itg_\textup{\uppercase\expandafter{\romannumeral1}}} & 1\end{pmatrix}, &k\in\Gamma^{(3)}_1\cup\Gamma^{(3)}_4,  \\
            \begin{pmatrix} 1 & -D_{\textup{\uppercase\expandafter{\romannumeral1}}}^{2}r_a^*e^{-2itg_\textup{\uppercase\expandafter{\romannumeral1}}}\\ 0 & 1\end{pmatrix}, & k\in\Gamma^{(3)*}_1\cup\Gamma^{(3)*}_4,  \\
            \begin{pmatrix} 1 & -D_{\textup{\uppercase\expandafter{\romannumeral1}}}^{2}r_{2,a}e^{-2itg_\textup{\uppercase\expandafter{\romannumeral1}}}\\ 0 & 1\end{pmatrix}, & k\in\Gamma^{(3)}_2\cup\Gamma^{(3)}_3, \\
            \begin{pmatrix} 1 & 0 \\ D_{\textup{\uppercase\expandafter{\romannumeral1}}}^{-2}r^*_{2,a}e^{2itg_\textup{\uppercase\expandafter{\romannumeral1}}} & 1\end{pmatrix}, & k\in\Gamma^{(3)*}_2\cup\Gamma^{(3)*}_3, \\
            \begin{pmatrix}
                1-r_rr^*_r&-D_\textup{\uppercase\expandafter{\romannumeral1}}^2r_r^*e^{-2itg_\textup{\uppercase\expandafter{\romannumeral1}}}\\D_\textup{\uppercase\expandafter{\romannumeral1}}^{-2}r_re^{2itg_\textup{\uppercase\expandafter{\romannumeral1}}}&1
            \end{pmatrix},&k\in(-\infty,-\eta_l)\cup(\eta_l,+\infty),\\
            \begin{pmatrix}
                1&-D_{\textup{\uppercase\expandafter{\romannumeral1}},+}^2r_{2,r}e^{-2itg_\textup{\uppercase\expandafter{\romannumeral1}}}\\D_{\textup{\uppercase\expandafter{\romannumeral1}},-}^{-2}r_{2,r}^*e^{2itg_\textup{\uppercase\expandafter{\romannumeral1}}}&1-r_{2,r}r^*_{2,r}\left(1-rr^*\right)^2
            \end{pmatrix},&k\in(-\eta_l,-c_l)\cup(c_l,\eta_l),\\
            \begin{pmatrix}
0&-1\\1&D_{\textup{\uppercase\expandafter{\romannumeral1}},+}D_{\textup{\uppercase\expandafter{\romannumeral1}},-}^{-1}\left(rr_{2,a}+r^*r^*_{2,a}+1\right)e^{-2itg_{\textup{\uppercase\expandafter{\romannumeral1}},+}}\\
            \end{pmatrix},&k\in(-c_l,-c_r)\cup(c_r,c_l),\\  
            \begin{pmatrix}0 & -1 \\ 1 & D^2_{\textup{\uppercase\expandafter{\romannumeral1}},+}\left(r_{2,a}+r^*_{2,a}\right)e^{-2itg_{\textup{\uppercase\expandafter{\romannumeral1}},+}}\end{pmatrix},  & k\in (-c_r,c_r).
            \end{cases}
    \end{equation}
    \item As $k\rightarrow\infty$ in $\mathbb{C}\setminus\Gamma^{(3)}$, we have $ M^{(3)}(k)=I+\mathcal{O}(k^{-1})$.
    \item As $k\rightarrow\pm c_l$, we have $ M^{(3)}(k)=\mathcal{O}\left((k\mp c_l)^{-1/4}\right)$.
\end{itemize}
\end{RHP}
\begin{remark}\rm
    To obtain $V^{(3)}$ for $k\in(-c_l,c_l)$ , we use the facts (a) $D_{\textup{\uppercase\expandafter{\romannumeral1}},+}D_{\textup{\uppercase\expandafter{\romannumeral1}},-}=r_{+}$,
    $r_+r^*_-=1$ for $k\in (-c_l,-c_{r})\cup(c_{r},c_l)$, (b) $D_{\textup{\uppercase\expandafter{\romannumeral1}},+}D_{\textup{\uppercase\expandafter{\romannumeral1}},-}=1$, $r_{+}=-r_{-}^*$ for $k\in (-c_{r}, c_{r})$,
    and (c) $g_{\textup{\uppercase\expandafter{\romannumeral1}}, +}+g_{\textup{\uppercase\expandafter{\romannumeral1}}, -}=0$ for $k\in (-c_l,c_l)$.
\end{remark}

\subsection{Analysis of RH problem for $M^{(3)}$}
It is readily seen that $V^{(3)} \to I$ exponentially 
on the contours $\Gamma^{(3)}_j \cup \Gamma_j^{(3)*}$ for $j=1,\dots,4$ as $t \to +\infty$.
Items (e) and (f) of Proposition \ref{analytic extension for r_2 RI}  
ensure that the $(2,2)$-entry of $V^{(3)}$ in \eqref{equ:jump V3 RI} decays to zero 
on the intervals $(-c_l, -c_r)$, $(c_r, c_l)$, and $(-c_l, c_l)$ as $t\to+\infty$. 
Both of these two facts lead us to consider the following global parametrix.  
\subsubsection{Global parametrix}
%As $t$ large enough, the jump matrix $V^{(3)}$ approaches
%\begin{equation}\label{equ:jump Vinfty}
%    V^{(\infty)}=\begin{pmatrix} 0 & -1 \\ 1 & 0 \end{pmatrix}, \quad k\in(-c_l,c_l).
%\end{equation}
%For $k\in\mathbb{C}\backslash[-c_l,c_l]$, $V^{(3)}\rightarrow I$ as $t\rightarrow \infty$.
%Then the following parametrix for $M^{(\infty)}$ is naturally established.
\begin{RHP}\label{RHP:Minfty RI}
\hfill
\begin{itemize}
    \item $M^{(\infty)}(k)$ is holomorphic for $k\in\mathbb{C}\backslash[-c_l,c_l]$.
    \item $M^{(\infty)}(k)$ has continuous boundary values $M_{\pm}^{(\infty)}(k)$ on $(-c_l,c_l)$ satisfying the following jump condition
    \begin{align*}
        M_{+}^{(\infty)}(k)=M_{-}^{(\infty)}(k)\begin{pmatrix}
            0 & -1\\
            1 & 0
        \end{pmatrix}.
    \end{align*}
    \item As $k\rightarrow\infty$ in $\mathbb{C}\setminus[-c_l, c_l]$, we have $M^{(\infty)}(k)=I+\mathcal{O}(k^{-1})$.
    \item As $k\rightarrow \pm c_l$, $M^{(\infty)}(k)=\mathcal{O}((k\mp c_l)^{-1/4})$.
\end{itemize}
\end{RHP}
Then the unique solution to $M^{(\infty)}$ is given by.
\begin{align}\label{equ:sol of Minfty}
    M^{(\infty)}(k)=\Delta_l(k),
\end{align}
where $\Delta_{l}(k)$ is defined by \eqref{equ:Delta_j} for $j=l$.

\subsubsection{Local parametrices near saddle points}
Let
\begin{align*}
    D_\varrho(\eta_l)=\left\{k: |k-\eta_l|<\varrho \right\}, \quad D_\varrho(-\eta_l)=\left\{k: |k+\eta_l|<\varrho \right\}
\end{align*}
be two small disks around $\eta_l$ and $-\eta_l$ respectively, where
\begin{equation*}
\varrho<\frac{1}{3}{\rm min}\left\{ |\eta_l-c_l|, |\eta_l+c_l|,   |\eta_l|  \right\}.
\end{equation*}
For $j\in\{r,l\}$, we intend to solve the following local RH problem for $M^{(j)}$.
\begin{RHP}
\hfill
\begin{itemize}
    \item $M^{(j)}(k)$ is holomorphic for $k\in\overline{D_{\varrho}(\pm \eta_l)}\setminus\Gamma^{(j)}$ (``$+$'' for $j=r$, and ``$-$'' for $j=l$),
    where
    \begin{align}\label{def:Gamma^(ell) RI}
        \Gamma^{(r)}:=D_\varrho(\eta_l)\cap \Gamma^{(3)},\quad \Gamma^{(l)}:=D_\varrho(-\eta_l)\cap \Gamma^{(3)}.
    \end{align}
    %see Figure \ref{fig:local jumps RI} for an illustration.
    \item  $M^{(j)}(k)$ has continuous boundary values $M_{\pm}^{(j)}(k)$ on $k\in \Gamma^{(j)}$ with the jump condition
    \begin{align*}
        M^{(j)}_{+}(k)=M^{(j)}_{-}(k)V^{(3)}(k)\big|_{\Gamma^{(j)}},
    \end{align*}
    where $V^{(3)}(k)$ is given by \eqref{equ:jump V3 RI}.
    \item As $k\rightarrow\infty$ in $\mathbb{C}\setminus\Gamma^{(j)}$, we have $M^{(j)}(k)=I+\mathcal{O}(k^{-1})$.
\end{itemize}
\end{RHP}

 The solution $M^{(j)}$ for this local RH problem can be 
constructed by the parabolic cylinder parametrix shown in Appendix
\ref{appendix:pc model} in a standard manner. 
First, we introduce the following change of variables. For $k \in D_{\varrho}(\eta_l)$, we set
\begin{align}
& g_{\textup{\uppercase\expandafter{\romannumeral1}}}(k)=g_{\textup{\uppercase\expandafter{\romannumeral1}}}(\eta_l)+\frac{1}{4t}\zeta_r^2, \quad \zeta_r(k)=t^{1/2}(k-\eta_l)\psi_r(k),
\end{align}
where $\psi_r(k)$ is given by
% is holomorphic on $D_\varrho(\eta_l)$, and is given by
\begin{equation}
\psi_r(k)=\left[2g_{\textup{\uppercase\expandafter{\romannumeral1}}}^{\prime\prime}(\eta_l)+4\sum_{m=3}^\infty \frac{g_{\textup{\uppercase\expandafter{\romannumeral1}}}^{(m)}(\eta_l)}{m!}(k-\eta_l)^{m-2}\right]^{\frac{1}{2}}.
\end{equation}
Similarly, for $k\in D_{\varrho}(-\eta_l)$, we define
\begin{align}
    & g_{\textup{\uppercase\expandafter{\romannumeral1}}}(k)=-g_{\textup{\uppercase\expandafter{\romannumeral1}}}(\eta_l)-\frac{1}{4t}\zeta_l^2,  \quad \zeta_l(k)=-t^{1/2}(k+\eta_l)\psi_l(k),
\end{align}
where $\psi_l(k)$ is given by
% is holomorphic on $D_\varrho(-\eta_l)$, and is given by
\begin{align}
\psi_l(k)=\left[2g_{\textup{\uppercase\expandafter{\romannumeral1}}}^{\prime\prime}(\eta_l)-4\sum_{m=3}^\infty \frac{g_{\textup{\uppercase\expandafter{\romannumeral1}}}^{(m)}(-\eta_l)}{m!}(k+\eta_l)^{m-2}\right]^{\frac{1}{2}}.
\end{align}
A straightforward computation shows that $g_{\textup{\uppercase\expandafter{\romannumeral1}}}''(\eta_l)=\frac{24\left(-\xi+c_l^2/2\right)}{\sqrt{-\xi-c_l^2/2}}>0$,
therefore $\psi_r(\eta_l)=\psi_l(-\eta_l)=\sqrt{2g_{\textup{\uppercase\expandafter{\romannumeral1}}}''(\eta_l)}>0$. 
Now the change of variables can be summarized as
\begin{equation}\label{change variable-3}
\begin{aligned}
    &k=\eta_l+t^{-1/2}\zeta_r\psi^{-1}_r,&& k\in D_\varrho(\eta_l),\\
    &k=-\eta_l-t^{-1/2}\zeta_l\psi^{-1}_l,&& k\in D_\varrho(-\eta_l).
\end{aligned}
\end{equation}
Recalling the behavior of $D_\textup{\uppercase\expandafter{\romannumeral1}}$ at $\pm\eta_l$ in \eqref{equ:singular behavior of D}, 
the function $D_\textup{\uppercase\expandafter{\romannumeral1}}(k)$ for $k\in D_\varrho(\pm \eta_l)$ can be expressed as
\begin{equation*}
    \begin{aligned}
    D_\textup{\uppercase\expandafter{\romannumeral1}}(k)&=(k-\eta_l)^{i\nu(\eta_l)}D_0D_{1,r}(k), &&k\in D_\varrho( \eta_l),\\
     D_\textup{\uppercase\expandafter{\romannumeral1}}(k)&=(-k-\eta_l)^{-i\nu(\eta_l)}D_0^{-1}D_{1,l}(k),&& k\in D_\varrho( -\eta_l),
\end{aligned}
\end{equation*}
   where
   \begin{align}
      & D_0=t^{-\frac{1}{2}i\nu(\eta_l)}\psi_r(\eta_l)^{-i\nu(\eta_l)}D_{b,r}(\eta_l),\label{D_0 RI}\\
      & D_{1,r}=e^{-i\nu(\eta_l)\log\left(\frac{\psi_r(k)}{\psi_r(\eta_l)}\right)}\frac{D_{b,r}(k)}{D_{b,r}(\eta_l)},&& k\in D_\varrho( \eta_l)\setminus(-\infty,\eta_l],\\
      & D_{1,l}=e^{i\nu(\eta_l)\log\left(\frac{\psi_l(k)}{\psi_l(-\eta_l)}\right)}\frac{D_{b,l}(k)}{D_{b,l}(-\eta_l)},&& k\in D_\varrho(-\eta_l)\setminus[-\eta_l,+\infty).
   \end{align}
   Here, $D_{b,r}(k)$ and $D_{b,l}(k)$ are defined as in \eqref{dbr}.  
   Due to the symmetry of $M(k)$, it suffices to construct the local parametrix near $\eta_l$; the construction near $-\eta_l$ follows in a similar manner.
As $t\to+\infty$, $k(\zeta_r)\to\eta_l$, we obtain
\begin{equation*}
    D_{1,r}\to1,\quad r_a\to r(\eta_l),\quad r_{2,a}\to\frac{\overline{r(\eta_l)}}{1-|r(\eta_l)|^2}.
\end{equation*}
Now we define the following local parametrix
\begin{align}
    &\tilde M^{(r)}(x,t;k)=Y^{(r)}(k)M^{(\textup{PC})}(\zeta_r(k);r(\eta_l))D_0^{-\sigma_3}e^{itg_\textup{\uppercase\expandafter{\romannumeral1}}(\eta_l)\sigma_3}, \label{tilde M(r)} \\
    & Y^{(r)}(k)=M^{(\infty)}(k)e^{-itg_\textup{\uppercase\expandafter{\romannumeral1}}(\eta_l)\sigma_3}D_0^{\sigma_3} \label{equ: Yr(k)}.
\end{align}
Here, $Y^{(r)}(k)$ is an analytic factor for $k\in D_{\varrho}(\eta_l)$, and $M^{(\textup{PC})}$ is the solution to RH problem \ref{RHP: pc parametrix}. 
Furthermore, as $t\to+\infty$, $M^{(\textup{PC})}(\zeta_r(k);r(\eta_l))$ takes the following asymptotics
\begin{equation}\label{MPC}
    M^{(\textup{PC})}(\zeta_r(k);r(\eta_l))=I+\frac{1}{t^{1/2}(k-\eta_l)\psi_r(\eta_l)}\begin{pmatrix}
       0 & -i{\beta}^{(\eta_l)}\\ i\overline{{\beta}^{(\eta_l)}} & 0\end{pmatrix}+\mathcal{O}(t^{-1}),
\end{equation}
where $\beta^{(\eta_l)}$ is defined as in \eqref{beta}. The $\tilde{M}^{(r)}$ defined in \eqref{tilde M(r)} could be viewed as an approximation of $M^{(r)}$ for large $t$,
which takes the jump matrix $\tilde V^{(r)}(k)$ as
\begin{equation}\label{jump matrix tildev3 I}
   \tilde V^{(r)}(k)= \begin{cases}
     \begin{pmatrix} 1 & 0\\ r(\eta_l)e^{\frac{i\zeta_r^2}{2}}\zeta_r^{-2i\nu(\eta_l)}& 1\end{pmatrix}, &k\in\Gamma^{(3)}_1\cap D_\varrho( \eta_l),  \\
            \begin{pmatrix} 1 & -\overline{r(\eta_l)}e^{-\frac{i\zeta_r^2}{2}}\zeta_r^{2i\nu(\eta_l)}\\ 0 & 1\end{pmatrix}, & k\in\Gamma^{(3)*}_1\cap D_\varrho( \eta_l),  \\
            \begin{pmatrix} 1 & -r_{2}(\eta_l)e^{-\frac{i\zeta_r^2}{2}}\zeta_r^{2i\nu(\eta_l)}\\ 0 & 1\end{pmatrix}, & k\in\Gamma^{(3)}_2\cap D_\varrho( \eta_l), \\
            \begin{pmatrix} 1 & 0 \\ \overline{r_{2}(\eta_l)}e^{\frac{i\zeta_r^2}{2}}\zeta_r^{-2i\nu(\eta_l)} & 1\end{pmatrix}, & k\in\Gamma^{(3)*}_2\cap D_\varrho( \eta_l).
\end{cases}
\end{equation}

For $\tilde M^{(l)}$, from the symmetric condition that $M^{(l)}(k)=\sigma_1M^{(r)}(-k)\sigma_1$, we can obtain the following local parametrix near $-\eta_l$ under the change of variable $k\rightarrow\zeta_l$:
\begin{align}
    & \tilde M^{(l)}(x,t;k)=Y^{(l)}(k)\left[\sigma_1M^{(\textup{PC})}(\zeta_l(k);r(\eta_l))\sigma_1\right]D_0^{\sigma_3}e^{-itg_\textup{\uppercase\expandafter{\romannumeral1}}(\eta_l)\sigma_3},\label{tilde M(l)}\\
    & Y^{(l)}(k)=M^{(\infty)}(k)e^{itg_\textup{\uppercase\expandafter{\romannumeral1}}(\eta_l)\sigma_3}D_0^{-\sigma_3}. \label{equ: Yl(k)}
\end{align}

Before we make the final transformation, we insert the following lemma, which describes the matching conditions between the local parametrices $\tilde M^{(j)}$ and the global parametrix $M^{(\infty)}$, and the approximation of the jump matrices.
\begin{lemma}\label{Lemma v3-vr}
    For $\xi\in\mathcal{R}_\textup{\uppercase\expandafter{\romannumeral1}}$ and $t>0$, the function $\tilde M^{(j)}$ has the following properties for $j\in\{l,r\}$:
    \begin{itemize}
        \item [\rm (a)] As $t\to\infty$, 
        \begin{equation*}
                \|\tilde M^{(j)}(k)M^\infty(k)^{-1}-I\|_{L^\infty(\partial D_\varrho(\pm\eta_l))}=\mathcal{O}(t^{-1}).
        \end{equation*}
        Here, the symbol ``$\pm$'' means ``$+$'' for $j=r$ and ``$-$'' for $j=l$.
        \item [\rm (b)] Across $\Gamma^{(j)}$ defined in \eqref{def:Gamma^(ell) RI}, the jump matrix $\tilde V^{(j)}$ of $\tilde M^{(j)}(x,t;k)$ 
        satisfies the following estimates:
        \begin{align*}
            & \|V^{(3)}-\tilde V^{(j)}\|_{L^1(\Gamma^{(j)})}=\mathcal{O}(t^{-1}\log t),\\
            & \|V^{(3)}-\tilde V^{(j)}\|_{L^2(\Gamma^{(j)})}=\mathcal{O}(t^{-\frac{3}{4}}\log t),\\
            & \|V^{(3)}-\tilde V^{(j)}\|_{L^\infty(\Gamma^{(j)})}=\mathcal{O}(t^{-\frac{1}{2}}\log t).
            %&\|V^{(3)}-\tilde V^{(j)}\|_{(L^1\cap L^2\cap L^\infty)(\mathbb{R}\cap D_\varrho(\pm \eta_l))}=\mathcal{O}(t^{-1}).
        \end{align*}
    \end{itemize}
\end{lemma}
\begin{proof}
    \hfill
    \begin{itemize}
        \item[(a)] Using the definition of $\tilde M^{(j)}(k)$ in \eqref{tilde M(r)}, \eqref{tilde M(l)}, 
        and the asymptotics of $M^{(\textup{PC})}$ in \eqref{MPC}, we have the desired estimate. 
        \item[(b)] Taking use the item (c) of Proposition \ref{analytic extension of r RI}, item (c) of Proposition \ref{analytic extension for r_2 RI} 
        as well as \eqref{estimate D1}, then the desired estimates could be obtained in a very standard way. 
        We refer to \cite[Lemma 6.3]{Arruda} for a detailed proof.
    \end{itemize}
   
\end{proof}
%\begin{remark}
 %   The existence of the factor $\Delta_l(k)$ in \eqref{tilde M(l)} and \eqref{tilde M(r)} makes sure that $\tilde M^{(j)}(k)M^\infty(k)^{-1}$ decays to $I$ uniformly for $k\in\partial D_{\varrho}(\pm\eta_l)$ as $t\to+\infty$. Here, $j\in\{l,r\}$.
%\end{remark}
\subsection{Small norm RH problem for $M^{(err)}$}
Define
\begin{equation}\label{def:Merr RI}
    M^{(err)}(x,t;k):=
        \begin{cases}
        M^{(3)}(x,t;k)\left(M^{(\infty)}\left(x,t;k\right)\right)^{-1}, & k\in \mathbb{C}\setminus\left(D_\varrho(\eta_l)\cup D_\varrho(-\eta_l)\right),\\
         M^{(3)}(x,t;k)\left(\tilde M^{(r)}\left(x,t;k\right)\right)^{-1}, & k\in D_\varrho(\eta_l),\\
         M^{(3)}(x,t;k)\left(\tilde M^{(l)}\left(x,t;k\right)\right)^{-1}, & k\in D_\varrho(-\eta_l).
\end{cases}
\end{equation}

Then $M^{(err)}$ satisfies the following RH problem.
\begin{RHP}\label{RHP: Merr RI}
\hfill
\begin{itemize}
    \item $M^{(err)}(k)$ is holomorphic for $k\in\mathbb{C}\setminus\Gamma^{(err)}$, where
    \begin{align*}
    \Gamma^{(err)}:=\Gamma^{(3)}\cup\partial D_\varrho(\eta_l)\cup\partial D_\varrho(-\eta_l);
    \end{align*}
    see Figure \ref{fig:jump contour Merr RI} for an illustration.
    \item $M^{(err)}(k)$ has continuous boundary values $M^{(err)}_\pm(k)$ on $k\in \Gamma^{(err)}$ with the jump condition
    \begin{equation*}
        M^{(err)}_{+}(k)=M^{(err)}_{-}(k)V^{(err)}(k),
    \end{equation*}
    where
     \begin{align}\label{equ:jump Verr RI}
        V^{(err)}(k)=
                \begin{cases}
                 M^{(\infty)}_-(k)V^{(3)}(k){M^{(\infty)}_+(k)}^{-1}, & k\in\Gamma^{(3)}\backslash \left(\overline{D_\varrho(\eta_l)}\cup \overline{D_\varrho(-\eta_l)}\right),\\
                 \tilde M^{(r)}(k){M^{(\infty)}(k)}^{-1}, & k\in\partial D_\varrho(\eta_l),\\
                \tilde M^{(l)}(k){M^{(\infty)}(k)}^{-1}, & k\in\partial D_\varrho(-\eta_l),\\
                \tilde M^{(r)}_-(k)V^{(3)}(k){\tilde M^{(r)}_+(k)}^{-1}, & k\in\Gamma^{(3)}\cap D_\varrho(\eta_l),\\
                \tilde M^{(l)}_-(k)V^{(3)}(k){\tilde M^{(l)}_+(k)}^{-1}, & k\in\Gamma^{(3)}\cap D_\varrho(-\eta_l).
                \end{cases}    
    \end{align}
    \item As $k\rightarrow\infty$ in $k\in\mathbb{C}\setminus\Gamma^{(err)}$, we have $M^{(err)}(k)=I+\mathcal{O}(k^{-1})$.
    \item As $k\rightarrow\pm c_l$, we have $M^{(err)}(k)=\mathcal{O}(1)$.
\end{itemize}
\end{RHP}

\begin{figure}[htbp]
\begin{center}
\tikzset{every picture/.style={line width=0.75pt}}
\begin{tikzpicture}[x=0.75pt,y=0.75pt,yscale=-1,xscale=1]

% Parameters
\def\yax{144}
\def\R{25}
\def\h{50}
\def\L{100}
\def\V{65}

% Centers and symmetry axis
\coordinate (OL) at (254,\yax);
\coordinate (OR) at (410,\yax);
\def\xmid{332}
\coordinate (U)  at (\xmid, \yax-\h);
\coordinate (D)  at (\xmid, \yax+\h);

% Far endpoints (symmetric)
\coordinate (Aul) at ($(OL)+(-\L,-\V)$);  % (154,79)
\coordinate (Adl) at ($(OL)+(-\L, \V)$);  % (154,209)
\coordinate (Aur) at ($(OR)+(\L,-\V)$);   % (510,79)
\coordinate (Adr) at ($(OR)+(\L, \V)$);   % (510,209)

% Real axis and right arrow
\draw (93,\yax) -- (569,\yax);
% 主轴箭头向右
\draw [shift={(569,\yax)}, rotate = 180]
  (10.93,-3.29) .. controls (6.95,-1.4) and (3.31,-0.3) .. (0,0)
  .. controls (3.31,0.3) and (6.95,1.4) .. (10.93,3.29);

% Middle short segment with arrow and ticks
\draw (290,\yax) -- (372,143.5);
\fill (290,\yax) circle (1.2pt);
\fill (312,\yax) circle (1.2pt);
\fill (372,143.5) circle (1.2pt);
\fill (350,143.5) circle (1.2pt);
% 中段小箭头向右
\draw [shift={(337,143.7)}, rotate =18 0]
  (10.93,-3.29) .. controls (6.95,-1.4) and (3.31,-0.3) .. (0,0)
  .. controls (3.31,0.3) and (6.95,1.4) .. (10.93,3.29);

\draw (574,136.4) node [anchor=north west] {\normalsize $\re k$};
\draw (282,148.4) node [anchor=north west] [font=\tiny] {$-c_{l}$};
\draw (306,148.4) node [anchor=north west] [font=\tiny] {$-c_{r}$};
\draw (346,148.4) node [anchor=north west] [font=\tiny] {$c_{r}$};
\draw (368,148.4) node [anchor=north west] [font=\tiny] {$c_{l}$};

% Two polylines
\draw (Aul) -- (OL) -- (U) -- (OR) -- (Aur);
\draw (Adl) -- (OL) -- (D) -- (OR) -- (Adr);

% Circles centered at intersections
\draw (OL) circle (\R);
\draw (OR) circle (\R);

% Centers and labels
\fill (OL) circle (1.2pt);
\draw (245,148.4) node [anchor=north west] [font=\tiny] {$-\eta_l$};
\fill (OR) circle (1.2pt);
\draw (406,148.4) node [anchor=north west] [font=\tiny] {$\eta_l$};

% ----------------- Arrows on segments (方向对齐为“外臂朝向圆心；中段朝右”) -----------------
% 上折线 Aul->OL：箭头沿 Aul->OL（指向圆心）
\draw [shift={(204,111.5)}, rotate = 213.0]
  (10.93,-3.29) .. controls (6.95,-1.4) and (3.31,-0.3) .. (0,0)
  .. controls (3.31,0.3) and (6.95,1.4) .. (10.93,3.29);
% 上折线 OL->U：箭头沿 OL->U（向右上）
\draw [shift={(293,119)}, rotate = 147.2]
  (10.93,-3.29) .. controls (6.95,-1.4) and (3.31,-0.3) .. (0,0)
  .. controls (3.31,0.3) and (6.95,1.4) .. (10.93,3.29);
% 上折线 U->OR：箭头沿 U->OR（向右下）
\draw [shift={(371,119)}, rotate = 212.8]
  (10.93,-3.29) .. controls (6.95,-1.4) and (3.31,-0.3) .. (0,0)
  .. controls (3.31,0.3) and (6.95,1.4) .. (10.93,3.29);
% 上折线 OR->Aur：箭头沿 OR->Aur（远离圆心，向右上）
\draw [shift={(460,111.5)}, rotate = 147]
  (10.93,-3.29) .. controls (6.95,-1.4) and (3.31,-0.3) .. (0,0)
  .. controls (3.31,0.3) and (6.95,1.4) .. (10.93,3.29);

% 下折线 Adl->OL：箭头沿 Adl->OL（指向圆心）
\draw [shift={(204,176.5)}, rotate = 147]
  (10.93,-3.29) .. controls (6.95,-1.4) and (3.31,-0.3) .. (0,0)
  .. controls (3.31,0.3) and (6.95,1.4) .. (10.93,3.29);
% 下折线 OL->D：箭头沿 OL->D（向右下）
\draw [shift={(293,169)}, rotate = 212.8]
  (10.93,-3.29) .. controls (6.95,-1.4) and (3.31,-0.3) .. (0,0)
  .. controls (3.31,0.3) and (6.95,1.4) .. (10.93,3.29);
% 下折线 D->OR：箭头沿 D->OR（向右上）
\draw [shift={(371,169)}, rotate = 147.2]
  (10.93,-3.29) .. controls (6.95,-1.4) and (3.31,-0.3) .. (0,0)
  .. controls (3.31,0.3) and (6.95,1.4) .. (10.93,3.29);
% 下折线 OR->Adr：箭头沿 OR->Adr（远离圆心，向右下）
\draw [shift={(460,176.5)}, rotate = 213.0]
  (10.93,-3.29) .. controls (6.95,-1.4) and (3.31,-0.3) .. (0,0)
  .. controls (3.31,0.3) and (6.95,1.4) .. (10.93,3.29);
% ----------------------------------------------------------------

% ----------------- Circle arrows (保持顺时针) -----------------
% 左圆：顺时针切向
\draw [shift={(254,119)}, rotate = 170]
  (10.93,-3.29) .. controls (6.95,-1.4) and (3.31,-0.3) .. (0,0)
  .. controls (3.31,0.3) and (6.95,1.4) .. (10.93,3.29);

% 右圆：顺时针切向
\draw [shift={(410,119)}, rotate = 170]
  (10.93,-3.29) .. controls (6.95,-1.4) and (3.31,-0.3) .. (0,0)
  .. controls (3.31,0.3) and (6.95,1.4) .. (10.93,3.29);

% ----------------------------------------------------------------

\end{tikzpicture}
\caption{The jump contours $\Gamma^{(err)}$ of RH problem $M^{(err)}$ for $\xi\in\mathcal{R}_\textup{\uppercase\expandafter{\romannumeral1}}$.}\label{fig:jump contour Merr RI}
\end{center}
\end{figure}
Following from item (c) of Proposition \ref{analytic extension of r RI}, items (c), (d), (f) of Proposition \ref{analytic extension for r_2 RI} 
as well as Lemma \ref{Lemma v3-vr}, we obtain for $p=1,2,\infty$,
\begin{align}\label{est:Verr-I RI}
    \Vert V^{(err)}(k)-I \Vert_{L^p}=
        \begin{cases}
         \mathcal{O}(e^{-ct}), & k\in\Gamma^{(err)}\setminus\left(\overline{D_\varrho(\eta_l)}\cup \overline{D_\varrho(-\eta_l)}\cup\mathbb{R}\right),\\
         \mathcal{O}(t^{-1/2}), &k\in\partial D_\varrho(\eta_l)\cup \partial D_\varrho(-\eta_l),\\
         \mathcal{O}(t^{-1}), &k\in\mathbb{R},\\
         \mathcal{O}(t^{\omega_p}\log t), &k\in \Gamma^{(err)}\cap \left(D_\varrho(\eta_l)\cup D_\varrho(-\eta_l)\right).
        \end{cases}
\end{align}
where $c>0$, $\omega_1=-1$,  $\omega_2=-3/4$ and   $\omega_\infty=-1/2$.

It then follows from the small norm RH problem theory \cite{DZ02} that there exists a unique solution to RH problem \ref{RHP: Merr RI}
for large positive $t$. Moreover, according to \cite{BCdecom}, the solution could be given by
\begin{equation}\label{equ:BCsolforError}
    M^{(err)}(k)=I+\frac{1}{2\pi i}\int_{\Gamma^{(err)}}\frac{\mu(s)(V^{(err)}(s)-I)}{s-k}\dif s,
\end{equation}
where $\mu\in I+L^{2}(\Gamma^{(err)})$ is the unique solution to the Fredholm type equation
\begin{align}\label{equ:fredholm ope equ}
    \mu=I+C_{(err)}\mu.
\end{align}
Here, $C_{(err)}:L^2(\Gamma^{(err)})\rightarrow L^2(\Gamma^{(err)})$ is an integral operator
defined by $C_{(err)}f(z)=C_{-}(f(V^{(err)}(z)-I))$ with $C_{-}$ being the Cauchy projection
operator on $\Gamma^{(err)}$. Therefore we have
\begin{equation}\label{est:mu-I RI}
    \Vert C_{(err)}\Vert \leqslant \Vert C_{-} \Vert_{L^{2}(\Gamma^{(err)})\rightarrow L^{2}(\Gamma^{(err)})}\Vert V^{(err)}-I \Vert_{L^\infty(\Gamma^{(err)})}
    = \mathcal{O}(t^{-1/2}\log t),
\end{equation}
which implies that $I-C_{(err)}$ is invertible for sufficiently large $t$. 
Thus, the solution $\mu$ to \eqref{equ:fredholm ope equ} exists uniquely, and
\begin{align}
    \Vert\mu-I \Vert_{L^2(\Gamma^{(err)})}=\mathcal{O}(t^{-1/2}).
\end{align}

For later use, we conclude this section with behavior of $M^{(err)}(k)$ at $k=\infty$.
By \eqref{equ:BCsolforError}, it follows that
\begin{equation*}
    M^{(err)}(k)=I+\frac{M^{(err)}_1}{k}+\mathcal{O}(k^{-2}), \quad k\rightarrow\infty,
\end{equation*}
where
\begin{equation}\label{equ:Merr_1 RI}
    M^{(err)}_1=-\frac{1}{2\pi i}\int_{\Gamma^{(err)}}\mu(s)(V^{(err)}(s)-I)\dif s.
\end{equation}
\begin{proposition}\label{prop: result of Merr_1 RI}
    With $M^{(err)}_1$ defined in \eqref{equ:Merr_1 RI}, we have, as $t\rightarrow+\infty$,
    \begin{align}\label{equ: Merr_1 RI ttoinfty}
        M^{(err)}_1&=\frac{t^{-1/2}}{\sqrt{2g_{\textup{\uppercase\expandafter{\romannumeral1}}}''(\eta_{l})}}\left[Y^{(r)}(\eta_{l})A^{(\eta_l)}Y^{(r)}(\eta_{l})^{-1}
        -Y^{(l)}(-\eta_{l})A^{(-\eta_l)}Y^{(l)}(-\eta_{l})^{-1}\right]+\mathcal{O}(t^{-1}\log t),
    \end{align}
    where 
    \begin{equation*}
        A^{(\eta_l)}=\begin{pmatrix}
       0 & -i{\beta}^{(\eta_l)}\\ i\overline{{\beta}^{(\eta_l)}} & 0\end{pmatrix},\quad A^{(-\eta_l)}=\begin{pmatrix}
       0 &i\overline{{\beta}^{(\eta_l)}}\\ -i{\beta}^{(\eta_l)} & 0\end{pmatrix}.
    \end{equation*}
 Moreover, the $(1, 2)$-entry of matrix $M^{(err)}_1$ can be represented as follows
    \begin{align}\label{equ:error func 12 RI}
     \left(M^{(err)}_1\right)_{12}=-it^{-\frac{1}{2}}\sqrt{\frac{2}{g_\textup{\uppercase\expandafter{\romannumeral1}}''(\eta_l)}}\mathrm{Re}\left(\beta^{(\eta_l)}e^{-2itg_\textup{\uppercase\expandafter{\romannumeral1}}(\eta_l)}D_0^2\right),
    \end{align}
    where $D_0$, $\beta^{(\eta_l)}$ and $g_\textup{\uppercase\expandafter{\romannumeral1}}$ are defined as in \eqref{D_0 RI}, \eqref{beta} and \eqref{equ:g function RI} respectively.
\end{proposition}
\begin{proof}
    Let us divide $M^{(err)}_1$ into five parts by
    \begin{align*}
     &I_1:=-\frac{1}{2\pi i}\int_{\Gamma^{(err)}}\left(\mu(s)-I\right)\left(V^{(err)}(s)-I\right)\dif s, \\
     &I_2:=-\frac{1}{2\pi i}\int_{\Gamma^{(err)}\setminus\left(\overline{D_\varrho(\eta_l)}\cup \overline{D_\varrho(-\eta_l)}\cup\mathbb{R}\right)}\left(V^{(err)}(s)-I\right)\dif s, \\
     &I_3:=-\frac{1}{2\pi i}\int_{\mathbb{R}}\left(V^{(err)}(s)-I\right)\dif s,\\
     &I_4:=-\frac{1}{2\pi i}\int_{\Gamma^{(err)}\cap \left(D_\varrho(\eta_l)\cup D_\varrho(-\eta_l)\right)}\left(V^{(err)}(s)-I\right)\dif s,\\
     &I_5:=-\frac{1}{2\pi i}\oint_{\partial D_\varrho(\eta_l)\cup \partial D_\varrho(-\eta_l)}\left(V^{(err)}(s)-I\right)\dif s.
    \end{align*}
    It follows from \eqref{est:Verr-I RI} and \eqref{est:mu-I RI} that $I_1=\mathcal{O}(t^{-1}\log t)$, $I_2=\mathcal{O}(e^{-ct})$, $I_3=\mathcal{O}(t^{-1})$ and $I_4=\mathcal{O}(t^{-1}\log t)$
    with some positive constant $c$. Following from \eqref{equ:sol of Minfty}, \eqref{MPC} and \eqref{equ:jump Verr RI},
    it obtains that
    \begin{align}
        I_{5}&=-\frac{1}{2\pi i}\oint_{\partial D_\varrho(\eta_l)}\frac{t^{-1/2}}{\sqrt{2g_{\textup{\uppercase\expandafter{\romannumeral1}}}''(\eta_{l})}(s-\eta_{l})}Y^{(r)}(s)A^{(\eta_l)}Y^{(r)}(s)^{-1}\dif s\nonumber\\
        &+\frac{1}{2\pi i}\oint_{\partial D_\varrho(-\eta_l)}\frac{1}{\sqrt{2g_{\textup{\uppercase\expandafter{\romannumeral1}}}''(\eta_{l})}(s+\eta_{l})}Y^{(l)}(s)A^{(-\eta_l)}Y^{(l)}(s)^{-1}(s)\dif s+\mathcal{O}(t^{-1}) \nonumber\\
        &=\frac{t^{-1/2}}{\sqrt{2g_{\textup{\uppercase\expandafter{\romannumeral1}}}''(\eta_{l})}}\left[Y^{(r)}(\eta_{l})A^{(\eta_l)}Y^{(r)}(\eta_{l})^{-1}
        -Y^{(l)}(-\eta_{l})A^{(-\eta_l)}Y^{(l)}(-\eta_{l})^{-1}\right]+\mathcal{O}(t^{-1}),\nonumber
    \end{align}
    where the last equality follows from the residue theorem. 
    Recalling the definitions of $Y^{(r)}$ in \eqref{equ: Yr(k)} and $Y^{(l)}$ in \eqref{equ: Yl(k)}, together with $M^{(\infty)}(-\eta_l)=M^{(\infty)}(\eta_{l})^{-1}$ and \eqref{g_1'}, we arrive at \eqref{equ:error func 12 RI}.
\end{proof}

\subsection{Proof of the part (\textup{\uppercase\expandafter{\romannumeral1}}) of Theorem \ref{thm:mainthm}}
By tracing back the transformations \eqref{def:M1 RI}, \eqref{def:M2 RI}, \eqref{def:M3 RI} and \eqref{def:Merr RI},
we conclude, for $k\in\mathbb{C}\setminus\Gamma^{(3)}$,
\begin{equation}  M(k)=M^{(err)}M^{(\infty)}D_{\textup{\uppercase\expandafter{\romannumeral1}}}^{\sigma_3}(k)e^{-it(g_{\textup{\uppercase\expandafter{\romannumeral1}}}-\theta)\sigma_3},
\end{equation}
where $M^{(err)}$ and $M^{(\infty)}$ are defined in
\eqref{def:Merr RI} and \eqref{equ:sol of Minfty} respectively.
From the reconstruction formula stated in \eqref{equ:recovering formula}, we obtain that
\begin{equation*}
q(x,t)=2i\left[(M^{(err)}_1)_{12}+\lim_{k\rightarrow\infty}k\left(\Delta_{l}(k)\right)_{12}\right].
\end{equation*}
It now follows from \eqref{equ:Delta_j} and \eqref{equ:error func 12 RI} that part (\textup{\uppercase\expandafter{\romannumeral1}}) of Theorem \ref{thm:mainthm} holds.

\section{Asymptotic analysis of the RH problem for $M$ in $\mathcal{R}_{\textup{\uppercase\expandafter{\romannumeral2}}}$}\label{sec:asymptotic analysis in RII}
%At the beginning of this section, it is indicated that we adopt the same notations (such as $M^{(1)}$, $M^{(2)}$, $M^{(err)}$) as those used in the previous section, and we believe this will not lead to any confusion.

\subsection{First transformation: $M\rightarrow M^{(1)}$}
\subsubsection{The $g$-function}
For $\xi\in\mathcal{R}_{\textup{\uppercase\expandafter{\romannumeral2}}}$, we introduce
\begin{equation}\label{equ:g function RII}
    g_{\textup{\uppercase\expandafter{\romannumeral2}}}(\xi;k):=4X^3_\eta(k), \quad X_{\eta}(\xi;k):=\sqrt{k^2-\eta^2(\xi)},
\end{equation}
with $\eta(\xi):=\sqrt{-2\xi}\in(c_{r}, c_l)$. The branch of the square root is chosen such that $X_\eta(k)=k+\mathcal{O}(k^{-1})$ as $k\rightarrow\infty$.

It's readily verified that the following properties for $g_{\textup{\uppercase\expandafter{\romannumeral2}}}$ defined in \eqref{equ:g function RII} hold true.
\begin{proposition}
    The function $g_{\textup{\uppercase\expandafter{\romannumeral2}}}$ defined in \eqref{equ:g function RII} satisfies the following properties:
\begin{itemize}
    \item $g_{\textup{\uppercase\expandafter{\romannumeral2}}}(\xi;k)$ is holomorphic for $k\in\mathbb{C}\backslash[-\eta(\xi),\eta(\xi)]$.
    \item For $k\in(-\eta,\eta)$, $g_{\textup{\uppercase\expandafter{\romannumeral2}},+}(\xi;k)+g_{\textup{\uppercase\expandafter{\romannumeral2}},-}(\xi;k)=0$.
    \item As $k\rightarrow\infty$ in $\mathbb{C}\backslash[-\eta(\xi),\eta(\xi)]$, we have $g_{\textup{\uppercase\expandafter{\romannumeral2}}}(\xi;k)=\theta(\xi;k)+\mathcal{O}(k^{-1})$.
    \item As $k\to\pm\eta$, 
    \begin{equation}
        g_{\textup{\uppercase\expandafter{\romannumeral2}}}(\xi;k)=4( 2\eta)^{3/2}(\pm k-\eta)^{3/2}\left[1\pm\frac{3(k\mp\eta)}{4\eta}+\mathcal{O}((k\mp\eta)^{2})\right].
    \end{equation}
\end{itemize}
\end{proposition}
The $k$-derivative of function $g_{\textup{\uppercase\expandafter{\romannumeral2}}}$ can be expressed by
\begin{equation*}   g_{\textup{\uppercase\expandafter{\romannumeral2}}}'(k)=12k\sqrt{k^2-\eta^2}.
\end{equation*}
The signature table of $\im g_{\textup{\uppercase\expandafter{\romannumeral2}}}$ is illustrated in Figure \ref{fig:signs img RII}.
\begin{figure}[htbp]
\begin{center}
\tikzset{every picture/.style={line width=0.75pt}} %set default line width to 0.75pt
\begin{tikzpicture}[x=0.75pt,y=0.75pt,yscale=-1,xscale=1]
%uncomment if require: \path (0,300); %set diagram left start at 0, and has height of 300
%Curve Lines [id:da23990346379567207]
\draw    (415,48) .. controls (389,48) and (388,229) .. (421,229) ;
%Straight Lines [id:da27525634770111207]
\draw    (161,139) -- (477,139) ;
%Curve Lines [id:da4970333595425316]
\draw    (205,48) .. controls (230,49) and (230,229) .. (202,230) ;
% Text Node
\draw (205,142.4) node [anchor=north west][inner sep=0.75pt]  [font=\tiny]  {$-\eta $};
\fill (223,139) circle (1.2pt);
% Text Node
\draw (399,141.4) node [anchor=north west][inner sep=0.75pt]  [font=\tiny]  {$\eta $};
\fill (396,139) circle (1.2pt);
% Text Node
\draw (171,145.4) node [anchor=north west][inner sep=0.75pt]  [font=\tiny]  {$-c_{l}$};
% Text Node
\draw (424,143.4) node [anchor=north west][inner sep=0.75pt]  [font=\tiny]  {$c_{l}$};
% Text Node
\draw (257.18,144.13) node [anchor=north west][inner sep=0.75pt]  [font=\tiny,rotate=-1.56]  {$-c_{r}$};
% Text Node
\draw (348,143.4) node [anchor=north west][inner sep=0.75pt]  [font=\tiny]  {$c_{r}$};
% Text Node
\draw (443,89.4) node [anchor=north west][inner sep=0.75pt]    {$+$};
% Text Node
\draw (448,170.4) node [anchor=north west][inner sep=0.75pt]    {$-$};
% Text Node
\draw (296,77.4) node [anchor=north west][inner sep=0.75pt]    {$-$};
% Text Node
\draw (301,189.4) node [anchor=north west][inner sep=0.75pt]    {$+$};
% Text Node
\draw (172,172.4) node [anchor=north west][inner sep=0.75pt]    {$-$};
% Text Node
\draw (175,90.4) node [anchor=north west][inner sep=0.75pt]    {$+$};
% Text Node

\fill (179,139) circle (1.2pt);
% Text Node

\fill (265,139) circle (1.2pt);
% Text Node

\fill (427,139) circle (1.2pt);
% Text Node
\fill (351,139) circle (1.2pt);
\end{tikzpicture}
\caption{ Signature table of  $\im g_{\textup{\uppercase\expandafter{\romannumeral2}}}(\xi;k)$ for $\xi\in\mathcal{R}_{\textup{\uppercase\expandafter{\romannumeral2}}}$.}\label{fig:signs img RII}
\end{center}
\end{figure}

\subsubsection{RH problem for $M^{(1)}$}
With the help of $g_{\textup{\uppercase\expandafter{\romannumeral2}}}$ defined in \eqref{equ:g function RII}, we introduce a
new matrix-valued function $M^{(1)}$ by
\begin{equation}\label{def:M1 RII}
    M^{(1)}(x,t;k)=M(x,t;k)e^{it\left(g_{\textup{\uppercase\expandafter{\romannumeral2}}}(\xi;k)-\theta(\xi;k)\right)\sigma_3}.
\end{equation}
Then the RH problem for $M^{(1)}$ reads as follows:
\begin{RHP}\label{RHP:M1 RII}
    \hfill
\begin{itemize}
    \item $M^{(1)}(k)$ is holomorphic for $k\in\mathbb{C}\backslash\mathbb{R}$.
    \item $M^{(1)}(k)$ has continuous boundary values $M_{\pm}^{(1)}(k)$ on $\mathbb{R}$ with the jump condition
    \begin{equation*}
        M^{(1)}_{+}(k)=M^{(1)}_{-}(k)V^{(1)}(k),
    \end{equation*}
    where
    \begin{equation*}
        V^{(1)}(k)=
            \begin{cases}
            \begin{pmatrix}1-rr^* & -r^*e^{-2itg_{\textup{\uppercase\expandafter{\romannumeral2}}}}\\ re^{2itg_{\textup{\uppercase\expandafter{\romannumeral2}}}} & 1\end{pmatrix}, & k\in(-\infty,-c_l)\cup(c_l,+\infty),  \\
            \begin{pmatrix}0 & -r_{-}^*e^{-2itg_{\textup{\uppercase\expandafter{\romannumeral2}}}}\\ r_{+}e^{2itg_{\textup{\uppercase\expandafter{\romannumeral2}}}} & 1 \end{pmatrix}, & k\in (-c_l,-\eta)\cup(\eta,c_l),\\
            \begin{pmatrix}0 & -r_{-}^{*} \\ r_{+} & e^{-2itg_{\textup{\uppercase\expandafter{\romannumeral2}},+}}\end{pmatrix}, & k\in (-\eta,-c_{r})\cup(c_{r},\eta),  \\
            \begin{pmatrix}0 & -1 \\ 1 & 0\end{pmatrix},  & k\in (-c_{r}, c_{r}).
            \end{cases} 
    \end{equation*}
    \item As $k\rightarrow\infty$ in $\mathbb{C}\backslash\mathbb{R}$, $M^{(1)}(k)=I+\mathcal{O}(k^{-1})$.
    \item $M^{(1)}(k)$ admits the same singular behavior as $M(k)$ at branch points.
\end{itemize}
\end{RHP}

\subsection{Second transformation: $M^{(1)}\rightarrow M^{(2)}$}
\subsubsection{The $D_{\textup{\uppercase\expandafter{\romannumeral2}}}$ function}
Similar to \eqref{equ:def D function RI}, we define an auxiliary function $D_{\textup{\uppercase\expandafter{\romannumeral2}}}(\xi;k):\mathcal{R}_{\textup{\uppercase\expandafter{\romannumeral2}}}\times\mathbb{C}\setminus[-\eta,\eta]\to\mathbb{C}$ by
\begin{equation}\label{equ:def D function RII}
D_{\textup{\uppercase\expandafter{\romannumeral2}}}(\xi;k):=\exp\left\{\frac{X_{\eta}(k)}{2\pi i}\left(\int_{-\eta}^{-c_{r}}+\int_{c_{r}}^{\eta}\right)\frac{\log r_{+}(s)}{X_{\eta+}(s)(s-k)}\dif s\right\}.
\end{equation}
The following proposition presents the necessary properties for $D_{\textup{\uppercase\expandafter{\romannumeral2}}}$.
\begin{proposition}\label{prop: D function RII}
    The function $D_{\textup{\uppercase\expandafter{\romannumeral2}}}$ defined in \eqref{equ:def D function RII} satisfies the following properties for $\xi\in \mathcal{R}_{\textup{\uppercase\expandafter{\romannumeral2}}}$:
\begin{itemize}
    \item[\rm (a)]$D_{\textup{\uppercase\expandafter{\romannumeral2}}}(k)$ is holomorphic for $k\in\mathbb{C}\setminus[-\eta,\eta]$.
    \item[\rm (b)]$D_{\textup{\uppercase\expandafter{\romannumeral2}}}(k)$ satisfies the jump relations:
    \begin{equation*}
        \begin{aligned}
       & D_{\textup{\uppercase\expandafter{\romannumeral2}},+}(k)D_{\textup{\uppercase\expandafter{\romannumeral2}},-}(k)=r_{+}(k), && k\in (-\eta,-c_{r})\cup(c_{r},\eta),\\
       & D_{\textup{\uppercase\expandafter{\romannumeral2}},+}(k)D_{\textup{\uppercase\expandafter{\romannumeral2}},-}(k)=1, && k\in (-c_{r}, c_{r}),\\
       & D_{\textup{\uppercase\expandafter{\romannumeral2}},+}(k)=D_{\textup{\uppercase\expandafter{\romannumeral2}},-}(k), && \textnormal{elsewhere}.
        \end{aligned}
    \end{equation*}
    \item[\rm (c)]As $k\rightarrow \infty$ in $\mathbb{C}\setminus[-\eta,\eta]$, $D_{\textup{\uppercase\expandafter{\romannumeral2}}}(\xi;k)=1+\mathcal{O}(k^{-1})$.
    \item[\rm (d)]$D_{\textup{\uppercase\expandafter{\romannumeral2}}}(k)$ shows the following behavior at endpoints $\pm c_r$
    \begin{equation}\label{Dasyeta}
        D_{\textup{\uppercase\expandafter{\romannumeral2}}}(k)=(\pm k- c_r)^{-\frac{\arg b_{r,0}}{2\pi}}e^{d_{\textup{\uppercase\expandafter{\romannumeral2}},r}}\left(1+\mathcal{O}\left((\pm k-c_r)^{1/2}\right)\right), \quad k\rightarrow\pm c_{r},\; k\in\mathbb{C}^{+} ,
    \end{equation}
    where
    \begin{equation*}
        d_{\textup{\uppercase\expandafter{\romannumeral2}},r}=\frac{\arg b_{r,0}}{2 \pi}\log (c_r-\eta)+\frac{\sqrt{\eta^2-c^2_r}}{2\pi}\int_{-\eta}^{-c_r}\frac{\log r_+(s)}{X_{\eta +}(s)(s+c_r)}\dif s,
    \end{equation*}
    with the principal branch being used for the complex power functions and the logarithms.
  \item[\rm (e)] As $k$ approaching $\pm\eta$ non-tangentially, $D_{\textup{\uppercase\expandafter{\romannumeral2}}}(k)$ has the following asymptotics
  \begin{align}
      &D_{\textup{\uppercase\expandafter{\romannumeral2}}}(k)=\sqrt{r_+(\eta)}\left[1+C_{\eta}\sqrt{k-\eta}+\frac{C_{\eta}^2+\frac{r_+^{\prime}(\eta)}{r_+(\eta)}}{2}(k-\eta)+\mathcal{O}((k-\eta)^{\frac{3}{2}})\right],&& k\to\eta,\\
     & D_{\textup{\uppercase\expandafter{\romannumeral2}}}(k)=\frac{1}{\sqrt{r_+(\eta)}}\left[1-C_{\eta}\sqrt{-k-\eta}+\frac{C_{\eta}^2-\frac{r_+^{\prime}(\eta)}{r_+(\eta)}}{2}(-k-\eta)+\mathcal{O}((-k-\eta)^{\frac{3}{2}})\right],&& k\to-\eta,
  \end{align}
  where
  \begin{equation}\label{C_eta}
      C_{\eta}=-\frac{\sqrt{2\eta}}{2\pi }\left[\int_{c_r}^\eta\frac{\frac{\log r_+(s)}{\sqrt{\eta+s}}-\frac{\log r_+(\eta)}{\sqrt{2\eta}}}{(s-\eta)\sqrt{\eta-s}}\dif s+\frac{2\log r_+(\eta)}{\sqrt{2\eta(\eta-c_r)}}+i\int_{-\eta}^{-c_r}\frac{\log r_+(s)}{X_{\eta_+}(s)(s-\eta)}\dif s\right]\in i\mathbb{R}.
  \end{equation}
\end{itemize}
\end{proposition}
\begin{proof}
    The proofs of properties (a)--(e) are similar to those for $D_{\textup{\uppercase\expandafter{\romannumeral1}}}$ in Proposition \ref{prop: D function RI}, and thus we omit the details here.
\end{proof}

\subsubsection{RH problem for $M^{(2)}$}
Let us define a new matrix-valued function $M^{(2)}$ by
\begin{equation}\label{def:M2 RII}
    M^{(2)}(x,t;k)=M^{(1)}(x,t;k)D_{\textup{\uppercase\expandafter{\romannumeral2}}}^{-\sigma_3}(\xi;k).
\end{equation}
Then RH problem for $M^{(2)}$ reads as follows:
\begin{RHP}
    \hfill
\begin{itemize}
    \item $M^{(2)}(k)$ is holomorphic for $k\in\mathbb{C}\backslash\mathbb{R}$.
    \item $M^{(2)}(k)$ has continuous boundary values $M_{\pm}^{(2)}(k)$ on $\mathbb{R}$ with the jump condition
    \begin{equation*}
        M^{(2)}_{+}(k)=M^{(2)}_{-}(k)V^{(2)}(k),
    \end{equation*}
    where
    \begin{equation*}
        V^{(2)}(k)=
            \begin{cases}
            \begin{pmatrix}1-rr^* & -D_{\textup{\uppercase\expandafter{\romannumeral2}}}^2r^*e^{-2itg_{\textup{\uppercase\expandafter{\romannumeral2}}}}\\ D_{\textup{\uppercase\expandafter{\romannumeral2}}}^{-2}re^{2itg_{\textup{\uppercase\expandafter{\romannumeral2}}}} & 1\end{pmatrix}, &k\in(-\infty,-\eta)\cup(\eta,+\infty),  \\
            \begin{pmatrix}0 & -1\\ 1 & D_{\textup{\uppercase\expandafter{\romannumeral2}},+}D_{\textup{\uppercase\expandafter{\romannumeral2}},-}^{-1}e^{-2itg_{\textup{\uppercase\expandafter{\romannumeral2}},+}} \end{pmatrix}, & k\in (-\eta,-c_{r})\cup(c_{r},\eta),\\
            \begin{pmatrix}0 & -1 \\ 1 & 0\end{pmatrix},  & k\in (-c_{r}, c_{r}).
            \end{cases}
    \end{equation*}
    \item As $k\rightarrow\infty$ in $\mathbb{C}\setminus\mathbb{R}$, we have $M^{(2)}(k)=I+\mathcal{O}(k^{-1})$.
\end{itemize}
\end{RHP}
\subsection{Third transformation: $M^{(2)}\rightarrow M^{(3)}$}
In order to open lenses, we also need to perform the analytic approximation for $r(k)$ on regions $U^{(3)}_1$ and $U^{(3)}_2$; see Figure \ref{fig:U_j domain RII} for an illustration.
\begin{proposition}[Analytic approximation of $r$]\label{analytic extension of r RII}There exist continuous functions
\begin{equation*}
    r_a: \mathcal{R}_{\textup{\uppercase\expandafter{\romannumeral2}}}  \times \left(\overline{U_1^{(3)}}\cup\overline{U_3^{(3)}}\right) \rightarrow \mathbb{C} \; \text { and } \; r_r: \mathcal{R}_{\textup{\uppercase\expandafter{\romannumeral2}}} \times\left((-\infty,-\eta)\cup(\eta,+\infty)\right) \rightarrow \mathbb{C},
\end{equation*}
which satisfy the following properties:
\begin{itemize}
    \item [\rm (a)] $r(k)=r_a(\xi; k)+r_r(\xi; k)$ for all $(\xi; k) \in \mathcal{R}_{\textup{\uppercase\expandafter{\romannumeral2}}}  \times\{(-\infty,-\eta)\cup(\eta,+\infty)\}$.
    \item [\rm (b)] $r_a(-k)=r_a^*(k)$ for $k\in\overline{U_1^{(3)}}\cup\overline{U_3^{(3)}}.$
    \item [\rm (c)]  For all $\xi \in \mathcal{R}_{\textup{\uppercase\expandafter{\romannumeral2}}}$, the function $r_a: U_1^{(3)}\cup U_3^{(3)} \rightarrow \mathbb{C}$ is holomorphic. Moreover, for 
    $(\xi;k)\in\mathcal{R}_{\textup{\uppercase\expandafter{\romannumeral2}}}\times\Big{(}\overline{U_1^{(3)}}\cup\overline{U_3^{(3)}}\Big{)}$, 
\begin{equation*}
    \left|r_a(\xi; k)-r\left(\pm \eta\right)-r^{\prime}\left(\pm \eta\right)\left(k\mp \eta\right)\right| \lesssim \left|k\mp \eta\right|^2 \mathrm{e}^{t|\operatorname{Im} g_{\textup{\uppercase\expandafter{\romannumeral2}}}(\xi;k)|}
\end{equation*}
and
\begin{equation*}
    \left|r_a(\xi; k)\right| \lesssim \frac{\mathrm{e}^{t|\operatorname{Im} g_{\textup{\uppercase\expandafter{\romannumeral2}}}(\xi; k)|}}{1+|k|^2}.
\end{equation*}
\item [\rm (d)] For all $\xi \in \mathcal{R}_{\textup{\uppercase\expandafter{\romannumeral2}}}$, the function $r_r\in L^p\left((-\infty,-\eta)\cup(\eta,+\infty)\right),\; p\in[1,+\infty]$, and as $t \rightarrow +\infty$,
\begin{equation*}
    \left\|r_r(\xi;k)\right\|_{L^p\left((-\infty,-\eta)\cup(\eta,+\infty)\right)} = \mathcal{O}\left(t^{-2}\right).
\end{equation*}
\end{itemize} 
\end{proposition}   
\subsubsection{RH problem for $M^{(3)}$}
Similar to \eqref{def:M3 RI}, we define
\begin{equation} \label{def:M3 RII}
    M^{(3)}(k)=M^{(2)}(k)D_{\textup{\uppercase\expandafter{\romannumeral2}}}^{\sigma_3}(k)G(k)D_{\textup{\uppercase\expandafter{\romannumeral2}}}^{-\sigma_3}(k),
\end{equation}
where
\begin{equation*}
    G(k):=
        \begin{cases}
        \begin{pmatrix}1 & 0 \\ -r_ae^{2itg_{\textup{\uppercase\expandafter{\romannumeral2}}}} & 1\end{pmatrix}, &k\in U^{(3)}_1\cup U^{(3)}_2, \\
        \begin{pmatrix}1 & -r_a^{*}e^{-2itg_{\textup{\uppercase\expandafter{\romannumeral2}}}} \\ 0 & 1\end{pmatrix}, & k\in U_1^{(3)*}\cup U_2^{(3)*}, \\
        I, & \textnormal{elsewhere}.
        \end{cases}
\end{equation*}
Then RH problem for $M^{(3)}$ reads as follows:
\begin{RHP}
\hfill
\begin{itemize}
    \item $M^{(3)}(k)$ is holomorphic for $k\in\mathbb{C}\backslash\Gamma^{(3)}$, where $\Gamma^{(3)}:=\cup_{j=1}^{2}(\Gamma^{(3)}_j\cup\Gamma_j^{(3)*})\cup\mathbb{R}$;
    see Figure \ref{fig:U_j domain RII} for an illustration.
    \item $M^{(3)}(k)$ has continuous boundary values $M_{\pm}^{(3)}(k)$ on $\Gamma^{(3)}$ with the jump condition
    \begin{equation*}
        M^{(3)}_{+}(k)=M^{(3)}_{-}(k)V^{(3)}(k),
    \end{equation*}
    where
    \begin{equation}\label{equ:jump V3 RII}
        V^{(3)}(k)=
            \begin{cases}
            \begin{pmatrix} 1 & 0\\ D_{\textup{\uppercase\expandafter{\romannumeral2}}}^{-2}r_ae^{2itg_{\textup{\uppercase\expandafter{\romannumeral2}}}} & 1\end{pmatrix},  &k\in\Gamma^{(3)}_1\cup\Gamma^{(3)}_2,  \\
            \begin{pmatrix} 1 & -D_{\textup{\uppercase\expandafter{\romannumeral2}}}^{2}r_a^*e^{-2itg_{\textup{\uppercase\expandafter{\romannumeral2}}}}\\ 0 & 1\end{pmatrix},  &k\in\Gamma^{(3)*}_1\cup\Gamma^{(3)*}_2,  \\
            \begin{pmatrix}
                1-r_rr^*_r&-D_{\textup{\uppercase\expandafter{\romannumeral2}}}^2r_r^*e^{-2itg_{\textup{\uppercase\expandafter{\romannumeral2}}}}\\D_{\textup{\uppercase\expandafter{\romannumeral2}}}^{-2}r_re^{2itg_{\textup{\uppercase\expandafter{\romannumeral2}}}}&1
            \end{pmatrix},&k\in(-\infty,-\eta)\cup(\eta,+\infty),\\
            \begin{pmatrix} 0 & -1\\ 1 & D_{\textup{\uppercase\expandafter{\romannumeral2}},+}D_{\textup{\uppercase\expandafter{\romannumeral2}}, -}^{-1}e^{-2itg_{\textup{\uppercase\expandafter{\romannumeral2}},+}}\end{pmatrix}, &k\in (-\eta,-c_{r})\cup(c_{r},\eta),\\
            \begin{pmatrix} 0 & -1 \\ 1 & 0\end{pmatrix}, &k\in (-c_{r}, c_{r}).
            \end{cases}
    \end{equation}
    \item As $k\rightarrow\infty$ in $\mathbb{C}\setminus\Gamma^{(3)}$, we have $M^{(3)}(k)=I+\mathcal{O}(k^{-1})$.
\end{itemize}
\end{RHP}

\begin{figure}[htbp]
\begin{center}
    \tikzset{every picture/.style={line width=0.75pt}} %set default line width to 0.75pt
    \begin{tikzpicture}[x=0.75pt,y=0.75pt,yscale=-1,xscale=1]
    %uncomment if require: \path (0,300); %set diagram left start at 0, and has height of 300
    %Straight Lines [id:da804783244241797]
    \draw    (171,150) -- (481,152) ;
    \draw [shift={(332,151.04)}, rotate = 180.37] [color={rgb, 255:red, 0; green, 0; blue, 0 }  ][line width=0.75]    (10.93,-3.29) .. controls (6.95,-1.4) and (3.31,-0.3) .. (0,0) .. controls (3.31,0.3) and (6.95,1.4) .. (10.93,3.29)   ;
    %Straight Lines [id:da4789393691693813]
    \draw    (481,152) -- (588,222) ;
    \draw [shift={(539.52,190.28)}, rotate = 213.19] [color={rgb, 255:red, 0; green, 0; blue, 0 }  ][line width=0.75]    (10.93,-3.29) .. controls (6.95,-1.4) and (3.31,-0.3) .. (0,0) .. controls (3.31,0.3) and (6.95,1.4) .. (10.93,3.29)   ;
    %Straight Lines [id:da7282944804854083]
    \draw    (582,86) -- (481,152) ;
    \draw [shift={(537.36,115.17)}, rotate = 146.84] [color={rgb, 255:red, 0; green, 0; blue, 0 }  ][line width=0.75]    (10.93,-3.29) .. controls (6.95,-1.4) and (3.31,-0.3) .. (0,0) .. controls (3.31,0.3) and (6.95,1.4) .. (10.93,3.29)   ;
    %Straight Lines [id:da8230642158560675]
    \draw    (171,150) -- (64.26,79.61) ;
    \draw [shift={(123.47,118.66)}, rotate = 213.4] [color={rgb, 255:red, 0; green, 0; blue, 0 }  ][line width=0.75]    (10.93,-3.29) .. controls (6.95,-1.4) and (3.31,-0.3) .. (0,0) .. controls (3.31,0.3) and (6.95,1.4) .. (10.93,3.29)   ;
    %Straight Lines [id:da5987324696981353]
    \draw    (69.76,215.63) -- (171,150) ;
    \draw [shift={(125.41,179.55)}, rotate = 147.05] [color={rgb, 255:red, 0; green, 0; blue, 0 }  ][line width=0.75]    (10.93,-3.29) .. controls (6.95,-1.4) and (3.31,-0.3) .. (0,0) .. controls (3.31,0.3) and (6.95,1.4) .. (10.93,3.29)   ;
    %Straight Lines [id:da6755588884525721]
    \draw   (481,152) -- (604,151) ;
    \draw [shift={(548.5,151.45)}, rotate = 179.53] [color={rgb, 255:red, 0; green, 0; blue, 0 }  ][line width=0.75]    (10.93,-3.29) .. controls (6.95,-1.4) and (3.31,-0.3) .. (0,0) .. controls (3.31,0.3) and (6.95,1.4) .. (10.93,3.29)   ;
    %Straight Lines [id:da3217015238494134]
    \draw    (48,151) -- (171,150) ;
    \draw [shift={(115.5,150.45)}, rotate = 179.53] [color={rgb, 255:red, 0; green, 0; blue, 0 }  ][line width=0.75]    (10.93,-3.29) .. controls (6.95,-1.4) and (3.31,-0.3) .. (0,0) .. controls (3.31,0.3) and (6.95,1.4) .. (10.93,3.29)   ;
    %Curve Lines [id:da4365637499478716]
    \draw  [dash pattern={on 0.84pt off 2.51pt}]  (499,60) .. controls (473,60) and (472,241) .. (505,241) ;
    %Curve Lines [id:da5362793018876433]
    \draw  [dash pattern={on 0.84pt off 2.51pt}]  (152,64) .. controls (177,65) and (177,245) .. (149,246) ;
    % Text Node
    \draw (468,156.4) node [anchor=north west][inner sep=0.75pt]  [font=\tiny]  {$\eta $};
          \fill (480,152) circle (1.2pt);
    % Text Node
    \draw (173,156.4) node [anchor=north west][inner sep=0.75pt]  [font=\tiny]  {$-\eta $};
     \fill (169,150) circle (1.2pt);
    % Text Node
    \draw (228,156.4) node [anchor=north west][inner sep=0.75pt]  [font=\tiny]  {$-c_{r}$};
    % Text Node
    \draw (412,156.4) node [anchor=north west][inner sep=0.75pt]  [font=\tiny]  {$c_{r}$};

    % Text Node
    \draw (541,72.4) node [anchor=north west][inner sep=0.75pt]  [font=\tiny]  {$\Gamma _{1}^{( 3)}$};
    % Text Node
    \draw (542,209.4) node [anchor=north west][inner sep=0.75pt]  [font=\tiny]  {$\Gamma _{1}^{( 3) *}$};
    % Text Node
    \draw (106,78.4) node [anchor=north west][inner sep=0.75pt]  [font=\tiny]  {$\Gamma _{2}^{( 3)}$};
    % Text Node
    \draw (95,204.4) node [anchor=north west][inner sep=0.75pt]  [font=\tiny]  {$\Gamma _{2}^{( 3) *}$};
    % Text Node
    \draw (543,114.4) node [anchor=north west][inner sep=0.75pt]  [font=\tiny]  {$U_{1}^{( 3)}$};
    % Text Node
    \draw (546,162.4) node [anchor=north west][inner sep=0.75pt]  [font=\tiny]  {$U_{1}^{( 3) *}$};
    % Text Node
    \draw (85,112.4) node [anchor=north west][inner sep=0.75pt]  [font=\tiny]  {$U_{2}^{( 3)}$};
    % Text Node
    \draw (85,163.4) node [anchor=north west][inner sep=0.75pt]  [font=\tiny]  {$U_{2}^{( 3) *}$};
    % Text Node
    \fill (234,150.5) circle (1.2pt);
    % Text Node
 
      \fill (409,151.5) circle (1.2pt);
    \end{tikzpicture}
\caption{The jump contours of RH problem for $M^{(3)}$ when $\xi\in\mathcal{R}_{\textup{\uppercase\expandafter{\romannumeral2}}}$.}\label{fig:U_j domain RII}
\end{center}
\end{figure}

\subsection{Analysis of RH problem for $M^{(3)}$}
It's also noted that $V^{(3)}\rightarrow I$ as $t\rightarrow+\infty$ on the contours
$\Gamma^{(3)}_j\cup\Gamma_j^{(3)*}$ for $j=1,2$. Moreover, by item (d) in Proposition \ref{analytic extension of r RII}, 
we have that, as $t\rightarrow +\infty$, $V^{(3)}(k)\to I$ for $k\in(-\infty,-\eta)\cup (\eta, +\infty)$.
Therefore, it follows that $M^{(3)}$ is approximated, to the leading order, by the global parametrix $M^{(\infty)}$ given below. 
% The sub-leading
% contribution stems from the local behavior near the saddle points $\pm\eta$, which is well
% approximated by the Airy parametrix.

\subsubsection{Global Parametrix}
%As $t$ large enough, the jump matrix $V^{(3)}$ approaches
%\begin{equation*}
 %   V^{(\infty)}=\begin{pmatrix}0 & -1 \\ 1 & 0 \end{pmatrix}, \quad k\in(-\eta,\eta).
%\end{equation*}
%For $k\in\mathbb{C}\backslash[-\eta,\eta]$, $V^{(3)}\rightarrow I$ as $t\rightarrow \infty$.
%Following the manners of constructing RH problem \ref{RHP:Minfty RI}, we similarly obtain:
\begin{RHP}\label{RHP:Minfty RII}
    \hfill
\begin{itemize}
    \item $M^{(\infty)}(k)$ is holomorphic for $k\in\mathbb{C}\backslash[-\eta,\eta]$.
    \item $M^{(\infty)}(k)$ has continuous boundary values on $(-\eta,\eta)$ satisfying the following jump condition
    \begin{align*}
        M^{(\infty)}_{+}(k)=M^{(\infty)}_{-}(k)\begin{pmatrix}
            0 & -1 \\
            1 & 0
        \end{pmatrix}.
    \end{align*}
    \item As $k\rightarrow\infty$ in $\mathbb{C}\backslash[-\eta,\eta]$, we have $M^{(\infty)}(k)=I+\mathcal{O}(k^{-1})$.
    \item As $k\rightarrow\pm \eta$, $M^{(\infty)}(k)=\mathcal{O}((k\mp \eta)^{-1/4})$.
\end{itemize}
\end{RHP}
Thus the unique solution to $M^{(\infty)}(k)$ is given by
    \begin{equation}\label{equ:sol of Minfty RII}
        M^{(\infty)}(k)=\frac{1}{2}
        \left(
            \begin{array}{cc}
            \chi_\eta(k)+\chi^{-1}_\eta(k) & i\left(\chi_\eta(k)-\chi_\eta^{-1}(k)\right)\\
            -i\left(\chi_\eta(k)-\chi_{\eta}^{-1}(k)\right) & \chi_\eta(k)+\chi_\eta^{-1}(k)
            \end{array}
        \right),
    \end{equation}
where
\begin{equation*}
\chi_\eta(k)=\left(\frac{k-\eta}{k+\eta}\right)^{\frac{1}{4}}=\left(\frac{k-\sqrt{-2\xi}}{k+\sqrt{-2\xi}}\right)^{\frac{1}{4}}.
\end{equation*}
For later use, it can be obtained that, as $k\rightarrow\eta$,
\begin{equation}\label{equ:Delta_eta asy to eta}
    \begin{aligned}
    M^{(\infty)}(k)&=\frac{(2\eta)^{\frac{1}{4}}}{2(k-\eta)^{1/4}}\left[
        \begin{pmatrix}
            1 & -i\\
            i & 1
        \end{pmatrix}+\frac{(k-\eta)^{\frac{1}{2}}}{(2\eta)^{\frac{1}{2}}}
        \begin{pmatrix}
            1 & i\\
            -i & 1
        \end{pmatrix}+
        \frac{k-\eta}{8\eta}
        \begin{pmatrix}
            1 & -i\\
            i & 1
        \end{pmatrix}\right.\\
        &+\left.\frac{(k-\eta)^{\frac{3}{2}}}{4(2\eta)^{3/2}}
        \begin{pmatrix}
            -1 & -i\\
            i & -1
        \end{pmatrix}
        +\mathcal{O}\left(\left(k-\eta\right)^2\right)
    \right].
\end{aligned}
\end{equation}

\subsubsection{Local parametrices near saddle points} \label{subsubsec: local para pm eta}
Let
\begin{align*}
    D_\varrho(\eta)=\left\{k: |k-\eta|<\varrho \right\}, \quad D_\varrho(-\eta)=\left\{k: |k+\eta|<\varrho \right\}
\end{align*}
be two small disks around $\pm\eta$ respectively, where
\begin{equation*}
\varrho<\frac{1}{3}{\rm min}\left\{ |\eta-c_l|, |\eta+c_l|,   |\eta|  \right\}.
\end{equation*}
For $j\in\{r,l\}$, we intend to solve the following local RH problem for $M^{(j)}$.
\begin{RHP}
\hfill
\begin{itemize}
    \item $M^{(j)}(k)$ is holomorphic for $k\in\overline{D_{\varrho}(\pm \eta)}\setminus\Gamma^{(j)}$ (``$+$'' for $j=r$, and ``$-$'' for $j=l$),
    where
    \begin{align}\label{def:Gamma^(ell) RII}
        \Gamma^{(r)}:=D_\varrho(\eta)\cap \Gamma^{(3)},\quad \Gamma^{(l)}:=D_\varrho(-\eta)\cap \Gamma^{(3)};
    \end{align}
    see Figure \ref{fig:local jumps RII} for an illustration.
    \item $M^{(j)}(k)$ has continuous boundary values $M_{\pm}^{(j)}(k)$ on $k\in \Gamma^{(j)}$ with the jump condition
    \begin{align*}
        M^{(j)}_{+}(k)=M^{(j)}_{-}(k)V^{(3)}(k)\big|_{\Gamma^{(j)}}.
    \end{align*}
    \item As $k\rightarrow\infty$ in $\mathbb{C}\setminus\Gamma^{(j)}$, we have $M^{(j)}(k)=I+\mathcal{O}(k^{-1})$.
\end{itemize}
\end{RHP}
\begin{figure}[htbp]
    \begin{center}
        \tikzset{every picture/.style={line width=0.75pt}} %set default line width to 0.75pt
        \begin{tikzpicture}[x=0.75pt,y=0.75pt,yscale=-1,xscale=1]
        %uncomment if require: \path (0,300); %set diagram left start at 0, and has height of 300
        %Straight Lines [id:da33277675709883603]
        \draw    (379,150) -- (451,150) ;
        \draw [shift={(421,150)}, rotate = 180] [color={rgb, 255:red, 0; green, 0; blue, 0 }  ][line width=0.75]    (10.93,-3.29) .. controls (6.95,-1.4) and (3.31,-0.3) .. (0,0) .. controls (3.31,0.3) and (6.95,1.4) .. (10.93,3.29)   ;
        %Straight Lines [id:da47858322445817936]
        \draw    (451,150) -- (512,112) ;
        \draw [shift={(486.59,127.83)}, rotate = 148.08] [color={rgb, 255:red, 0; green, 0; blue, 0 }  ][line width=0.75]    (10.93,-3.29) .. controls (6.95,-1.4) and (3.31,-0.3) .. (0,0) .. controls (3.31,0.3) and (6.95,1.4) .. (10.93,3.29)   ;
        %Straight Lines [id:da014407024076072972]
        \draw    (451,150) -- (508,195) ;
        \draw [shift={(484.21,176.22)}, rotate = 218.29] [color={rgb, 255:red, 0; green, 0; blue, 0 }  ][line width=0.75]    (10.93,-3.29) .. controls (6.95,-1.4) and (3.31,-0.3) .. (0,0) .. controls (3.31,0.3) and (6.95,1.4) .. (10.93,3.29)   ;
        %Straight Lines [id:da5582157046480072]
        \draw    (451,150) -- (523,150) ;
        \draw [shift={(493,150)}, rotate = 180] [color={rgb, 255:red, 0; green, 0; blue, 0 }  ][line width=0.75]    (10.93,-3.29) .. controls (6.95,-1.4) and (3.31,-0.3) .. (0,0) .. controls (3.31,0.3) and (6.95,1.4) .. (10.93,3.29)   ;
        %Shape: Circle [id:dp13067313586329687]
        \draw  [dash pattern={on 0.84pt off 2.51pt}] (379,150) .. controls (379,110.24) and (411.24,78) .. (451,78) .. controls (490.76,78) and (523,110.24) .. (523,150) .. controls (523,189.76) and (490.76,222) .. (451,222) .. controls (411.24,222) and (379,189.76) .. (379,150) -- cycle ;
        %Straight Lines [id:da014670513874418756]
        \draw    (262,152) -- (190,152) ;
        \draw [shift={(233,152)}, rotate = 180] [color={rgb, 255:red, 0; green, 0; blue, 0 }  ][line width=0.75]    (10.93,-3.29) .. controls (6.95,-1.4) and (3.31,-0.3) .. (0,0) .. controls (3.31,0.3) and (6.95,1.4) .. (10.93,3.29)   ;
        %Straight Lines [id:da962135146927656]
        \draw    (190,152) -- (129,190) ;
        \draw [shift={(165.44,167.3)}, rotate = 148.08] [color={rgb, 255:red, 0; green, 0; blue, 0 }  ][line width=0.75]    (10.93,-3.29) .. controls (6.95,-1.4) and (3.31,-0.3) .. (0,0) .. controls (3.31,0.3) and (6.95,1.4) .. (10.93,3.29)   ;
        %Straight Lines [id:da8888520104896684]
        \draw    (190,152) -- (133,107) ;
        \draw [shift={(166.99,133.84)}, rotate = 218.29] [color={rgb, 255:red, 0; green, 0; blue, 0 }  ][line width=0.75]    (10.93,-3.29) .. controls (6.95,-1.4) and (3.31,-0.3) .. (0,0) .. controls (3.31,0.3) and (6.95,1.4) .. (10.93,3.29)   ;
        %Straight Lines [id:da7190892313729087]
        \draw    (190,152) -- (118,152) ;
        \draw [shift={(161,152)}, rotate = 180] [color={rgb, 255:red, 0; green, 0; blue, 0 }  ][line width=0.75]    (10.93,-3.29) .. controls (6.95,-1.4) and (3.31,-0.3) .. (0,0) .. controls (3.31,0.3) and (6.95,1.4) .. (10.93,3.29)   ;
        %Shape: Circle [id:dp7502061853017183]
        \draw  [dash pattern={on 0.84pt off 2.51pt}] (262,152) .. controls (262,191.76) and (229.76,224) .. (190,224) .. controls (150.24,224) and (118,191.76) .. (118,152) .. controls (118,112.24) and (150.24,80) .. (190,80) .. controls (229.76,80) and (262,112.24) .. (262,152) -- cycle ;
        % Text Node
        \draw (444,153.4) node [anchor=north west][inner sep=0.75pt]  [font=\tiny]  {$\eta $};
        % Text Node
        \draw (474,100.4) node [anchor=north west][inner sep=0.75pt]  [font=\tiny]  {$\Gamma _{1}^{( r)}$};
        % Text Node
        \draw (400,121.4) node [anchor=north west][inner sep=0.75pt]  [font=\tiny]  {$\Gamma _{2}^{( r)}$};
        % Text Node
        \draw (463,179.4) node [anchor=north west][inner sep=0.75pt]  [font=\tiny]  {$\Gamma _{3}^{( r)}$};
        % Text Node
        \draw (525,138.4) node [anchor=north west][inner sep=0.75pt]  [font=\tiny]  {$\Gamma _{4}^{( r)}$};
        % Text Node
        \draw (220,127.4) node [anchor=north west][inner sep=0.75pt]  [font=\tiny]  {$\Gamma _{2}^{( l)}$};
        % Text Node
        \draw (156,94.4) node [anchor=north west][inner sep=0.75pt]  [font=\tiny]  {$\Gamma _{1}^{( l)}$};
        % Text Node
        \draw (95,136.4) node [anchor=north west][inner sep=0.75pt]  [font=\tiny]  {$\Gamma _{4}^{( l)}$};
        % Text Node
        \draw (145,181.4) node [anchor=north west][inner sep=0.75pt]  [font=\tiny]  {$\Gamma _{3}^{( l)}$};
        % Text Node
        \draw (182,156.4) node [anchor=north west][inner sep=0.75pt]  [font=\tiny]  {$-\eta $};
        % Text Node
        \draw (497,63.4) node [anchor=north west][inner sep=0.75pt]  [font=\scriptsize]  {$ D_\varrho(\eta)$};
        % Text Node
        \draw (115,66.4) node [anchor=north west][inner sep=0.75pt]  [font=\scriptsize]  {$ D_\varrho(-\eta)$};
        \end{tikzpicture}
    \caption{The local jump contours of RH problem for $M^{(r)}$ (right) and $M^{(l)}$ (left) for $\xi\in\mathcal{R}_{\textup{\uppercase\expandafter{\romannumeral2}}}$.}\label{fig:local jumps RII}
    \end{center}
    \end{figure}
In the rest part of this section, we focus on the construction of $M^{(r)}$ near $k=\eta$, and the construction of $M^{(l)}$ near $k=-\eta$ can be obtained in a similar way.
By the definition of function $g_{\textup{\uppercase\expandafter{\romannumeral2}}}$ in \eqref{equ:g function RII}, we
define the fractional power ${g_{\textup{\uppercase\expandafter{\romannumeral2}}}(\xi;k)}^{3/2}$ for $k\in(-\infty,\eta]$ with the branch fixed
by the requirement that $g_{\textup{\uppercase\expandafter{\romannumeral2}}}^{2/3}>0$ for $k>\eta$. Introduce
\begin{equation}\label{def:f function}
    f(k):=-\left(\frac{3}{2}g_{\textup{\uppercase\expandafter{\romannumeral2}}}(k)\right)^{\frac{2}{3}}, \quad k\in\mathbb{C}\backslash(-\infty, \eta].
\end{equation}
As $k\rightarrow\eta$, it follows that
\begin{equation}\label{fasyeta}
f(\xi;k)=-2\cdot 6^{\frac{2}{3}}\eta\left(k-\eta\right)\left(1+\frac{k-\eta}{2\eta}+\mathcal{O}\left(\left(k-\eta\right)^2\right)\right).
\end{equation}
Introduce a new scaled variable
\begin{align}
    \zeta_r(\xi;k)=t^{2/3}f(\xi;k),
\end{align}
which implies that
\begin{equation}\label{equ:change of variable zeta(k) RII}
    \frac{4}{3}\zeta_{r}^{\frac{3}{2}}=
        \begin{cases}
        2itg_{\textup{\uppercase\expandafter{\romannumeral2}}}(k), \quad &\im k>0, \\
 -2itg_{\textup{\uppercase\expandafter{\romannumeral2}}}(k), \quad &\im k<0,\\
      2itg_{\textup{\uppercase\expandafter{\romannumeral2}}, +}(k)=-2itg_{\textup{\uppercase\expandafter{\romannumeral2}}, -}(k),  &k<\eta,
        \end{cases}
\end{equation}
where the cut $(\cdot)^{3/2}$ runs along $\mathbb{R}^{-}$ in the Airy parametrix stated in Appendix \ref{appendix:Airy model}.

Let $\mathfrak{r}^{(r)}(k)$ be the truncated Taylor polynomial of 
$r(k)$ at $k=\eta$, which is given by
\begin{equation}\label{r(r)}
    \mathfrak{r}^{(r)}(k)=r(\eta)+r^\prime(\eta)(k-\eta),
\end{equation}
and thus $\mathfrak{r}^{(r)}(k)-r(k)=\mathcal{O}\left(|k-\eta|^2\right)$ as $k\to\eta$. 
Furthermore, we can define $\mathfrak{D}^{(r)}(\xi;k)$, which could be seen as an 
approximation of the function $D_{\textup{\uppercase\expandafter{\romannumeral2}}}(\xi;k)$ defined in \eqref{equ:def D function RII}
as $k\to\eta$:
% an asymptotics of $D_{\textup{\uppercase\expandafter{\romannumeral2}}}(\xi;k)$ as $k\to\eta$ for $(\xi;k)\in\mathcal{R}_{\textup{\uppercase\expandafter{\romannumeral2}}}\times\overline{D_{\varrho}\left(\eta\right)}$:
\begin{equation}\label{D(r)}
    \mathfrak{D}^{(r)}(\xi;k):=\exp\left\{ \frac{X_\eta( k)}{2 \pi i}\left[\int_{c_r}^{\eta} \frac{\frac{1}{2} \log (\mathfrak{r}^{(r)}( s) / \overline{\mathfrak{r}^{(r)}( s)})}{X_{\eta+}(s)(s-k)} \dif s+\int_{\eta}^{c_l} \frac{\log |\mathfrak{r}^{(r)}( s)|}{X_\eta(s)(s-k)} \dif s+\sum_{n=0}^{1} c_n(\xi)\left(k-\eta\right)^n\right]\right\},
\end{equation}
where the real coefficients $c_0(\xi)$ and $c_1(\xi)$ are uniquely determined by the requirement that
\begin{equation*}
    D_{\textup{\uppercase\expandafter{\romannumeral2}}}(\xi; k)=\mathfrak{D}^{(r)}(\xi; k)+\mathcal{O}\left(\left|k-\eta\right|^2\right), \quad  \quad k \rightarrow \eta.
\end{equation*}
Moreover, we can obtain the jump of $\mathfrak{D}^{(r)}(k)$ for $k\in D_{\varrho}\left(\eta\right)\cap \mathbb{R}$ as follows:
\begin{equation}\label{jump of mathfrak D on R}
    \begin{aligned}
    &\frac{\mathfrak{D}^{(r)}_-(k)}{\mathfrak{D}_+^{(r)}(k)}=\frac{1}{|\mathfrak{r}^{(r)}(k)|},\quad &&k\in(\eta,\eta+\varrho),\\
    &\mathfrak{D}^{(r)}_+(k)\mathfrak{D}^{(r)}_-(k)=\frac{\sqrt{\mathfrak{r}^{(r)}(k)}}{\sqrt{\left(\mathfrak{r}^{(r)}\right)^*(k)}},\quad &&k\in(\eta-\varrho,\eta).
\end{aligned}
\end{equation}

% we can solve
% $M^{(r)}$ by using the Airy parametrix exhibited in Appendix \ref{appendix:Airy model}
% in a standard way. More precisely, let us define the local parametrix $\tilde M^{(r)}$ by
Under the change of variable \eqref{equ:change of variable zeta(k) RII}, 
now we define the following local parametrix $\tilde M^{(r)}$ via the 
Airy model problem exhibited in Appendix \ref{appendix:Airy model}:
\begin{equation}\label{R2Localeta}
    \tilde M^{(r)}(x,t;k):=P^{(r)}(\xi;k)\Psi^{(\textnormal{Ai})}\left(\zeta_{r}\left(\xi;k\right)\right)Q^{(r)}(\xi;k),
\end{equation}
where $P^{(r)}$ is a matching factor defined by
\begin{equation}\label{equ:Pr RII}
    P^{(r)}(k):=M^{(\infty)}(k) {Q^{(r)}}(k)^{-1} N^{-1}\zeta_{r}(k)^\frac{\sigma_3}{4}.
\end{equation}
Here $M^{(\infty)}$ is given by \eqref{equ:sol of Minfty RII}, 
$N$ and $Q^{(r)}$ are defined by
\begin{equation}\label{def:Q^{(r)}}
    N=\frac{1}{\sqrt{2}}\begin{pmatrix}1&i\\i&1
        
    \end{pmatrix},\quad Q^{(r)}(k):=
        \begin{cases}
       e^{\frac{\pi}{2}i\sigma_3}\mathfrak{r}^{(r)}(k)^{\frac{\sigma_3}{2}}\mathfrak{D}^{(r)}(k)^{-\sigma_3}\sigma_3e^{\frac{\pi}{2}i\sigma_3},   &k\in\mathbb{C}^{+}\cap \overline{D_{\varrho}\left(\eta\right)}, \\
       i\left(\mathfrak{r}^{(r)}\right)^*(k)^{\frac{\sigma_3}{2}}\mathfrak{D}^{(r)}(k)^{\sigma_3}\sigma_1\sigma_3e^{\frac{\pi}{2}i\sigma_3},   &k\in\mathbb{C}^{-}\cap \overline{D_{\varrho}\left(\eta\right)}.\\
        \end{cases}
\end{equation}
The function $\tilde{M}^{(r)}$ serves as an approximation to $M^{(r)}$ for sufficiently large $t$.
From the definition \eqref{R2Localeta}, its jump matrix $\tilde V^{(r)}(k)$ is given by
\begin{equation}\label{tilde VR}
   \tilde V^{(r)}=\begin{cases}
       \begin{pmatrix}
           1&0\\ \left[\mathfrak{D}^{(r)}\right]^{-2}\mathfrak{r}^{(r)}e^{2itg_{\textup{\uppercase\expandafter{\romannumeral2}}}}&1
       \end{pmatrix},&k\in\Gamma_1^{(r)},\\
       \begin{pmatrix}
0&-1\\1&\frac{\mathfrak{D}_+^{(r)}}{\mathfrak{D}_-^{(r)}}\frac{1}{|\mathfrak{r}^{(r)}|}e^{-2itg_{\textup{\uppercase\expandafter{\romannumeral2}},+}}\end{pmatrix}, &k\in\Gamma_2^{(r)},\\
\begin{pmatrix}
    1&-\left[\mathfrak{D}^{(r)}\right]^{2}\left(\mathfrak{r}^{(r)}\right)^*e^{-2itg_{\textup{\uppercase\expandafter{\romannumeral2}}}}\\0&1
\end{pmatrix}, &k\in\Gamma_3^{(r)},\\
I,&k\in\Gamma_4^{(r)}.
   \end{cases} 
\end{equation}
Recalling the jump of $\mathfrak{D}^{(r)}(k)$ on interval $(\eta-\varrho,\eta+\varrho)$ in \eqref{jump of mathfrak D on R}, 
we can obtain that $P^{(r)}(k)$ is analytic on $D_{\varrho}\left(\eta\right)$.  
Moreover, as $k\to\eta$, it follows from \eqref{equ:Delta_eta asy to eta} and \eqref{fasyeta} that
\begin{equation}\label{asymptotic for P^R}
    P^{(r)}(k)=\begin{pmatrix}
        \frac{-1+i}{2}(6t)^{\frac{1}{6}}(2\eta)^{\frac{1}{2}}&\frac{\frac{-1+i}{2}\left(1+C_\eta\right)}{(6t)^{\frac{1}{6}}(2\eta)^{\frac{1}{2}}}\\
        \frac{-1-i}{2}(6t)^{\frac{1}{6}}(2\eta)^{\frac{1}{2}}&\frac{\frac{1+i}{2}\left(1-C_\eta\right)}{(6t)^{\frac{1}{6}}(2\eta)^{\frac{1}{2}}}
    \end{pmatrix}+\mathcal{O}(|k-\eta|),\quad k\to\eta,
\end{equation}
where $C_\eta$ is defined as in \eqref{C_eta}.

Similarly, the RH problem for $\tilde M^{(l)}$ (an approximation of $M^{(l)}$) 
can be solved in a similar manner by introducing a new scaling variable $\zeta_l(\xi;k)$ and 
a holomorphic factor $P^{(l)}(\xi;k)$, analogous to \eqref{equ:Pr RII} with appropriate modifications. 
Specifically, we define the local parametrix $\tilde M^{(l)}(x,t;k)$ by
\begin{equation*}
    \tilde M^{(l)}(x,t;k):=P^{(l)}(\xi;k)\Psi^{(\textnormal{Ai})}\left(\zeta_{l}(\xi;k)\right)Q^{(l)}(\xi;k),
\end{equation*}
where 
\begin{equation*}
    \zeta_l(\xi,k)=t^{2/3}f(\xi;k),\quad k\in\mathbb{C}\backslash[-\eta, +\infty)
\end{equation*}
with the branch fixed by the requirement that $g_{\textup{\uppercase\expandafter{\romannumeral2}}}^{2/3}>0$ for $k<-\eta$.
The analytic factor $P^{(l)}$ is given by
\begin{equation*}
    P^{(l)}(k):=M^{(\infty)}(k) {Q^{(l)}}(k)^{-1} N^{-1}\zeta_{l}(k)^\frac{\sigma_3}{4},
\end{equation*}
where
\begin{equation}\label{def:Q^{(l)}}
    Q^{(l)}(k):=
        \begin{cases}
      ie^{-\frac{\pi}{2}i\sigma_3}\mathfrak{r}^{(l)}(k)^{-\frac{\sigma_3}{2}}\mathfrak{D}^{(l)}(k)^{\sigma_3}\sigma_1\sigma_3 ,   &k\in\mathbb{C}^{+}\cap \overline{D_{\varrho}\left(-\eta\right)}, \\
      -i\left(\mathfrak{r}^{(l)}\right)^*(k)^{-\frac{\sigma_3}{2}}\mathfrak{D}^{(l)}(k)^{-\sigma_3}e^{-\frac{\pi}{2}i\sigma_3} ,   &k\in\mathbb{C}^{-}\cap \overline{D_{\varrho}\left(-\eta\right)}.\\
        \end{cases}
\end{equation}
Here, $\mathfrak{r}^{(l)}$ and $\mathfrak{D}^{(l)}$ are approximations of $r$ and $D_{\textup{\uppercase\expandafter{\romannumeral2}}}$ near $-\eta$,
which are analogous to \eqref{r(r)} and \eqref{D(r)} respectively. 
It can be obtained from \eqref{def:Q^{(r)}} and \eqref{def:Q^{(l)}} that $Q_-^{(l)}(-\eta)=Q_+^{(r)}(-\eta)^{-1}$. 
As $k\to-\eta$, analogously to \eqref{asymptotic for P^R}, we have that
% $P^{(l)}(k)$ is analytic on $D_{\varrho}\left(-\eta\right)$ with the following asymptotics as $k\to-\eta$
\begin{equation}\label{asymptotic for P^l}
    P^{(l)}(k)=\begin{pmatrix}
        \frac{\frac{-1-i}{2}\left(1-C_\eta\right)}{(6t)^{\frac{1}{6}}(2\eta)^{\frac{1}{2}}}&\frac{1+i}{2}(6t)^{\frac{1}{6}}(2\eta)^{\frac{1}{2}}\\
        \frac{\frac{1-i}{2}\left(1+C_\eta\right)}{(6t)^{\frac{1}{6}}(2\eta)^{\frac{1}{2}}}&\frac{1-i}{2}(6t)^{\frac{1}{6}}(2\eta)^{\frac{1}{2}}
    \end{pmatrix}+\mathcal{O}(|k+\eta|),\quad k\to-\eta.
\end{equation}
Now we introduce the following lemma, which describes the matching conditions between the local parametrices $\tilde M^{(j)}$, $j\in\{r,l\}$ and the global parametrix $M^{(\infty)}$, and the approximation of the jump matrices.
\begin{lemma}\label{Lemma v3-vr-II}
    For $\xi\in\mathcal{R}_\textup{\uppercase\expandafter{\romannumeral2}}$ and $t>0$, the function $\tilde M^{(j)}$ has the following properties for $j\in\{l,r\}$:
    \begin{itemize}
        \item [\rm (a)] As $t\to\infty$, 
        \begin{equation*}
            \|\tilde M^{(j)}(k)M^\infty(k)^{-1}-I\|_{L^\infty(\partial D_\varrho(\pm\eta))}=\mathcal{O}(t^{-1}).
        \end{equation*}
        Here, the script ``$\pm$'' stands for ``$+$'' when $j=r$, and ``$-$'' when $j=l$.
 \item [\rm (b)] Across $\Gamma^{(j)}$, the jump matrix $\tilde V^{(j)}$ of $\tilde M^{(j)}(x,t;k)$ satisfies
 \begin{align*}
    & \|V^{(3)}-\tilde V^{(j)}\|_{L^1(\Gamma^{(j)})}=\mathcal{O}(t^{-2}),\\
    & \|V^{(3)}-\tilde V^{(j)}\|_{L^2(\Gamma^{(j)})}=\mathcal{O}(t^{-\frac{5}{3}}),\\
     & \|V^{(3)}-\tilde V^{(j)}\|_{L^\infty(\Gamma^{(j)})}=\mathcal{O}(t^{-\frac{4}{3}}).
    % &\|V^{(3)}-\tilde V^{(j)}\|_{(L^1\cap L^2\cap L^\infty)(\mathbb{R}\cap D_\varrho(\pm \eta_l))}=\mathcal{O}(t^{-1}).
 \end{align*}
    \end{itemize}
\end{lemma}
\begin{proof}
\hfill
\begin{itemize}
    \item [(a)]  It follows directly from the asymptotic behaviors of $P^{(j)}(k)$ in \eqref{asymptotic for P^R} and \eqref{asymptotic for P^l}, 
    as well as the large-$\zeta$ expansion of $\Psi^{(\textnormal{Ai})}(\zeta)$ in \eqref{asymptotics for Airy model}.
     \item [(b)] The estimates can be obtained by using the definitions of $V^{(3)}$ in \eqref{equ:jump V3 RII} and $\tilde V^{(j)}$ in \eqref{tilde VR},
    along with the asymptotic behavior of the exponential factors $e^{\pm 2itg_{\textup{\uppercase\expandafter{\romannumeral2}}}(k)}$ on $\Gamma^{(j)}$ for $j\in\{l,r\}$.
\end{itemize}

\end{proof}
%\begin{remark}
  %  It can be readily seen from \eqref{asymptotic for P^l} and \eqref{asymptotic for P^R} that $P^{(l)}(-\eta)=-\sigma_1P^{(r)}(\eta)\sigma_1$, which is consistent with the symmetry of $M^{(\infty)}(k)$ and $\zeta_j(k)(j\in\{l,r\})$ near $\pm\eta$.
%\end{remark}

\subsection{Small norm RH problem for $M^{(err)}$}\label{subsec:small norm RH RII}
Define
\begin{align}\label{def:Merr RII}
    M^{(err)}(x,t;k):=
            \begin{cases}
             M^{(3)}(x,t;k)\left(M^{(\infty)}\left(x,t;k\right)\right)^{-1}, & k\in \mathbb{C}\setminus\left(D_{\varrho}\left(\eta\right)\cup D_{\varrho}\left(-\eta\right)\right),\\
             M^{(3)}(x,t;k)\left(\tilde M^{(r)}\left(x,t;k\right)\right)^{-1}, & k\in D_{\varrho}\left(\eta\right),\\
             M^{(3)}(x,t;k)\left(\tilde M^{(l)}\left(x,t;k\right)\right)^{-1}, & k\in D_{\varrho}\left(-\eta\right).
            \end{cases}
    \end{align}
 It is readily seen that $M^{(err)}(k)$ satisfies the following RH problem.
    \begin{RHP}\label{RHP: Merr RII}
    \hfill
    \begin{itemize}
        \item $M^{(err)}(k)$ is holomorphic for $k\in\mathbb{C}\setminus\Gamma^{(err)}$, where
        \begin{equation*}
        \Gamma^{(err)}:=\partial D_{\varrho}\left(\eta\right)\cup\partial D_{\varrho}\left(-\eta\right)\cup \Gamma^{(3)};
        \end{equation*}
        see Figure \ref{fig:jump contour Merr RII} for an illustration.
        \item $M^{(err)}(k)$ has continuous boundary values $M^{(err)}_\pm(k)$ on $k\in \Gamma^{(err)}$ with the jump condition
        \begin{equation*}
            M^{(err)}_{+}(k)=M^{(err)}_{-}(k)V^{(err)}(k),
        \end{equation*}
        where
         \begin{align}\label{equ:jump Verr RII}
           V^{(err)}(k)=
                \begin{cases}
                 M^{(\infty)}_-(k)V^{(3)}(k){M^{(\infty)}_+(k)}^{-1}, & k\in\Gamma^{(3)}\backslash \left(\overline{D_\varrho(\eta)}\cup \overline{D_\varrho(-\eta)}\right),\\
                 \tilde M^{(r)}(k){M^{(\infty)}(k)}^{-1}, & k\in\partial D_\varrho(\eta),\\
               \tilde  M^{(l)}(k){M^{(\infty)}(k)}^{-1}, & k\in\partial D_\varrho(-\eta),\\
                 \tilde M^{(r)}_-(k)V^{(3)}(k){\tilde M^{(r)}_+(k)}^{-1}, & k\in\Gamma^{(3)}\cap D_\varrho(\eta),\\
                \tilde M^{(l)}_-(k)V^{(3)}(k){\tilde M^{(l)}_+(k)}^{-1}, & k\in\Gamma^{(3)}\cap D_\varrho(-\eta).
                \end{cases}    
        \end{align}
        \item As $k\rightarrow\infty$ in $k\in\mathbb{C}\setminus\Gamma^{(err)}$, we have $M^{(err)}(k)=I+\mathcal{O}(k^{-1})$.
        \item As $k\rightarrow\pm \eta$, we have $M^{(err)}(k)=\mathcal{O}(1)$.
    \end{itemize}
    \end{RHP}

\begin{figure}[H]
\begin{center}
    \tikzset{every picture/.style={line width=0.75pt}} %set default line width to 0.75pt
    \begin{tikzpicture}[x=0.75pt,y=0.75pt,yscale=-1,xscale=1]
    %uncomment if require: \path (0,300); %set diagram left start at 0, and has height of 300
    %Straight Lines [id:da6027154388226197]
    \draw    (175,92) -- (234,130) ;
    \draw [shift={(209.54,114.25)}, rotate = 212.78] [color={rgb, 255:red, 0; green, 0; blue, 0 }  ][line width=0.75]    (10.93,-3.29) .. controls (6.95,-1.4) and (3.31,-0.3) .. (0,0) .. controls (3.31,0.3) and (6.95,1.4) .. (10.93,3.29)   ;
    %Straight Lines [id:da4883968644071177]
    \draw    (175,197) -- (233,159) ;
    \draw [shift={(209.02,174.71)}, rotate = 146.77] [color={rgb, 255:red, 0; green, 0; blue, 0 }  ][line width=0.75]    (10.93,-3.29) .. controls (6.95,-1.4) and (3.31,-0.3) .. (0,0) .. controls (3.31,0.3) and (6.95,1.4) .. (10.93,3.29)   ;
    %Shape: Circle [id:dp9867043624777176]
    \draw  [color={rgb, 255:red, 0; green, 0; blue, 0 }  ,draw opacity=1 ] (229,146) .. controls (229,132.19) and (240.19,121) .. (254,121) .. controls (267.81,121) and (279,132.19) .. (279,146) .. controls (279,159.81) and (267.81,171) .. (254,171) .. controls (240.19,171) and (229,159.81) .. (229,146) -- cycle ;
    %Shape: Circle [id:dp030706272146549418]
    \draw  [color={rgb, 255:red, 0; green, 0; blue, 0 }  ,draw opacity=1 ] (391,146) .. controls (391,132.19) and (402.19,121) .. (416,121) .. controls (429.81,121) and (441,132.19) .. (441,146) .. controls (441,159.81) and (429.81,171) .. (416,171) .. controls (402.19,171) and (391,159.81) .. (391,146) -- cycle ;
    %Straight Lines [id:da9608622256105952]
    \draw    (434,128) -- (492,90) ;
    \draw [shift={(468.02,105.71)}, rotate = 146.77] [color={rgb, 255:red, 0; green, 0; blue, 0 }  ][line width=0.75]    (10.93,-3.29) .. controls (6.95,-1.4) and (3.31,-0.3) .. (0,0) .. controls (3.31,0.3) and (6.95,1.4) .. (10.93,3.29)   ;
    %Straight Lines [id:da32150311787190367]
    \draw    (437,160) -- (492,195) ;
    \draw [shift={(469.56,180.72)}, rotate = 212.47] [color={rgb, 255:red, 0; green, 0; blue, 0 }  ][line width=0.75]    (10.93,-3.29) .. controls (6.95,-1.4) and (3.31,-0.3) .. (0,0) .. controls (3.31,0.3) and (6.95,1.4) .. (10.93,3.29)   ;
    \draw  [color={rgb, 255:red, 0; green, 0; blue, 0 }  ,draw opacity=1 ] (410.76,116.06) -- (422.03,122.35) -- (409.37,124.87) ;
    \draw  [color={rgb, 255:red, 0; green, 0; blue, 0 }  ,draw opacity=1 ] (245.96,116.99) -- (258.1,121.37) -- (246.02,125.9) ;
    %Straight Lines [id:da3129636136282947]
    \draw   (391,146) -- (573,145.51) ;
    \draw [shift={(575,145.5)}, rotate = 179.84] [color={rgb, 255:red, 0; green, 0; blue, 0 }  ][line width=0.75]    (10.93,-3.29) .. controls (6.95,-1.4) and (3.31,-0.3) .. (0,0) .. controls (3.31,0.3) and (6.95,1.4) .. (10.93,3.29)   ;
    %Straight Lines [id:da7758624842521005]
    \draw    (279,146) -- (391,146) ;
    \draw [shift={(341,146)}, rotate = 180] [color={rgb, 255:red, 0; green, 0; blue, 0 }  ][line width=0.75]    (10.93,-3.29) .. controls (6.95,-1.4) and (3.31,-0.3) .. (0,0) .. controls (3.31,0.3) and (6.95,1.4) .. (10.93,3.29)   ;
    %Straight Lines [id:da4791272443053982]
    \draw   (93,144) -- (279,146) ;
    % Text Node
    \draw (408,148.4) node [anchor=north west][inner sep=0.75pt]  [font=\tiny]  {$\eta $};
    % Text Node
    \draw (245,148.4) node [anchor=north west][inner sep=0.75pt]  [font=\tiny]  {$-\eta $};
    % Text Node
    \draw (579,137.4) node [anchor=north west][inner sep=0.75pt]    {$\re k$};
    % Text Node
    \draw (300,148.4) node [anchor=north west][inner sep=0.75pt]  [font=\tiny]  {$-c_{r}$};
    % Text Node
    \draw (349,148.4) node [anchor=north west][inner sep=0.75pt]  [font=\tiny]  {$c_{r}$};
    \draw (234,130)--(256,145.4);
    \draw (233,159)--(256,145.4);
    \draw (434,128)--(410,146.4);
    \draw (437,160)--(410,146.4);
     \fill (256,145.4) circle (1.2pt);
          \fill (410,146.4) circle (1.2pt);
               \fill (354,146.4) circle (1.2pt);
               \fill (308,146.4) circle (1.2pt);         
    \end{tikzpicture}
\caption{The jump contours $\Gamma^{(err)}$ of RH problem for $M^{(err)}$ for $\xi\in\mathcal{R}_{\textup{\uppercase\expandafter{\romannumeral2}}}$.}\label{fig:jump contour Merr RII}
\end{center}
\end{figure}

Using items (c)--(d) in Proposition \ref{analytic extension of r RII} and Lemma \ref{Lemma v3-vr-II}, 
a straightforward calculation yields that for $p=1,2,\infty$,
\begin{align}\label{est:Verr-I RII}
    \Vert V^{(err)}(k)-I \Vert_{L^p}=\begin{cases}
         \mathcal{O}(e^{-ct}), & k\in\Gamma^{(err)}\setminus\left(\overline{D_\varrho(\eta)}\cup \overline{D_\varrho(-\eta)}\cup\mathbb{R}\right),\\
         \mathcal{O}(t^{-1}), &k\in\partial D_\varrho(\eta)\cup \partial D_\varrho(-\eta),\\
         \mathcal{O}(t^{-2}), &k\in\mathbb{R},\\
         \mathcal{O}(t^{\omega_p}), &k\in \Gamma^{(err)}\cap \left(D_\varrho(\eta)\cup D_\varrho(-\eta)\right),
        \end{cases}
\end{align}
where $\omega_1=-2$, $\omega_2=-5/3$, $\omega_\infty=-4/3$ and $c>0$. 

It then follows from the small norm RH problem theory \cite{DZ02} that there exists
a unique solution to RH problem \ref{RHP: Merr RII} for large positive $t$. Furthermore,
according to Beals-Coifman theory \cite{BCdecom} again, the solution to $M^{(err)}$ can be given by
\begin{equation}\label{BCsolforErrorR2}
    M^{(err)}(k)=I+\frac{1}{2\pi i}\int_{\Gamma^{(err)}}\frac{\mu(s)(V^{(err)}(s)-I)}{s-k}\dif s,
\end{equation}
where $\mu\in I+L^{2}(\Gamma^{(err)})$ is the unique solution of \eqref{equ:fredholm ope equ}.
Analogous to the estimate \eqref{est:mu-I RI}, we have
\begin{equation}\label{est:mu-I RII}
    \Vert C_{(err)}\Vert \leqslant \Vert C_{-} \Vert_{L^{2}(\Gamma^{(err)})\rightarrow L^{2}(\Gamma^{(err)})}\Vert V^{(err)}-I \Vert_{L^{\infty}(\Gamma^{(err)})}
    \lesssim \mathcal{O}(t^{-1}),
\end{equation}
which implies that $I-C_{(err)}$ is invertible for large positive $t$ in this case.

To evaluate $M^{(err)}_1$ defined in \eqref{equ:Merr_1 RI}, we give the next proposition which 
is an analogue of Proposition \ref{prop: result of Merr_1 RI}.
\begin{proposition}\label{prop: result of Merr_1 RII}
    With $M^{(err)}_1$ defined in \eqref{equ:Merr_1 RI}, we have, as $t\rightarrow+\infty$,
    \begin{equation}\label{equ:Merr_1 RII}
        M^{(err)}_1=\frac{t^{-1}}{1152\eta^2}
        \begin{pmatrix}
            iC_\eta^2[7+5(2\eta)^{\frac{1}{2}}]-12i&-2[7+5(2\eta)^{\frac{1}{2}}]C_\eta\\2[7+5(2\eta)^{\frac{1}{2}}]C_\eta&-iC_\eta^2[7+5(2\eta)^{\frac{1}{2}}]+12i
        \end{pmatrix}+\mathcal{O}(t^{-2}),
    \end{equation}
    with $\eta=\sqrt{-2\xi}$, and $C_\eta$ given by \eqref{C_eta}.
\end{proposition}
\begin{proof}
    Let us divide $M^{(err)}_1$ into four parts by
    \begin{align*}
        &I_1:=-\frac{1}{2\pi i}\int_{\Gamma^{(err)}}\left(\mu(s)-I\right)\left(V^{(err)}(s)-I\right)\dif s, \\
        &I_2:=-\frac{1}{2\pi i}\int_{\Gamma^{(err)}\setminus\left(D_\varrho(\eta)\cup D_\varrho(-\eta)\right)}\left(V^{(err)}(s)-I\right)\dif s, \\
        &I_3:=-\frac{1}{2\pi i}\oint_{\partial D_\varrho(\eta)\cup \partial D_\varrho(-\eta)}\left(V^{(err)}(s)-I\right)\dif s,\\
        &I_4:=-\frac{1}{2\pi i}\int_{\Gamma^{(err)}\cap\left(D_\varrho(\eta)\cup D_\varrho(-\eta)\right)}\left(V^{(err)}(s)-I\right)\dif s.
       \end{align*}
Analogous to Proposition \ref{prop: result of Merr_1 RI}, the main contribution to
$M^{(err)}_1$ stems from the $I_3$. For the other three parts, we have the following estimates:
$I_1=\mathcal{O}(t^{-2})$ by \eqref{est:mu-I RII}, $I_2=I_4=\mathcal{O}(e^{-ct})$ by \eqref{est:Verr-I RII}. 

Now we turn to estimate $I_3$, let us divide $I_3:=I_3^{(r)}+I_3^{(l)}$, where
$I_3^{(r)}$ and $I_3^{(l)}$ are two contour integrations along $\partial D_\varrho(\eta)$ and $\partial D_\varrho(-\eta)$
respectively. Detailed analysis will be given for $I_3^{(r)}$ below, and the similar manners could be
applied to $I_3^{(l)}$. To proceed, let us define
\begin{equation}\label{def:T^{(r)}_j}   T^{(r)}_j(\xi;k):=f^{-3j/2}M^{(\infty)}\left(Q^{(r)}\right)^{-1} \Psi_j^{\mathrm{(Ai)}} Q^{(r)}\left(M^{(\infty)}\right)^{-1}, \quad j\geqslant 1,
\end{equation}
where $f$, $M^{(\infty)}$, $\Psi_j^{\mathrm{(Ai)}}$ and $Q^{(r)}$ are defined
in \eqref{def:f function}, \eqref{equ:sol of Minfty RII}, \eqref{airy N} and \eqref{def:Q^{(r)}}, respectively.
From the definition of \eqref{def:T^{(r)}_j}, it could be verified that the expansion
$\sum_{j=1}^{\infty}t^{-j}T^{(r)}_j$ converges absolutely for large positive $t$ for
that $\xi\in\mathcal{R}_{\textup{\uppercase\expandafter{\romannumeral2}}}$ and $k\in\partial D_\varrho(\eta)$.
It then follows from \eqref{equ:jump Verr RII} that
\begin{align}\label{def:Ir1}
    I_3^{(r)}&=-\frac{1}{2\pi i}\oint_{\partial D_\varrho(\eta)}\left(\sum_{j=1}^{\infty}\frac{T_j^{(r)}(\xi;k)}{t^j}\right)\dif s
    =t^{-1}\textnormal{Res}_{k=\eta}T^{(r)}_{1}(\xi;k)+\mathcal{O}(t^{-2}).
\end{align}
It remains to estimate $\textnormal{Res}_{k=\eta}T^{(r)}_{1}(\xi;k)$.
Since $\textnormal{Res}_{k=\eta}T^{(r)}_{1}(\xi;k)$ is independent of $t$ 
but dependent of the variable $\xi=x/(12t)$, 
then by the expression \eqref{airy N}, we can rewrite
\begin{equation}\label{equ:Tr1 new form}
    T^{(r)}_{1}(\xi;k)=P^{(r)}(\xi,1;k)f^{-\frac{\sigma_3}{4}}N\frac{\Psi_1^{(\text{Ai})}}{f^{\frac{3}{2}}}N^{-1}f^{\frac{\sigma_3}{4}}\left({P^{(r)}(\xi,1;k)}\right)^{-1},
\end{equation}
where $P^{(r)}(\xi,1;k)$ is obtained by setting $t=1$ in the expression of $P^{(r)}(\xi;k)$ given by \eqref{equ:Pr RII}:
% the holomorphic factor near $\eta$ defined by
\begin{equation*}
    P^{(r)}(\xi,1;k):=M^{(\infty)}(\xi;k)\left({Q^{(r)}(\xi;k)}\right)^{-1}N^{-1}(f(\xi;k))^{\frac{\sigma_3}{4}}.
\end{equation*}
On account of \eqref{fasyeta}, we obtain that, for $k\rightarrow \eta$
\begin{align}\label{equ:fasyeta 1/3 3/2}
   \left\{
        \begin{aligned}
        & f^{1/4}=e^{-\frac{i\pi}{4}}2^{\frac{1}{4}}6^{\frac{1}{6}}\eta^{\frac{1}{4}}(k-\eta)^{\frac{1}{4}}\left(1+\frac{k-\eta}{8\eta}+\mathcal{O}\left( \left(k-\eta\right)^2\right)\right),\\
        & f^{3/2}=e^{\frac{i\pi}{2}}2^{\frac{3}{2}}6\eta^{\frac{3}{2}}(k-\eta)^{\frac{2}{3}}\left(1+\frac{3}{4}\frac{k-\eta}{\eta}+\mathcal{O}\left(\left(k-\eta\right)^2\right)\right).
        \end{aligned}
        \right.
\end{align}
Thus, as $k\rightarrow\eta$, it is followed that the middle factor of $T^{(r)}_{1}$ defined in \eqref{equ:Tr1 new form}
is given by
\begin{align}\label{equ:middle factor ktoeta}
    f^{-\frac{\sigma_3}{4}}N\frac{\Psi_1^{(\text{Ai})}}{f^{\frac{3}{2}}}N^{-1}f^{\frac{\sigma_3}{4}}&=\begin{pmatrix}
        0&\frac{5i}{1152\cdot6^{1/3}\eta^2(k-\eta)^2}-\frac{5i}{1152\cdot6^{1/3}\eta^3(k-\eta)}\\-\frac{7i}{96\cdot6^{2/3}\eta(k-\eta)}&0
    \end{pmatrix}+\mathcal{O}(1).
\end{align}
The next task is to compute the quantities $P^{(r)}(\xi,1;\eta)$ and the
$k$-derivative of $P^{(r)}$ at $k=\eta$. Indeed, the former one equals to
\eqref{asymptotic for P^R} at $t=1$. As for the latter one, 
the term $(k-\eta)^{3/2}$ in the expansion of $D_{\textup{\uppercase\expandafter{\romannumeral2}}}$ 
does not contribute to $\textnormal{Res}_{k=\eta}T^{(r)}_{1}$, which indicates that it is not necessary to calculate the full expression for the entries of $\left[P^{(r)}\right]^\prime(\eta)$.
Therefore, substituting $P^{(r)}(\eta)$ given by \eqref{asymptotic for P^R}, $\left[P^{(r)}\right]^\prime(\eta)$ and \eqref{equ:middle factor ktoeta} into \eqref{equ:Tr1 new form},
we obtain
\begin{align}\label{Res:Tr1}
    \textnormal{Res}_{k=\eta}T^{(r)}_{1}=\frac{1}{2304\eta^2}
    \begin{pmatrix}
      iC_\eta^2[7+5(2\eta)^{\frac{1}{2}}]-12i&-[7+5(2\eta)^{\frac{1}{2}}]C_\eta(C_\eta+2)-2\\-[7+5(2\eta)^{\frac{1}{2}}]C_\eta(C_\eta-2)-2& -iC_\eta^2[7+5(2\eta)^{\frac{1}{2}}]+12i
    \end{pmatrix}.
\end{align}
Similar to estimate $I^{(r)}_3$, we also have
\begin{align}\label{def:Il1}
I_3^{(l)}=t^{-1}\textnormal{Res}_{k=-\eta}T^{(l)}_{1}+\mathcal{O}(t^{-2}),
\end{align}
where
\begin{align}\label{Res:Tl1}
    \textnormal{Res}_{k=-\eta}T^{(l)}_{1}=
    \frac{1}{2304\eta^2}
    \begin{pmatrix}
      iC_\eta^2[7+5(2\eta)^{\frac{1}{2}}]-12i&[7+5(2\eta)^{\frac{1}{2}}]C_\eta(C_\eta-2)+2\\ [7+5(2\eta)^{\frac{1}{2}}]C_\eta(C_\eta+2)+2& -iC_\eta^2[7+5(2\eta)^{\frac{1}{2}}]+12i
    \end{pmatrix}.
\end{align}

On account of \eqref{def:Ir1}, \eqref{Res:Tr1}, \eqref{def:Il1} and \eqref{Res:Tl1},
it arrives at the desired result given in \eqref{equ:Merr_1 RII}.
\end{proof}
\subsection{Proof of the part (\textup{\uppercase\expandafter{\romannumeral2}}) of Theorem \ref{thm:mainthm}}
By tracing back the transformations \eqref{def:M1 RII}, \eqref{def:M2 RII}, \eqref{def:M3 RII} and \eqref{def:Merr RII},
we conclude, for $k\in\mathbb{C}\setminus\Gamma^{(3)}$,
\begin{equation}
M(k)=M^{(err)}M^{(\infty)}D_{\textup{\uppercase\expandafter{\romannumeral2}}}^{\sigma_3}(k)e^{-it(g_{\textup{\uppercase\expandafter{\romannumeral2}}}-\theta)\sigma_3},
\end{equation}
where $M^{(err)}$, $M^{(\infty)}$ are defined in 
\eqref{def:Merr RII} and \eqref{equ:sol of Minfty RII} respectively.

Together with the reconstruction formula stated in \eqref{equ:recovering formula}, it
is accomplished that
\begin{align}
    q(x,t)=2i\left[(M^{(err)}_1)_{12}+\lim_{k\rightarrow\infty}k\left(\Delta_{\eta}(k)\right)_{12}\right].
\end{align}
On account of \eqref{equ:sol of Minfty RII} and \eqref{equ:Merr_1 RII}, we obtain
the part (II) of Theorem \ref{thm:mainthm}.

\section{Asymptotic analysis of the RH problem for $M$ in $\mathcal{R}_{{\textup{\uppercase\expandafter{\romannumeral3}}}_{a}}$}\label{sec:asymptotic analysis in RIII}
\subsection{First transformation: $M\rightarrow M^{(1)}$}

\subsubsection{The $g$-function}
For $\xi\in\mathcal{R}_{\textup{\uppercase\expandafter{\romannumeral3}}_a}\cup\mathcal{R}_{\textup{\uppercase\expandafter{\romannumeral3}}_b}$, we introduce
\begin{equation}\label{equ:g function RIII}
g_{\textup{\uppercase\expandafter{\romannumeral3}}}(\xi;k):=(4k^2+12\xi+2c_{r}^2)X_{r}(k), \quad X_{r}(k)=\sqrt{k^2-c_{r}^2},
\end{equation}
where the branch of the square root being chosen such that $X_r(k)=k+\mathcal{O}(k^{-1})$ as $k\rightarrow\infty$.
It's readily verified that the following properties for $g_{\textup{\uppercase\expandafter{\romannumeral3}}}$ defined in \eqref{equ:g function RIII} hold true.
\begin{proposition}The function $g_{\textup{\uppercase\expandafter{\romannumeral3}}}$  defined in \eqref{equ:g function RIII} satisfies the following properties:
    \begin{itemize}
        \item $g_{\textup{\uppercase\expandafter{\romannumeral3}}}(\xi;k)$ is holomorphic for $k\in\mathbb{C}\backslash[-c_{r},c_{r}]$.
        \item As $k\rightarrow\infty$ in $\mathbb{C}\setminus[-c_{r},c_{r}]$,
        we have $g_{\textup{\uppercase\expandafter{\romannumeral3}}}(\xi;k)=\theta(\xi;k)+\mathcal{O}(k^{-1})$.
        \item For $k\in(-c_{r},c_{r})$, $g_{\textup{\uppercase\expandafter{\romannumeral3}},+}(\xi;k)+g_{\textup{\uppercase\expandafter{\romannumeral3}},-}(\xi;k)=0$.
    \end{itemize}
\end{proposition}

A straightforward calculation shows that $\pm \eta_{r}(\xi)=\pm \sqrt{-3\xi-c_r^2/2}$ are the two real zeros
for $\xi\in\mathcal{R}_{\textup{\uppercase\expandafter{\romannumeral3}}_a}$, while $\eta_{r}(\xi)=\pm i\sqrt{3\xi+c_r^2/2}$
are two purely imaginary zeros for $\xi\in\mathcal{R}_{\textup{\uppercase\expandafter{\romannumeral3}}_b}$.
Noticing that $\im g_{\textup{\uppercase\expandafter{\romannumeral3}},+}>0$ for $k\in(-c_r,-\eta_r)\cup(\eta_r, c_r)$,
then the signature table for $\im g_{\textup{\uppercase\expandafter{\romannumeral3}}}$ is illustrated in Figure \ref{fig:signs img RIII}.
% It's readily seen that the $k$-derivative of $g_{\textup{\uppercase\expandafter{\romannumeral3}}}(k)$ is given by
% \begin{equation}
%     g_{\textup{\uppercase\expandafter{\romannumeral3}}}'(k)=\frac{12k\left(k-\eta_r(\xi)\right)\left(k+\eta_r(\xi)\right)}{X_r(k)}, \quad \eta_r(\xi)=\sqrt{-\xi+\frac{c_{r}^2}{2}}\in(0, c_{r}).
% \end{equation}

\begin{figure}[htbp]
\begin{center}
    \tikzset{every picture/.style={line width=0.75pt}} %set default line width to 0.75pt
    \begin{tikzpicture}[x=0.75pt,y=0.75pt,yscale=-1,xscale=1]
    %uncomment if require: \path (0,300); %set diagram left start at 0, and has height of 300
    %Curve Lines [id:da23990346379567207]
    \draw    (372,52) .. controls (346,52) and (345,233) .. (378,233) ;
    %Straight Lines [id:da27525634770111207]
    \draw    (161,139) -- (477,139) ;
    %Curve Lines [id:da4970333595425316]
    \draw    (248,50) .. controls (273,51) and (273,231) .. (245,232) ;
    % Text Node
    \draw (245,141.4) node [anchor=north west][inner sep=0.75pt]  [font=\tiny]  {$-\eta_r$};
    \fill (266,139) circle (1.2pt);
    % Text Node
    \draw (356,142.4) node [anchor=north west][inner sep=0.75pt]  [font=\tiny]  {$\eta_r$};
    \fill (353,139) circle (1.2pt);
    % Text Node
    \draw (171,145.4) node [anchor=north west][inner sep=0.75pt]  [font=\tiny]  {$-c_{l}$};
    % Text Node
    \draw (430,143.4) node [anchor=north west][inner sep=0.75pt]  [font=\tiny]  {$c_{l}$};
    % Text Node
    \draw (202.18,144.13) node [anchor=north west][inner sep=0.75pt]  [font=\tiny,rotate=-1.56]  {$-c_{r}$};
    % Text Node
    \draw (400,144.4) node [anchor=north west][inner sep=0.75pt]  [font=\tiny]  {$c_{r}$};
    % Text Node
    \draw (443,89.4) node [anchor=north west][inner sep=0.75pt]    {$+$};
    % Text Node
    \draw (448,170.4) node [anchor=north west][inner sep=0.75pt]    {$-$};
    % Text Node
    \draw (296,77.4) node [anchor=north west][inner sep=0.75pt]    {$-$};
    % Text Node
    \draw (301,189.4) node [anchor=north west][inner sep=0.75pt]    {$+$};
    % Text Node
    \draw (172,172.4) node [anchor=north west][inner sep=0.75pt]    {$-$};
    % Text Node
    \draw (175,90.4) node [anchor=north west][inner sep=0.75pt]    {$+$};
    % Text Node
   
     \fill (177,139) circle (1.2pt);
    % Text Node
\fill (210,139) circle (1.2pt);
    % Text Node
\fill (435,139) circle (1.2pt);
    % Text Node
  \fill (405,139) circle (1.2pt);
    \end{tikzpicture}
\caption{ Signature table of the function $\im g_{\textup{\uppercase\expandafter{\romannumeral3}}}(\xi;k)$ for $\xi\in\mathcal{R}_{\textup{\uppercase\expandafter{\romannumeral3}}_a}$.}\label{fig:signs img RIII}
\end{center}
\end{figure}

\subsubsection{RH problem for $M^{(1)}$}
With the aid of  function $g_{\textup{\uppercase\expandafter{\romannumeral3}}}$ defined in \eqref{equ:g function RIII}, 
we define $M^{(1)}$ by
\begin{equation}\label{def:M1 RIII}
    M^{(1)}(x,t;k):=M(x,t;k)e^{it\left(g_{\textup{\uppercase\expandafter{\romannumeral3}}}(\xi;k)-\theta(\xi;k)\right)\sigma_3}.
\end{equation}
Then RH problem for $M^{(1)}$ reads as follows:
\begin{RHP}\label{RHP:M1 RIII}
\hfill
\begin{itemize}
    \item $M^{(1)}(k)$ is holomorphic for $k\in\mathbb{C}\backslash\mathbb{R}$.
    \item $M^{(1)}(k)$ has continuous boundary values $M_{\pm}^{(1)}(k)$ on $\mathbb{R}$ with the jump condition
    \begin{equation}
        M^{(1)}_{+}(k)=M^{(1)}_{-}(k)V^{(1)}(k), \quad k\in\mathbb{R},
    \end{equation}
    where
    \begin{equation}\label{equ:jump V1 RIII}
        V^{(1)}(k)=
            \begin{cases}
            \begin{pmatrix}1-rr^* & -r^*e^{-2itg_{\textup{\uppercase\expandafter{\romannumeral3}}}}\\ re^{2itg_{\textup{\uppercase\expandafter{\romannumeral3}}}} & 1\end{pmatrix}, &k\in(-\infty,-c_l)\cup(c_l,+\infty),  \\
            \begin{pmatrix}0 & -r_{-}^*e^{-2itg_{\textup{\uppercase\expandafter{\romannumeral3}}}}\\ r_{+}e^{2itg_{\textup{\uppercase\expandafter{\romannumeral3}}}} & 1 \end{pmatrix}, &k\in (-c_l,-c_{r})\cup(c_{r},c_l),\\
            \begin{pmatrix}0 & -1 \\ 1 & 0\end{pmatrix},  &k\in (-c_{r}, c_{r}).
            \end{cases}
    \end{equation}
    \item As $k\rightarrow\infty$ in $\mathbb{C}\setminus\mathbb{R}$, we have $M^{(1)}(k)=I+\mathcal{O}(k^{-1})$.
    \item $M^{(1)}(k)$ admits the same singular behavior as $M(k)$ at branch points $\pm c_{r}$.
\end{itemize}
\end{RHP}

\subsection{Second transformation: $M^{(1)}\to M^{(2)}$}
In contrast to the previous sections, it is not required to introduce an auxiliary $D$ function
to open lenses. Nevertheless, we still need to perform an analytic approximation for 
the reflection coefficient $r(k)$.
\begin{proposition}[Analytic approximation of $r$]\label{analytic extension of r RIII}
There exist continuous functions
\begin{equation*}
    r_a: \mathcal{R}_{\textup{\uppercase\expandafter{\romannumeral3}}_a}  \times \left(\overline{U_1^{(2)}}\cup\overline{U_2^{(2)}}\right) \rightarrow \mathbb{C} \; \text { and } \; r_r: \mathcal{R}_{\textup{\uppercase\expandafter{\romannumeral3}}_a} \times\left((-\infty,-c_r)\cup(c_r,+\infty)\right) \rightarrow \mathbb{C},
\end{equation*}
which satisfy the following properties:
\begin{itemize}
    \item [\rm (a)] $r(k)=r_a(\xi; k)+r_r(\xi; k)$ for all $(\xi; k) \in \mathcal{R}_{\textup{\uppercase\expandafter{\romannumeral3}}_a}  \times\{(-\infty,-c_r)\cup(c_r,+\infty)\}$.
    \item [\rm (b)] $r_a(-k)=r_a^*(k)$ for $k\in\overline{U_1^{(2)}}\cup\overline{U_2^{(2)}}.$
    \item [\rm (c)]  For all $\xi \in \mathcal{R}_{\textup{\uppercase\expandafter{\romannumeral3}}_a}$, the function $r_a: U_1^{(2)}\cup U_2^{(2)} \rightarrow \mathbb{C}$ is holomorphic. Moreover, for $(\xi;k)\in\mathcal{R}_{\textup{\uppercase\expandafter{\romannumeral3}}_a}\times\left(\overline{U_1^{(2)}}\cup\overline{U_2^{(2)}}\right)$, we have 
\begin{equation*}
    \left|r_a(\xi; k)-r\left(\pm c_r\right)\right| \lesssim \left|k\mp c_r\right|^{1/2} \mathrm{e}^{t|\operatorname{Im} g_{\textup{\uppercase\expandafter{\romannumeral3}}}(\xi;k)|},
\end{equation*}
and
\begin{equation*}
    \left|r_a(\xi; k)\right| \lesssim \frac{\mathrm{e}^{t|\operatorname{Im} g_{\textup{\uppercase\expandafter{\romannumeral3}}}(\xi; k)|}}{1+|k|^2}.
\end{equation*}
\item [\rm (d)] For all $\xi \in \mathcal{R}_{\textup{\uppercase\expandafter{\romannumeral3}}_a}$, the function $r_r\in L^p\left((-\infty,-c_r)\cup(c_r,+\infty)\right),\; p\in[1,+\infty]$, and as $t \rightarrow +\infty$,
\begin{equation*}
    \left\|r_r(\xi;k)\right\|_{L^p\left((-\infty,-c_r)\cup(c_r,+\infty)\right)}=  \mathcal{O}\left(t^{-3}\right).
\end{equation*}
\end{itemize} 
\end{proposition}

\begin{remark}
    It should be noted that item (d) of Proposition \ref{analytic extension of r RIII} provides a decay rate that is different from those 
    in Proposition \ref{analytic extension of r RI} and Proposition \ref{analytic extension of r RII}. 
    This stronger decay rate is essential, since it would ensure that 
    the sub-leading contribution only arises from the local parametrix near the critical points    
    for $\xi\in\mathcal{R}_{\textup{\uppercase\expandafter{\romannumeral3}}_a}$. 
    The validity of this item relies on the initial data possessing sufficient smoothness and decay, 
    as specified in Assumption \ref{assumption on q_0}. 
\end{remark}

Now we are ready to open lenses. Define $M^{(2)}(x,t;k)$ by
\begin{equation}\label{def:M2 RIII}
    M^{(2)}(k)=M^{(1)}(k)G(k),
\end{equation}
where
\begin{equation*}
    G(k):=
        \begin{cases}
        \begin{pmatrix}1 & 0 \\ -r_ae^{2itg_{\textup{\uppercase\expandafter{\romannumeral3}}}} & 1\end{pmatrix}, & k\in U^{(2)}_1\cup U^{(2)}_2, \\
        \begin{pmatrix}1 & -r_a^{*}e^{-2itg_{\textup{\uppercase\expandafter{\romannumeral3}}}} \\ 0 & 1\end{pmatrix}, & k\in U_1^{(2)*}\cup U_2^{(2)*}, \\
        I, &\textnormal{elsewhere}.
        \end{cases}
\end{equation*}
Then RH conditions for $M^{(2)}$ are listed below.
\begin{RHP}
    \hfill
\begin{itemize}
    \item $M^{(2)}(k)$ is holomorphic for $k\in\mathbb{C}\backslash\Gamma^{(2)}$, where $\Gamma^{(2)}:=\cup_{j=1}^{2}(\Gamma^{(2)}_j\cup\Gamma_j^{(2)*})\cup\mathbb{R}$;
    see Figure \ref{fig:U_j domain RIII} for an illustration.
    \item $M^{(2)}(k)$ has continuous boundary values $M_{\pm}^{(2)}(k)$ on $\Gamma^{(2)}$ with the jump condition
    \begin{equation*}
        M^{(2)}_{+}(k)=M^{(2)}_{-}(k)V^{(2)}(k),
    \end{equation*}
    where
    \begin{equation*}
        V^{(2)}(k)=
            \begin{cases}
            \begin{pmatrix} 1 & 0\\ r_ae^{2itg_{\textup{\uppercase\expandafter{\romannumeral3}}}} & 1\end{pmatrix}, &k\in\Gamma^{(2)}_1\cup\Gamma^{(2)}_2,  \\
            \begin{pmatrix} 1 & -r_a^*e^{-2itg_{\textup{\uppercase\expandafter{\romannumeral3}}}}\\ 0 & 1\end{pmatrix}, & k\in\Gamma^{(2)*}_1\cup\Gamma^{(2)*}_2,  \\
            \begin{pmatrix}
                1-r_rr^*_r&-r_r^*e^{-2itg_{\textup{\uppercase\expandafter{\romannumeral3}}}}\\r_re^{2itg_{\textup{\uppercase\expandafter{\romannumeral3}}}}&1
            \end{pmatrix},&k\in(-\infty,-c_r)\cup(c_r,+\infty),\\
            \begin{pmatrix}0 & -1 \\ 1 & 0\end{pmatrix},  & k\in (-c_{r}, c_{r}).
            \end{cases}
    \end{equation*}
    \item As $k\rightarrow\infty$ in $\mathbb{C}\backslash\Gamma^{(2)}$, we have $M^{(2)}(k)=I+\mathcal{O}(k^{-1})$.
    \item As $k\rightarrow\pm c_{r}$, $M^{(2)}(k)=\mathcal{O}((k\mp c_{r})^{-1/4})$.
\end{itemize}
\end{RHP}

\begin{figure}[H]
\begin{center}
    \tikzset{every picture/.style={line width=0.75pt}} %set default line width to 0.75pt
    \begin{tikzpicture}[x=0.75pt,y=0.75pt,yscale=-1,xscale=1]
    %uncomment if require: \path (0,300); %set diagram left start at 0, and has height of 300
    %Straight Lines [id:da804783244241797]
    \draw    (171,150) -- (481,152) ;
    \draw [shift={(332,151.04)}, rotate = 180.37] [color={rgb, 255:red, 0; green, 0; blue, 0 }  ][line width=0.75]    (10.93,-3.29) .. controls (6.95,-1.4) and (3.31,-0.3) .. (0,0) .. controls (3.31,0.3) and (6.95,1.4) .. (10.93,3.29)   ;
    %Straight Lines [id:da4789393691693813]
    \draw    (481,152) -- (588,222) ;
    \draw [shift={(539.52,190.28)}, rotate = 213.19] [color={rgb, 255:red, 0; green, 0; blue, 0 }  ][line width=0.75]    (10.93,-3.29) .. controls (6.95,-1.4) and (3.31,-0.3) .. (0,0) .. controls (3.31,0.3) and (6.95,1.4) .. (10.93,3.29)   ;
    %Straight Lines [id:da7282944804854083]
    \draw    (582,86) -- (481,152) ;
    \draw [shift={(537.36,115.17)}, rotate = 146.84] [color={rgb, 255:red, 0; green, 0; blue, 0 }  ][line width=0.75]    (10.93,-3.29) .. controls (6.95,-1.4) and (3.31,-0.3) .. (0,0) .. controls (3.31,0.3) and (6.95,1.4) .. (10.93,3.29)   ;
    %Straight Lines [id:da8230642158560675]
    \draw    (171,150) -- (64.26,79.61) ;
    \draw [shift={(123.47,118.66)}, rotate = 213.4] [color={rgb, 255:red, 0; green, 0; blue, 0 }  ][line width=0.75]    (10.93,-3.29) .. controls (6.95,-1.4) and (3.31,-0.3) .. (0,0) .. controls (3.31,0.3) and (6.95,1.4) .. (10.93,3.29)   ;
    %Straight Lines [id:da5987324696981353]
    \draw    (69.76,215.63) -- (171,150) ;
    \draw [shift={(125.41,179.55)}, rotate = 147.05] [color={rgb, 255:red, 0; green, 0; blue, 0 }  ][line width=0.75]    (10.93,-3.29) .. controls (6.95,-1.4) and (3.31,-0.3) .. (0,0) .. controls (3.31,0.3) and (6.95,1.4) .. (10.93,3.29)   ;
    %Straight Lines [id:da6755588884525721]
    \draw    (481,152) -- (604,151) ;
    \draw [shift={(548.5,151.45)}, rotate = 179.53] [color={rgb, 255:red, 0; green, 0; blue, 0 }  ][line width=0.75]    (10.93,-3.29) .. controls (6.95,-1.4) and (3.31,-0.3) .. (0,0) .. controls (3.31,0.3) and (6.95,1.4) .. (10.93,3.29)   ;
    %Straight Lines [id:da3217015238494134]
    \draw   (48,151) -- (171,150) ;
    \draw [shift={(115.5,150.45)}, rotate = 179.53] [color={rgb, 255:red, 0; green, 0; blue, 0 }  ][line width=0.75]    (10.93,-3.29) .. controls (6.95,-1.4) and (3.31,-0.3) .. (0,0) .. controls (3.31,0.3) and (6.95,1.4) .. (10.93,3.29)   ;
    %Curve Lines [id:da4365637499478716]
    \draw  [dash pattern={on 0.84pt off 2.51pt}]  (431,63) .. controls (405,63) and (404,244) .. (437,244) ;
    %Curve Lines [id:da5362793018876433]
    \draw  [dash pattern={on 0.84pt off 2.51pt}]  (225,65) .. controls (250,66) and (250,246) .. (222,247) ;
    % Text Node
    \draw (399,155.4) node [anchor=north west][inner sep=0.75pt]  [font=\tiny]  {$\eta_{r} $};
    \fill (412,152) circle (1.2pt);
    % Text Node
    \draw (224,153.4) node [anchor=north west][inner sep=0.75pt]  [font=\tiny]  {$-\eta_{r} $};
     \fill (244,150.5) circle (1.2pt);
    % Text Node
    \draw (173,153.4) node [anchor=north west][inner sep=0.75pt]  [font=\tiny]  {$-c_{r}$};
    
    % Text Node
    \draw (476,158.4) node [anchor=north west][inner sep=0.75pt]  [font=\tiny]  {$c_{r}$};
    % Text Node
    \draw (541,73.4) node [anchor=north west][inner sep=0.75pt]  [font=\tiny]  {$\Gamma _{1}^{( 2)}$};
    % Text Node
    \draw (542,209.4) node [anchor=north west][inner sep=0.75pt]  [font=\tiny]  {$\Gamma _{1}^{( 2) *}$};
    % Text Node
    \draw (106,78.4) node [anchor=north west][inner sep=0.75pt]  [font=\tiny]  {$\Gamma _{2}^{( 2)}$};
    % Text Node
    \draw (95,204.4) node [anchor=north west][inner sep=0.75pt]  [font=\tiny]  {$\Gamma _{2}^{( 2) *}$};
    % Text Node
    \draw (543,114.4) node [anchor=north west][inner sep=0.75pt]  [font=\tiny]  {$U_{1}^{( 2)}$};
    % Text Node
    \draw (546,162.4) node [anchor=north west][inner sep=0.75pt]  [font=\tiny]  {$U_{1}^{( 2) *}$};
    % Text Node
    \draw (85,112.4) node [anchor=north west][inner sep=0.75pt]  [font=\tiny]  {$U_{2}^{( 2)}$};
    % Text Node
    \draw (85,163.4) node [anchor=north west][inner sep=0.75pt]  [font=\tiny]  {$U_{2}^{( 2) *}$};
    % Text Node
  
     \fill (170,150) circle (1.2pt);
    % Text Node
   
        \fill (481,152) circle (1.2pt);
    \end{tikzpicture}
\caption{The jump contours of RH problem for $M^{(2)}$ when $\xi\in\mathcal{R}_{\textup{\uppercase\expandafter{\romannumeral3}}_a}$.}\label{fig:U_j domain RIII}
\end{center}
\end{figure}

\subsection{Analysis of RH problem for $M^{(2)}$}
It's also noted that $V^{(2)}\rightarrow I$ as $t\rightarrow+\infty$ on the contours
$\Gamma^{(2)}_j\cup\Gamma_j^{(2)*}$ for $j=1,2$. Moreover, by item (d) in Proposition \ref{analytic extension of r RIII}, 
we have that, as $t\rightarrow +\infty$, $V^{(2)}(k)\to I$ for $k\in(-\infty,-c_r)\cup (c_r, +\infty)$.
Therefore, it follows that $M^{(2)}$ is approximated, to the leading order, by the global parametrix $M^{(\infty)}$ given below. 

% Note that $V^{(2)}\rightarrow I$ as $t\rightarrow+\infty$ on the contours
% $\Gamma^{(2)}_j\cup\Gamma_j^{(2)*}$ for $j=1,2$, it follows that $M^{(2)}$ is approximated,
% to the leading order, by the global parametrix $M^{(\infty)}$ given below. The sub-leading
% contribution stems from the local behavior near the branch points $\pm c_{r}$, which is well
% approximated by the Bessel parametrix.

\subsubsection{Global parametrix}
%As $t$ large enough, the jump matrix $V^{(2)}$ approaches
%\begin{equation}\label{equ:jump Vinfty RIII}
 %   V^{(\infty)}=\begin{pmatrix} 0 & -1 \\ 1 & 0 \end{pmatrix}, \quad k\in(-c_{r},c_{r}).
%\end{equation}
%For $k\in\mathbb{C}\backslash[-c_{r},c_{r}]$, $V^{(2)}\rightarrow I$ as $t\rightarrow \infty$.
%Then it is naturally established the following parametrix for $M^{(\infty)}$.
\begin{RHP}\label{RHP:Minfty RIII}
\hfill
\begin{itemize}
    \item $M^{(\infty)}(k)$ is holomorphic for $k\in\mathbb{C}\backslash[-c_{r},c_{r}]$.
    \item $M^{(\infty)}(k)$ has continuous boundary values on $(-c_{r},c_{r})$ satisfying the following jump condition
    \begin{align*}
        M_{+}^{(\infty)}(k)=M_{-}^{(\infty)}(k)
        \begin{pmatrix}
            0 & -1\\
            1 & 0
        \end{pmatrix}.
    \end{align*}
    \item As $k\rightarrow\infty$ in $\mathbb{C}\setminus[-c_{r}, c_{r}]$, we have $M^{(\infty)}(k)=I+\mathcal{O}(k^{-1})$.
    \item As $k\rightarrow \pm c_{r}$, $M^{(\infty)}(k)=\mathcal{O}((k\mp c_{r})^{-1/4})$. 
\end{itemize}
\end{RHP}
Then the unique solution of $M^{(\infty)}$ is given by
\begin{align}\label{equ:sol of Minfty RIII}
    M^{(\infty)}=\Delta_r(k)
\end{align}
with $\Delta_{r}(k)$ defined by \eqref{equ:Delta_j} for the subscript $j$ being chosen as $r$.
For later use, the higher order asymptotic behavior of $M^{(\infty)}$ as 
$k\rightarrow c_r$ is
\begin{equation}\label{equ:Delta_eta asy to c_r}
    \begin{aligned}
    M^{(\infty)}(k)&=\frac{(2c_r)^{\frac{1}{4}}}{2(k-c_r)^{1/4}}\left[
        \begin{pmatrix}
            1 & -i\\
            i & 1
        \end{pmatrix}+\frac{(k-c_r)^{\frac{1}{2}}}{(2c_r)^{\frac{1}{2}}}
        \begin{pmatrix}
            1 & i\\
            -i & 1
        \end{pmatrix}+
        \frac{k-c_r}{8c_r}
        \begin{pmatrix}
            1 & -i\\
            i & 1
        \end{pmatrix}\right.\\
        &+\left.\frac{(k-c_r)^{\frac{3}{2}}}{4(2c_r)^{3/2}}
        \begin{pmatrix}
            -1 & -i\\
            i & -1
        \end{pmatrix}
        +\mathcal{O}\left(\left(k-c_r\right)^2\right)
    \right].
\end{aligned}
\end{equation}
\subsubsection{Local parametrices near $\pm c_{r}$}
Let
\begin{align*}
    D_\varrho(c_r)=\left\{k: |k-c_r|<\varrho \right\}, \quad D_\varrho(-c_r)=\left\{k: |k+c_r|<\varrho \right\}
\end{align*}
be two small disks around $\pm c_r$ respectively, where
\begin{equation*}
\varrho<\frac{1}{3}{\rm min}\left\{ |c_r-\eta_r|, |c_l-c_r|  \right\}.
\end{equation*}
For $j\in\{r,l\}$, we intend to solve the following local RH problem for $M^{(j)}$.
\begin{RHP}
\hfill
\begin{itemize}
    \item $M^{(j)}(k)$ is holomorphic for $k\in\overline{D_{\varrho}(\pm c_r)}\setminus\Gamma^{(j)}$ (``$+$'' for $j=r$, and ``$-$'' for $j=l$),
    where
    \begin{align}\label{def:Gamma^(ell) RIII}
        \Gamma^{(r)}:=D_\varrho(c_r)\cap \Gamma^{(2)},\quad \Gamma^{(l)}:=D_\varrho(-c_r)\cap \Gamma^{(2)};
    \end{align}
    \item $M^{(j)}(k)$ has continuous boundary values $M_{\pm}^{(j)}(k)$ on $k\in \Gamma^{(j)}$ with the jump condition
    \begin{align*}
        M^{(j)}_{+}(k)=M^{(j)}_{-}(k)V^{(2)}(k)\big|_{\Gamma^{(j)}}.
    \end{align*}
    \item As $k\rightarrow\infty$ in $\mathbb{C}\setminus\Gamma^{(j)}$, we have $M^{(j)}(k)=I+\mathcal{O}(k^{-1})$.
\end{itemize}
\end{RHP}
 
In the rest part of this section, we focus on the construction of $M^{(r)}$ near $k=c_r$, and 
the construction of $M^{(l)}$ near $-c_r$ can be obtained by using the symmetry.
Introduce the change of variables:
\begin{equation}\label{change of varibale RIII}
    f(k)=-\frac{1}{4}g_{\textup{\uppercase\expandafter{\romannumeral3}}}^2(k),\quad \zeta_r(\xi;k)=t^2f(\xi;k),
\end{equation}
which implies that 
\begin{equation*}
    -4\zeta_r^{\frac{1}{2}}=\begin{cases}
        2itg_{\textup{\uppercase\expandafter{\romannumeral3}}},& \mathrm{Im}k>0,\\
         -2itg_{\textup{\uppercase\expandafter{\romannumeral3}}},& \mathrm{Im}k<0,
    \end{cases}
\end{equation*}
where the cut $(\cdot)^{1/2}$ runs along $\mathbb{R}^{-}$.
As $k\to c_r$, it follows that
\begin{equation}\label{equ:asy f III}
    f(\xi;k)=-18c_r(c_r^2+2\xi)^2(k-c_r)\left[1+\frac{19c_r^2+6\xi}{6c_r(c_r^2+2\xi)}(k-c_r)+\mathcal{O}\left((k-c_r)^2\right)\right].
\end{equation}

Under the change of variable \eqref{change of varibale RIII}, the construction of the local parametrix proceeds in a manner analogous to that described in Section \ref{subsubsec: local para pm eta}. 
For brevity, we outline only the essential steps below.
Define the local parametrix $\tilde{M}^{(r)}$ by
\begin{equation}\label{R3Localcr}
    \tilde M^{(r)}(x,t;k):=P^{(r)}(\xi;k)\Psi^{({\rm Be})}_{\alpha}\left(\zeta_{r}\left(\xi;k\right)\right)Q^{(r)}(\xi;k), \quad \alpha=-\frac{\arg r(c_r)}{\pi} \ {\rm with} \ |r(c_r)|=1.
\end{equation}
Here $P^{(r)}$ is the matching factor given by
\begin{equation}\label{equ:Pr RIII}
    P^{(r)}(k):=M^{(\infty)}(k){Q^{(r)}}(k)^{-1} N^{-1}\zeta_{r}(k)^\frac{\sigma_3}{4}(2\pi)^{\frac{\sigma_3}{2}},
\end{equation}
with $M^{(\infty)}$ defined in \eqref{equ:sol of Minfty RIII}, and
\begin{equation}\label{def:Q^{(r)} III}
    Q^{(r)}(k):=
        \begin{cases}
        \sigma_3,   &k\in\mathbb{C}^{+}\cap \overline{D_{\varrho}\left(c_r\right)}, \\
       \sigma_1,   &k\in\mathbb{C}^{-}\cap \overline{D_{\varrho}\left(c_r\right)}.\\
        \end{cases}
\end{equation}
The function $\tilde{M}^{(r)}$ serves as an approximation to $M^{(r)}$ for sufficiently large $t$.
From the definition \eqref{R3Localcr}, its jump matrix $\tilde V^{(r)}(k)$ is given by
\begin{equation*}
        \tilde V^{(r)}(k)=
            \begin{cases}
            \begin{pmatrix} 1 & 0\\ e^{-\alpha\pi i}e^{2itg_{\textup{\uppercase\expandafter{\romannumeral3}}}} & 1\end{pmatrix}, &k\in\Gamma^{(2)}_1\cap D_\varrho(c_r),  \\
            \begin{pmatrix} 1 & -e^{-\alpha\pi i}e^{-2itg_{\textup{\uppercase\expandafter{\romannumeral3}}}}\\ 0 & 1\end{pmatrix}, & k\in\Gamma^{(2)*}_1\cap D_\varrho(c_r),  \\
            \begin{pmatrix}0 & -1 \\ 1 & 0\end{pmatrix},  & k\in (c_{r}-\varrho, c_{r}),\\
             I,&k\in(c_r,c_r+\varrho).
            \end{cases}
\end{equation*}
It can be verified that $P^{(r)}(k)$ defined in \eqref{equ:Pr RIII} is analytic on $D_{\varrho}\left(c_r\right)$. 
Moreover, following the steps to obtain \eqref{asymptotic for P^R}, we can also derive that as $k\to c_r$
\begin{equation}\label{asymptotic for P^R III}
    P^{(r)}(k)=\begin{pmatrix}
      ( 1-i)(3t)^{\frac{1}{2}}(c_r^3\pi+2c_r\xi\pi)^{\frac{1}{2}}&\frac{\frac{1-i}{4}}{(3t)^{\frac{1}{2}}(c_r^3\pi+2c_r\xi\pi)^{\frac{1}{2}}}\\
       ( 1+i)(3t)^{\frac{1}{2}}(c_r^3\pi+2c_r\xi\pi)^{\frac{1}{2}}&\frac{\frac{-1-i}{4}}{(3t)^{\frac{1}{2}}(c_r^3\pi+2c_r\xi\pi)^{\frac{1}{2}}}
    \end{pmatrix}+\mathcal{O}(|k-c_r|),\quad k\to c_r.
\end{equation}

Similarly, the RH problem for $\tilde M^{(l)}$ (an approximation of $M^{(l)}$) can be solved in a similar manner.
Define the local parametrix
\begin{equation}\label{R3Localcl}
    \tilde M^{(l)}(x,t;k):=P^{(l)}(\xi;k)\Psi^{({\rm Be})}_{\alpha}\left(\zeta_{l}\left(\xi;k\right)\right)Q^{(l)}(\xi;k),
\end{equation}
where
\begin{equation*}
    \zeta_l(\xi;k)=t^2f(\xi;k),\quad P^{(l)}(k)=M^{(\infty)}(k){Q^{(l)}}(k)^{-1} N^{-1}\zeta_{l}(k)^\frac{\sigma_3}{4}(2\pi)^{\frac{\sigma_3}{2}}
\end{equation*}
with
\begin{equation*}
     Q^{(l)}(k):=
        \begin{cases}
        \sigma_1,   &k\in\mathbb{C}^{+}\cap \overline{D_{\varrho}\left(-c_r\right)}, \\
       \sigma_3,   &k\in\mathbb{C}^{-}\cap \overline{D_{\varrho}\left(-c_r\right)}.\\
        \end{cases}
\end{equation*}
As $k\to -c_r$, we have the following asymptotics for $P^{(l)}$:
\begin{equation}\label{asymptotic for P^l III}
     P^{(l)}(k)=\begin{pmatrix}
     \frac{\frac{1+i}{4}}{(3t)^{\frac{1}{2}}(c_r^3\pi+2c_r\xi\pi)^{\frac{1}{2}}}&(- 1-i)(3t)^{\frac{1}{2}}(c_r^3\pi+2c_r\xi\pi)^{\frac{1}{2}}\\\frac{\frac{-1+i}{4}}{(3t)^{\frac{1}{2}}(c_r^3\pi+2c_r\xi\pi)^{\frac{1}{2}}}
       & (-1+i)(3t)^{\frac{1}{2}}(c_r^3\pi+2c_r\xi\pi)^{\frac{1}{2}}
    \end{pmatrix}+\mathcal{O}(|k+c_r|),\quad k\to -c_r.
\end{equation}
Now we introduce the following lemma, which describes the matching conditions between the local parametrices $\tilde M^{(j)}$, $j\in\{r,l\}$ and the global parametrix $M^{(\infty)}$, and the approximation of the jump matrices.
\begin{lemma}\label{Lemma v3-vr-III}
    For $\xi\in\mathcal{R}_{\textup{\uppercase\expandafter{\romannumeral3}}_a}$ and $t>0$, the function $\tilde M^{(j)}$ has the following properties for $j\in\{l,r\}$:
    \begin{itemize}
        \item [\rm (a)] As $t\to\infty$, 
        \begin{equation*}
            \|\tilde M^{(j)}(k)M^\infty(k)^{-1}-I\|_{L^\infty(\partial D_\varrho(\pm c_r))}=\mathcal{O}(t^{-1}).
        \end{equation*}
        Here, the script ``$\pm$'' stands for ``$+$'' when $j=r$, and ``$-$'' when $j=l$.
 \item [\rm (b)] Across $\Gamma^{(j)}$, the jump matrix $\tilde V^{(j)}$ of $\tilde M^{(j)}(x,t;k)$ satisfies
 \begin{align*}
    & \|V^{(2)}-\tilde V^{(j)}\|_{L^1(\Gamma^{(j)})}=\mathcal{O}(t^{-3}),\\
    & \|V^{(2)}-\tilde V^{(j)}\|_{L^2(\Gamma^{(j)})}=\mathcal{O}(t^{-4}),\\
     & \|V^{(2)}-\tilde V^{(j)}\|_{L^\infty(\Gamma^{(j)})}=\mathcal{O}(t^{-1}).
    % &\|V^{(3)}-\tilde V^{(j)}\|_{(L^1\cap L^2\cap L^\infty)(\mathbb{R}\cap D_\varrho(\pm \eta_l))}=\mathcal{O}(t^{-1}).
 \end{align*}
\end{itemize}
\end{lemma}
\begin{proof}
    The proof is similar to that of Lemma \ref{Lemma v3-vr-II}.
\end{proof}

\subsection{Small norm RH problem for $M^{(err)}$}\label{subsec:small norm RH RIII}
Define
\begin{align}\label{def:Merr RIII}
    M^{(err)}(x,t;k):=
            \begin{cases}
             M^{(2)}(x,t;k)\left(M^{(\infty)}\left(x,t;k\right)\right)^{-1}, & k\in \mathbb{C}\setminus\left(D_{\varrho}\left(c_r\right)\cup D_{\varrho}\left(-c_r\right)\right),\\
             M^{(2)}(x,t;k)\left(\tilde M^{(r)}\left(x,t;k\right)\right)^{-1}, & k\in D_{\varrho}\left(c_r\right),\\
             M^{(2)}(x,t;k)\left(\tilde M^{(l)}\left(x,t;k\right)\right)^{-1}, & k\in D_{\varrho}\left(-c_r\right).
            \end{cases}
    \end{align}
 It is readily seen that $M^{(err)}$ satisfies the following RH problem.
    \begin{RHP}\label{RHP: Merr RIII}
    \hfill
    \begin{itemize}
        \item $M^{(err)}(k)$ is holomorphic for $k\in\mathbb{C}\setminus\Gamma^{(err)}$, where
        \begin{align}
        \Gamma^{(err)}:=\partial D_{\varrho}\left(c_r\right)\cup\partial D_{\varrho}\left(-c_r\right)\cup \Gamma^{(3)}.
        \end{align}
        \item $M^{(err)}(k)$ has continuous boundary values $M^{(err)}_\pm(k)$ on $k\in \Gamma^{(err)}$ with the jump condition
        \begin{equation*}
            M^{(err)}_{+}(k)=M^{(err)}_{-}(k)V^{(err)}(k),
        \end{equation*}
        where
         \begin{align}\label{equ:jump Verr RIII}
           V^{(err)}(k)=
                \begin{cases}
                 M^{(\infty)}_-(k)V^{(2)}(k){M^{(\infty)}_+(k)}^{-1}, & k\in\Gamma^{(3)}\backslash \left(\overline{D_\varrho(c_r)}\cup \overline{D_\varrho(-c_r)}\right),\\
                 \tilde M^{(r)}(k){M^{(\infty)}(k)}^{-1}, & k\in\partial D_\varrho(c_r),\\
               \tilde  M^{(l)}(k){M^{(\infty)}(k)}^{-1}, & k\in\partial D_\varrho(-c_r),\\
                 \tilde M^{(r)}_-(k)V^{(2)}(k){\tilde M^{(r)}_+(k)}^{-1}, & k\in\Gamma^{(3)}\cap D_\varrho(c_r),\\
                \tilde M^{(l)}_-(k)V^{(2)}(k){\tilde M^{(l)}_+(k)}^{-1}, & k\in\Gamma^{(3)}\cap D_\varrho(-c_r).
                \end{cases}    
        \end{align}
        \item As $k\rightarrow\infty$ in $k\in\mathbb{C}\setminus\Gamma^{(err)}$, we have $M^{(err)}(k)=I+\mathcal{O}(k^{-1})$.
        \item As $k\rightarrow\pm c_r$, we have $M^{(err)}(k)=\mathcal{O}(1)$.
    \end{itemize}
    \end{RHP}
It follows from Lemma \ref{Lemma v3-vr-III} that for $p=1,2,\infty$,
\begin{align}\label{est:Verr-I RIII}
    \Vert V^{(err)}(k)-I \Vert_{L^p}=\begin{cases}
         \mathcal{O}(e^{-ct}), & k\in\Gamma^{(err)}\setminus\left(\overline{D_\varrho(c_r)}\cup \overline{D_\varrho(-c_r)}\cup\mathbb{R}\right),\\
         \mathcal{O}(t^{-1}), &k\in\partial D_\varrho(c_r)\cup \partial D_\varrho(-c_r),\\
         \mathcal{O}(t^{-3}), &k\in\mathbb{R},\\
         \mathcal{O}(t^{\omega_p}), &k\in \Gamma^{(err)}\cap \left(D_\varrho(c_r)\cup D_\varrho(-c_r)\right),
        \end{cases}
\end{align}
where $\omega_1=-3$, $\omega_2=-4$, $\omega_\infty=-1$ and $c>0$. 

The solution for $M^{(err)}$ can be given by
\begin{equation}\label{BCsolforErrorR3}
    M^{(err)}(k)=I+\frac{1}{2\pi i}\int_{\Gamma^{(err)}}\frac{\mu(s)(V^{(err)}(s)-I)}{s-k}\dif s,
\end{equation}
where $\mu\in I+L^{2}(\Gamma^{(err)})$ is the unique solution of \eqref{equ:fredholm ope equ}.
Analogous to the estimate \eqref{est:mu-I RI}, we have
\begin{equation}\label{est:mu-I RIII}
    \Vert C_{(err)}\Vert \leqslant \Vert C_{-} \Vert_{L^{2}(\Gamma^{(err)})\rightarrow L^{2}(\Gamma^{(err)})}\Vert V^{(err)}-I \Vert_{L^{\infty}(\Gamma^{(err)})}
    \lesssim \mathcal{O}(t^{-1}),
\end{equation}
which implies that $I-C_{(err)}$ is invertible for large positive $t$ in this case.

Now we still give a proposition to evaluate the asymptotic behavior 
of $M^{(err)}$ at infinity, which is analogous to Proposition \ref{prop: result of Merr_1 RII}.
\begin{proposition}\label{prop: result of Merr_1 RIII}
    With $M^{(err)}_1$ defined in \eqref{equ:Merr_1 RI}, we have, as $t\rightarrow+\infty$,
    \begin{equation}\label{equ:Merr_1 RIII}
        M^{(err)}_1=\begin{pmatrix}
            -\frac{(\mu-2)it^{-1}}{96(c_r^2+2\xi)}&\frac{(\mu-1)(\mu+9)it^{-2}}{9216c_r(c_r^2+2\xi)^2}\\
            -\frac{(\mu-1)(\mu+9)it^{-2}}{9216c_r(c_r^2+2\xi)^2}&\frac{(\mu-2)it^{-1}}{96(c_r^2+2\xi)}
        \end{pmatrix}+\mathcal{O}(t^{-3}),
    \end{equation}
    where $\mu=\frac{4\left[\arg r(c_r)\right]^2}{\pi^2}$.
\end{proposition}
\begin{proof}
    Let us divide $M^{(err)}_1$ into four parts by
    \begin{align*}
        &I_1:=-\frac{1}{2\pi i}\int_{\Gamma^{(err)}}\left(\mu(s)-I\right)\left(V^{(err)}(s)-I\right)\dif s, \\
        &I_2:=-\frac{1}{2\pi i}\int_{\Gamma^{(err)}\setminus\left(D_\varrho(c_r)\cup D_\varrho(-c_r)\right)}\left(V^{(err)}(s)-I\right)\dif s, \\
        &I_3:=-\frac{1}{2\pi i}\oint_{\partial D_\varrho(c_r)\cup \partial D_\varrho(-c_r)}\left(V^{(err)}(s)-I\right)\dif s,\\
        &I_4:=-\frac{1}{2\pi i}\int_{\Gamma^{(err)}\cap\left(D_\varrho(c_r)\cup D_\varrho(-c_r)\right)}\left(V^{(err)}(s)-I\right)\dif s.
       \end{align*}
From \eqref{est:mu-I RIII}, \eqref{est:Verr-I RIII}, and item (d) of Proposition \ref{analytic extension of r RIII}, we deduce that $I_1 = \mathcal{O}(t^{-3})$, $I_2 = \mathcal{O}(e^{-ct})$, and $I_4 = \mathcal{O}(t^{-3})$.

Now we turn to estimate $I_3$. Similar to the steps shown in Proposition \ref{prop: result of Merr_1 RII}, 
we divide $I_3:=I_3^{(r)}+I_3^{(l)}$, where
$I_3^{(r)}$ and $I_3^{(l)}$ are two contour integrations along $\partial D_\varrho(c_r)$ and $\partial D_\varrho(-c_r)$
respectively. We focus the analysis on $I^{(r)}_3$.
To proceed, let us define
\begin{equation}\label{def:T^{(r)}_j_III}   
    T^{(r)}_j(\xi;k):=f^{-j/2}M^{(\infty)}\left(Q^{(r)}\right)^{-1} \Psi_{\alpha, \hspace*{0.1em} j}^{\mathrm{(Be)}} Q^{(r)}\left(M^{(\infty)}\right)^{-1}, \quad j=1,2,
\end{equation}
where $f$, $M^{(\infty)}$, ${\Psi_{\alpha, \hspace*{0.1em} j}^{\mathrm{(Be)}}}$ and $Q^{(r)}$ are defined
in \eqref{change of varibale RIII}, \eqref{equ:sol of Minfty RIII}, \eqref{bessel asy} and \eqref{def:Q^{(r)} III}, respectively. 
It then follows from \eqref{equ:jump Verr RIII} and residue theorem that
\begin{align}\label{def:Ir1_III}  
    I_3^{(r)}&=t^{-1}\textnormal{Res}_{k=c_r}T^{(r)}_{1}(\xi;k)+t^{-2}\textnormal{Res}_{k=c_r}T^{(r)}_{2}(\xi;k)+\mathcal{O}(t^{-3}).
\end{align}
It remains to estimate  $\textnormal{Res}_{k=c_r}T^{(r)}_{1}(\xi;k)$ and $\textnormal{Res}_{k=c_r}T^{(r)}_{2}(\xi;k)$.
By the expression \eqref{bessel asy}, we rewrite
\begin{equation}\label{equ:Tr1 new form_III}
    T^{(r)}_{j}(\xi;k)=P^{(r)}(\xi,1;k)(2\pi)^{-\frac{\sigma_3}{2}}f^{-\frac{\sigma_3}{4}}N\frac{\Psi_{\alpha,j}^{(\text{Be})}}{f^{\frac{j}{2}}}N^{-1}f^{\frac{\sigma_3}{4}}(2\pi)^{\frac{\sigma_3}{2}}\left({P^{(r)}(\xi,1;k)}\right)^{-1},\quad j=1,2.
\end{equation}
where $P^{(r)}(x,1;k)$ is the holomorphic factor near $\eta$,
which is obtained by setting $t=1$ in \eqref{equ:Pr RIII}, i.e.,
\begin{equation*}
    P^{(r)}(\xi,1;k):=M^{(\infty)}(x,1;k)\left({Q^{(r)}(x,1;k)}\right)^{-1}N^{-1}(f(x,1;k))^{\frac{\sigma_3}{4}}(2\pi)^{\frac{\sigma_3}{2}}.
\end{equation*}
On account of \eqref{change of varibale RIII}, we obtain that, for $k\rightarrow c_r$, the middle factor of $T^{(r)}_{j}$ defined in \eqref{equ:Tr1 new form_III}
is given by
\begin{equation}\label{equ:middle factor ktoeta III}
    \begin{aligned}
    &f^{-\frac{\sigma_3}{4}}N\frac{\Psi_{\alpha,1}^{(\text{Be})}}{f^{\frac{1}{2}}}N^{-1}f^{\frac{\sigma_3}{4}}=\frac{1}{32}\begin{pmatrix}
        -\mu f^{-\frac{1}{2}}&(\mu-2)if^{-1}\\(-\mu-6)i&\mu f^{-\frac{1}{2}}
    \end{pmatrix},\\
    &f^{-\frac{\sigma_3}{4}}N\frac{\Psi_{\alpha,2}^{(\text{Be})}}{f}N^{-1}f^{\frac{\sigma_3}{4}}=\frac{\mu-1}{512}\begin{pmatrix}
       (\mu-9) f^{-1}&0\\0&(\mu+15) f^{-1}
    \end{pmatrix},
\end{aligned}
\end{equation}
where $\mu=4\left[\arg r(c_r)\right]^2/{\pi^2}$. 
Since $P^{(r)}$ is analytic near $k = c_r$, only the terms involving $f^{-1}$ contribute to the residue of $T_j^{(r)}$ at $c_r$. Substituting $P^{(r)}(c_r)$ from \eqref{asymptotic for P^R III} and the asymptotic expansion of $f$ from \eqref{equ:asy f III} into \eqref{equ:Tr1 new form_III}, we obtain
\begin{align}\label{Res:Tr1_III}
    \textnormal{Res}_{k=c_r}T^{(r)}_{1}=\frac{\mu-2}{192(c_r^2+2\xi)}
    \begin{pmatrix}
      -i&1\\1& i
    \end{pmatrix},\quad \textnormal{Res}_{k=c_r}T^{(r)}_{2}=\frac{(\mu-1)(\mu+9)}{18432c_r(c_r^2+2\xi)^2}
    \begin{pmatrix}
      -1&i\\-i& -1
    \end{pmatrix}.
\end{align}
Similar to estimate $I^{(r)}_3$, we also have
\begin{align}\label{def:Il1_III}
I_3^{(l)}&=t^{-1}\textnormal{Res}_{k=-c_r}T^{(l)}_{1}+t^{-2}\textnormal{Res}_{k=-c_r}T^{(l)}_{2}+\mathcal{O}(t^{-3}).
\end{align}
where
\begin{align}\label{Res:Tl1-III}
  \textnormal{Res}_{k=-c_r}T^{(l)}_{1}=\frac{\mu-2}{192(c_r^2+2\xi)}
    \begin{pmatrix}
      -i&-1\\-1& i
    \end{pmatrix},\quad \textnormal{Res}_{k=-c_r}T^{(l)}_{2}=\frac{(\mu-1)(\mu+9)}{18432c_r(c_r^2+2\xi)^2}
    \begin{pmatrix}
      1&i\\-i& 1
    \end{pmatrix}.
\end{align}

Substituting \eqref{def:Ir1_III}, \eqref{Res:Tr1_III}, \eqref{def:Il1_III} and \eqref{Res:Tl1-III} into $I_3$,
it arrives at the desired result as stated in \eqref{equ:Merr_1 RIII}.
\end{proof}

\subsection{Proof of the part ($\textup{\uppercase\expandafter{\romannumeral3}}_a$) of Theorem \ref{thm:mainthm}}
By tracing back the transformations \eqref{def:M1 RIII}, \eqref{def:M2 RIII}, and \eqref{def:Merr RIII},
we conclude, for $k\in\mathbb{C}\setminus\Gamma^{(2)}$,
\begin{equation}
    M(k)=M^{(err)}M^{(\infty)}e^{-it(g_{\textup{\uppercase\expandafter{\romannumeral3}}}-\theta)\sigma_3}.
\end{equation}
Together with the reconstruction formula stated in \eqref{equ:recovering formula}, it
is accomplished that
\begin{align}
q(x,t)=2i\left[(M^{(err)}_1)_{12}+\lim_{k\rightarrow\infty}k\left(\Delta_{r}(k)\right)_{12}\right].
\end{align}
On the account of \eqref{equ:sol of Minfty RIII} and \eqref{equ:Merr_1 RIII}, we obtain
the part ($\textup{\uppercase\expandafter{\romannumeral3}}_a$) of Theorem \ref{thm:mainthm}.

\section{Asymptotic analysis of the RH problem for $M$ in $\mathcal{R}_{{\textup{\uppercase\expandafter{\romannumeral3}}}_b}$}\label{sec:asymptotic analysis in RIV}
To investigate large-time asymptotics of the Cauchy problem \eqref{equ:mkdv}--\eqref{boundaryconditions}
in the region $\mathcal{R}_{\textup{\uppercase\expandafter{\romannumeral3}}_b}$, we introduce the same 
function $g_{\textup{\uppercase\expandafter{\romannumeral3}}}$ as defined in \eqref{equ:g function RIII}, 
which has two two purely imaginary zeros $\eta_{r}(\xi)=\pm i\sqrt{3\xi+c_r^2/2}$ for $\xi\in\mathcal{R}_{\textup{\uppercase\expandafter{\romannumeral3}}_b}$. The signature table of the $g_{\textup{\uppercase\expandafter{\romannumeral3}}}$ function is illustrated in Figure \ref{fig:signs img RIV}.

\begin{figure}[H]
\begin{center}
    \tikzset{every picture/.style={line width=0.75pt}} %set default line width to 0.75pt
    \begin{tikzpicture}[x=0.75pt,y=0.75pt,yscale=-1,xscale=1]
    %uncomment if require: \path (0,300); %set diagram left start at 0, and has height of 300
    %Straight Lines [id:da27525634770111207]
    \draw    (161,139) -- (477,139) ;
    %Shape: Arc [id:dp3597890455591639]
    \draw  [draw opacity=0] (209.6,239.24) .. controls (209.59,239.02) and (209.58,238.8) .. (209.58,238.57) .. controls (209.68,222.69) and (254.82,210.09) .. (310.41,210.44) .. controls (365.26,210.78) and (409.79,223.6) .. (410.86,239.21) -- (310.23,239.21) -- cycle ; \draw   (209.6,239.24) .. controls (209.59,239.02) and (209.58,238.8) .. (209.58,238.57) .. controls (209.68,222.69) and (254.82,210.09) .. (310.41,210.44) .. controls (365.26,210.78) and (409.79,223.6) .. (410.86,239.21) ;
    %Shape: Arc [id:dp9446144700687025]
    \draw  [draw opacity=0] (409.86,32.08) .. controls (409.88,32.3) and (409.89,32.52) .. (409.89,32.75) .. controls (409.88,48.63) and (364.8,61.48) .. (309.21,61.44) .. controls (254.36,61.41) and (209.77,48.83) .. (208.6,33.24) -- (309.23,32.67) -- cycle ; \draw   (409.86,32.08) .. controls (409.88,32.3) and (409.89,32.52) .. (409.89,32.75) .. controls (409.88,48.63) and (364.8,61.48) .. (309.21,61.44) .. controls (254.36,61.41) and (209.77,48.83) .. (208.6,33.24) ;
    % Text Node
    \draw (171,145.4) node [anchor=north west][inner sep=0.75pt]  [font=\tiny]  {$-c_{l}$};
    % Text Node
    \draw (424,143.4) node [anchor=north west][inner sep=0.75pt]  [font=\tiny]  {$c_{l}$};
    % Text Node
    \draw (202.18,144.13) node [anchor=north west][inner sep=0.75pt]  [font=\tiny,rotate=-1.56]  {$-c_{r}$};
    % Text Node
    \draw (400,144.4) node [anchor=north west][inner sep=0.75pt]  [font=\tiny]  {$c_{r}$};
    % Text Node
    
      \fill (180,139) circle (1.2pt);
    % Text Node
   
      \fill (213,139) circle (1.2pt);
    % Text Node
   
     \fill (429,139) circle (1.2pt);
    % Text Node
   
    \fill (406,139) circle (1.2pt);
    % Text Node
    \draw (297,40) node [anchor=north west][inner sep=0.75pt]  [font=\tiny]  {$i\sqrt{3\xi+\frac{c_r^2}{2}}$};
    % Text Node
    \draw (297,217.4) node [anchor=north west][inner sep=0.75pt]  [font=\tiny]  {$-i\sqrt{3\xi+\frac{c_r^2}{2}}$};
    % Text Node
   
     \fill (306,61.5) circle (1.2pt);
    % Text Node
   
    % Text Node
  
      \fill (306,210.5) circle (1.2pt);
    % Text Node
   
    % Text Node
    \draw (297,17.4) node [anchor=north west][inner sep=0.75pt]    {$-$};
    % Text Node
    \draw (297,93.4) node [anchor=north west][inner sep=0.75pt]    {$+$};
    % Text Node
    \draw (297,175.4) node [anchor=north west][inner sep=0.75pt]    {$-$};
    % Text Node
    \draw (297,246.4) node [anchor=north west][inner sep=0.75pt]    {$+$};
    \end{tikzpicture}
\caption{ Signature table of the function $\im g_{\textup{\uppercase\expandafter{\romannumeral3}}}(\xi;k)$ for $\xi\in\mathcal{R}_{\textup{\uppercase\expandafter{\romannumeral3}}_b}$.}\label{fig:signs img RIV}
\end{center}
\end{figure}

Define two contours $\Gamma_1=k_1+i\sqrt{3\xi+{c_r^2}/{2}}$, $\Gamma_1^*=k_1-i\sqrt{3\xi+{c_r^2}/{2}}$, which are parallel to $\mathbb{R}$
and two domains $U_1=\{k: 0<\im k<i\sqrt{3\xi+{c_r^2}/{2}}\}$, $U_1^{*}=\{k: -i\sqrt{3\xi+{c_r^2}/{2}}<\im k<0\}$; see Figure \ref{fig:deformtosaddle}
for an illustration. Since the $g_{\textup{\uppercase\expandafter{\romannumeral3}}}$ function takes the same form in both regions $\mathcal{R}_{\textup{\uppercase\expandafter{\romannumeral3}}_a}$ and $\mathcal{R}_{\textup{\uppercase\expandafter{\romannumeral3}}_b}$, 
our analysis starts with the matrix-valued function $M^{(1)}$ satisfying RH problem \ref{RHP:M1 RIII}.

In the region $\mathcal{R}_{\textup{\uppercase\expandafter{\romannumeral3}}_b}$, the signature table of $\im g_{\textup{\uppercase\expandafter{\romannumeral3}}}$ (see Figure~\ref{fig:signs img RIV}) 
indicates that the exponential oscillations in the jump matrices on $(-\infty,-c_r)\cup(c_r,+\infty)$ can be controlled by deforming the contours to the lines $\Gamma_1$ and $\Gamma_1^*$, which pass through the zeros at 
$k = \pm i\sqrt{3\xi + c_r^2/2}$. Specifically, we move the off-diagonal exponential terms from the real axis to these parallel lines, where they decay exponentially as $t \to +\infty$.
The following proposition is required to perform this deformation.
\begin{proposition}[Analytic approximation of $r$]\label{analytic extension of r RIII_b}There exist continuous functions
\begin{equation*}
    r_a: \mathcal{R}_{\textup{\uppercase\expandafter{\romannumeral3}}_b}  \times \left(\overline{U_1}\cup\overline{U_1^*}\right) \rightarrow \mathbb{C} \; \text { and } \; r_r: \mathcal{R}_{\textup{\uppercase\expandafter{\romannumeral3}}_b} \times\left((-\infty,-c_r)\cup(c_r,+\infty)\right) \rightarrow \mathbb{C},
\end{equation*}
which satisfy the following properties:
\begin{itemize}
    \item [\rm (a)] $r(k)=r_a(\xi; k)+r_r(\xi; k)$ for all $(\xi; k) \in \mathcal{R}_{\textup{\uppercase\expandafter{\romannumeral3}}_b}  \times\{(-\infty,-c_r)\cup(c_r,+\infty)\}$.

    \item [\rm (b)]  For all $\xi \in \mathcal{R}_{\textup{\uppercase\expandafter{\romannumeral3}}_b}$, the function $r_a: U_1\cup U_1^{*} \rightarrow \mathbb{C}$ is holomorphic. Moreover, for $k\in\overline{U_1}\cup\overline{U_1^{*}},$ 
\begin{equation*}
    \left|r_a(\xi; k)\right| \lesssim \frac{\mathrm{e}^{t|\operatorname{Im} g_{\textup{\uppercase\expandafter{\romannumeral3}}}(\xi; k)|}}{1+|k|^2}.
\end{equation*}
\item [\rm (c)] For all $\xi \in \mathcal{R}_{\textup{\uppercase\expandafter{\romannumeral3}}_b}$, the function $r_r\in L^p\left((-\infty,-c_r)\cup(c_r,+\infty)\right),\; p\in[1,+\infty]$, and as $t \rightarrow +\infty$,
\begin{equation*}
    \left\|r_r(\xi;k)\right\|_{L^p\left((-\infty,-c_r)\cup(c_r,+\infty)\right)} = \mathcal{O}\left(t^{-2}\right).
\end{equation*}
\end{itemize} 
\end{proposition}   
We are now ready to define the following transformation
\begin{equation*}
    M^{(2)}(x,t;k):=M^{(1)}(x,t;k)G(x,t;k),
\end{equation*}
where
\begin{equation*}
    G(k):=
        \begin{cases}
        \begin{pmatrix}1 & 0 \\ -r_ae^{2itg_{\textup{\uppercase\expandafter{\romannumeral3}}}} & 1\end{pmatrix}, & k\in U_1, \\
        \begin{pmatrix}1 & -r_a^*e^{-2itg_{\textup{\uppercase\expandafter{\romannumeral3}}}} \\  0& 1\end{pmatrix}, & k\in U_1^*, \\
        I, &\textnormal{elsewhere}.
        \end{cases}
\end{equation*}
Then RH conditions for $M^{(2)}$ are listed below.
\begin{RHP}
\hfill
\begin{itemize}
\item $M^{(2)}(k)$ is holomorphic for $k\in\mathbb{C}\backslash([-c_{r},c_{r}]\cup\Gamma_1\cup\Gamma_1^*)$; see Figure \ref{fig:deformtosaddle} for an illustration.
\item $M^{(2)}(k)$ has continuous boundary values $M_{\pm}^{(2)}(k)$ on $\mathbb{R}\cup\Gamma_1\cup\Gamma_1^*$ with the jump condition
\begin{equation*}
    M^{(2)}_{+}(k)=M^{(2)}_{-}(k)V^{(2)}(k),
\end{equation*}
where
\begin{equation*}
    V^{(2)}(k)=
        \begin{cases}
        \begin{pmatrix}1 & 0 \\ r_ae^{2itg_{\textup{\uppercase\expandafter{\romannumeral3}}}} & 1\end{pmatrix}, & k\in \Gamma_1, \\
        \begin{pmatrix}1 & -r_a^*e^{-2itg_{\textup{\uppercase\expandafter{\romannumeral3}}}} \\0  & 1\end{pmatrix}, & k\in \Gamma_1^*, \\
        \begin{pmatrix}0 & -1 \\ 1 & 0\end{pmatrix}, & k\in (-c_{r}, c_{r}),\\
         \begin{pmatrix}
                1-r_rr^*_r&-r_r^*e^{-2itg_{\textup{\uppercase\expandafter{\romannumeral3}}}}\\r_re^{2itg_{\textup{\uppercase\expandafter{\romannumeral3}}}}&1
            \end{pmatrix},&k\in(-\infty,-c_r)\cup(c_r,+\infty).
        \end{cases}
\end{equation*}
\item As $k\rightarrow\infty$ in $\mathbb{C}\backslash(\mathbb{R}\cup\Gamma_1\cup\Gamma_1^*)$, $M^{(2)}(k)=I+\mathcal{O}(k^{-1})$.
\item As $k\rightarrow \pm c_{r}$, $M^{(2)}=\mathcal{O}((k\mp c_{r})^{-1/4})$.
\end{itemize}
\end{RHP}

\begin{figure}[H]
\begin{center}
\tikzset{every picture/.style={line width=0.75pt}} %set default line width to 0.75pt
\begin{tikzpicture}[x=0.75pt,y=0.75pt,yscale=-1,xscale=1]
%uncomment if require: \path (0,300); %set diagram left start at 0, and has height of 300
%Straight Lines [id:da27525634770111207]
\draw    (213,140.5) -- (405,140.5) ;
\draw [shift={(315,140.5)}, rotate = 180] [color={rgb, 255:red, 0; green, 0; blue, 0 }  ][line width=0.75]    (10.93,-3.29) .. controls (6.95,-1.4) and (3.31,-0.3) .. (0,0) .. controls (3.31,0.3) and (6.95,1.4) .. (10.93,3.29)   ;
%Straight Lines [id:da6919388215126605]
\draw    (162,178) -- (479,178.5) ;
\draw [shift={(326.5,178.26)}, rotate = 180.09] [color={rgb, 255:red, 0; green, 0; blue, 0 }  ][line width=0.75]    (10.93,-3.29) .. controls (6.95,-1.4) and (3.31,-0.3) .. (0,0) .. controls (3.31,0.3) and (6.95,1.4) .. (10.93,3.29)   ;
%Straight Lines [id:da016253489074316674]
\draw    (161,97) -- (478,97.5) ;
\draw [shift={(325.5,97.26)}, rotate = 180.09] [color={rgb, 255:red, 0; green, 0; blue, 0 }  ][line width=0.75]    (10.93,-3.29) .. controls (6.95,-1.4) and (3.31,-0.3) .. (0,0) .. controls (3.31,0.3) and (6.95,1.4) .. (10.93,3.29)   ;
%Straight Lines [id:da926928405378789]
\draw    (157,140.5) -- (213,140.5) ;
%Straight Lines [id:da9032208085691953]
\draw  (405,140.5) -- (476,140.5) ;
% Text Node
\draw (202.18,144.13) node [anchor=north west][inner sep=0.75pt]  [font=\tiny,rotate=-1.56]  {$-c_{r}$};
% Text Node
\draw (400,144.4) node [anchor=north west][inner sep=0.75pt]  [font=\tiny]  {$c_{r}$};
% Text Node

 \fill (212,141) circle (1.2pt);
% Text Node

 \fill (406,141) circle (1.2pt);
% Text Node
%\draw (306,82.4) node [anchor=north west][inner sep=0.75pt]    {$.$};
% Text Node
\draw (319,75) node [anchor=north west][inner sep=0.75pt]  [font=\tiny]  {$i\sqrt{3\xi+\frac{c_r^2}{2}}$};
% Text Node
%\draw (308,165.4) node [anchor=north west][inner sep=0.75pt]    {$.$};
% Text Node
\draw (322.5,181.65) node [anchor=north west][inner sep=0.75pt]  [font=\tiny]  {$-i\sqrt{3\xi+\frac{c_r^2}{2}}$};
% Text Node
\draw (486,91.4) node [anchor=north west][inner sep=0.75pt]  [font=\tiny]  {$\Gamma _{1}$};
% Text Node
\draw (483,168.4) node [anchor=north west][inner sep=0.75pt]  [font=\tiny]  {$\Gamma _{1}^{*}$};
% Text Node
\draw (159,107.4) node [anchor=north west][inner sep=0.75pt]  [font=\tiny]  {$U_{1}$};
% Text Node
\draw (162,148.4) node [anchor=north west][inner sep=0.75pt]  [font=\tiny]  {$U_{1}^{*}$};
% Text Node
\draw (482,135.4) node [anchor=north west][inner sep=0.75pt]  [font=\scriptsize]  {$\mathbb{R}$};
\end{tikzpicture}
    \caption{The jump contours of RH problem for $M^{(2)}$ when $\xi\in\mathcal{R}_{\textup{\uppercase\expandafter{\romannumeral3}}_b}$.}\label{fig:deformtosaddle}
\end{center}
\end{figure}

As $t\rightarrow\infty$, it follows from item (c) of Proposition \ref{analytic extension of r RIII_b} that
\begin{equation}
    M^{(2)}\rightarrow \Delta_{r}(k), \quad \vert V^{(2)}(k)-I \vert=\mathcal{O}\left(t^{-2}\right).
\end{equation}
Using the reconstruction formula \eqref{equ:recovering formula}, we have
\begin{equation*}
    q(x,t)=c_{r}+\mathcal{O}(t^{-2}),
\end{equation*}
which is the part ($\textup{\uppercase\expandafter{\romannumeral3}}_b$) of Theorem \ref{thm:mainthm}.

\section*{Acknowledgments}
The authors thank the anonymous reviewers for their careful review and useful suggestions, which improved the manuscript.
\paragraph{\bf  Funding Declaration.} T.-Y. Xu is partially supported by China Postdoctoral Science Foundation under grant
number 2024M760480 and Shanghai Post-Doctoral Excellence Program under grant number
2024100.

\paragraph{\bf Conflict of Interest.} The author has no conflict of interest to declare that are relevant to the content of this
article

\appendix	
\renewcommand\thefigure{\Alph{section}\arabic{figure}}
\setcounter{figure}{0}
\renewcommand{\theequation}{\thesection.\arabic{equation}}
\setcounter{equation}{0}
\section{Parabolic cylinder parametrix}\label{appendix:pc model}
    Denote by $\kappa$ a complex parameter. Find a $2 \times 2$ matrix-valued function $M^{(\textup{PC})}(\zeta)$ satisfying the following RH problem:
    \begin{RHP}\label{RHP: pc parametrix}
    \hfill
    \begin{itemize}
    \item $M^{(\textup{PC})}(\zeta)$ is holomorphic for $\zeta\in\mathbb{C}\backslash \Gamma^{(\textup{PC})}$,
    where $\Gamma^{(\textup{PC})}:=\left\{\mathbb{R}e^{\pm \frac{\pi}{4}i}\right\}\cup\left\{\mathbb{R}e^{\pm \frac{3\pi i}{4}}\right\}$;
    see Figure \ref{fig:pc jump contours} for an illustration.
    \item For $\zeta\in\Gamma^{(\textup{PC})}$, we have
    \begin{equation*}
    M^{(\textup{PC})}_{+}(\zeta)=M^{(\textup{PC})}_{-}(\zeta)V^{(\textup{PC})}(\zeta),
    \end{equation*}
    where
    \begin{align*}
    V^{(\textup{PC})}(\zeta)=\left\{\begin{array}{ll}
    \zeta^{i\nu \hat{\sigma}_3}e^{-\frac{i\zeta^2}{4}\hat{\sigma}_3}\left(\begin{array}{cc}
    1 & 0\\
    \kappa & 1
    \end{array}\right),  & \zeta\in\mathbb{R}^+e^{i\pi/4},\\[10pt]
    \zeta^{i\nu \hat{\sigma}_3}e^{-\frac{i\zeta^2}{4}\hat{\sigma}_3}\left(\begin{array}{cc}
    1 & -\kappa^*\\
    0 & 1
    \end{array}\right),   & \zeta\in \mathbb{R}^+e^{-i\pi/4},\\[10pt]
    \zeta^{i\nu \hat{\sigma}_3}e^{-\frac{i\zeta^2}{4}\hat{\sigma}_3}\left(\begin{array}{cc}
    1& 0\\
    \frac{\kappa}{1-\kappa\kappa^*}&1
    \end{array}\right),   & \zeta\in \mathbb{R}^+e^{-i3\pi/4},\\[10pt]
    \zeta^{i\nu \hat{\sigma}_3}e^{-\frac{i\zeta^2}{4}\hat{\sigma}_3}\left(\begin{array}{cc}
    1 & -\frac{\kappa^*}{1-\kappa\kappa^*}\\
    0 & 1
    \end{array}\right),   & \zeta\in \mathbb{R}^+e^{i3\pi/4},
    \end{array}\right.
    \end{align*}
    and $\nu=\nu(\eta)$ with some parameter $\eta$.
    \item As $\zeta \rightarrow \infty$ in $\mathbb{C}\backslash \Gamma^{(\textup{PC})}$, we have $M^{(\textup{PC})}(\zeta)= I+M^{(\textup{PC})}_{1}\zeta^{-1}+\mathcal{O}(\zeta^{-2})$.
    \end{itemize}
    \end{RHP}

\begin{figure}[htbp]
        \centering
\tikzset{every picture/.style={line width=0.75pt}} %set default line width to 0.75pt        
\begin{tikzpicture}[x=0.75pt,y=0.75pt,yscale=-1,xscale=1]
%uncomment if require: \path (0,300); %set diagram left start at 0, and has height of 300
%Straight Lines [id:da5918515832937193] 
\draw    (300,160) -- (387,80) ;
\draw [shift={(347.92,115.94)}, rotate = 137.4] [color={rgb, 255:red, 0; green, 0; blue, 0 }  ][line width=0.75]    (10.93,-3.29) .. controls (6.95,-1.4) and (3.31,-0.3) .. (0,0) .. controls (3.31,0.3) and (6.95,1.4) .. (10.93,3.29)   ;
%Straight Lines [id:da922809362099621] 
\draw    (300,160) -- (219,81) ;
\draw [shift={(264.51,125.39)}, rotate = 224.28] [color={rgb, 255:red, 0; green, 0; blue, 0 }  ][line width=0.75]    (10.93,-3.29) .. controls (6.95,-1.4) and (3.31,-0.3) .. (0,0) .. controls (3.31,0.3) and (6.95,1.4) .. (10.93,3.29)   ;
%Straight Lines [id:da848338411520934] 
\draw    (381,239) -- (300,160) ;
\draw [shift={(345.51,204.39)}, rotate = 224.28] [color={rgb, 255:red, 0; green, 0; blue, 0 }  ][line width=0.75]    (10.93,-3.29) .. controls (6.95,-1.4) and (3.31,-0.3) .. (0,0) .. controls (3.31,0.3) and (6.95,1.4) .. (10.93,3.29)   ;
%Straight Lines [id:da7501926791857755] 
\draw    (213,240) -- (300,160) ;
\draw [shift={(260.92,195.94)}, rotate = 137.4] [color={rgb, 255:red, 0; green, 0; blue, 0 }  ][line width=0.75]    (10.93,-3.29) .. controls (6.95,-1.4) and (3.31,-0.3) .. (0,0) .. controls (3.31,0.3) and (6.95,1.4) .. (10.93,3.29)   ;

% Text Node
\draw (295,163.4) node [anchor=north west][inner sep=0.75pt]  [font=\small]  {$0$};
% Text Node
\draw (393,74.4) node [anchor=north west][inner sep=0.75pt]  [font=\small]  {$\mathbb{R}^{+} e{^{\pi i/4}}$};
% Text Node
\draw (389,232.4) node [anchor=north west][inner sep=0.75pt]  [font=\small]  {$\mathbb{R}^{+} e{^{-\pi i/4}}$};
% Text Node
\draw (193,57.4) node [anchor=north west][inner sep=0.75pt]  [font=\small]  {$\mathbb{R}^{+} e{^{3\pi i/4}}$};
% Text Node
\draw (185,241.4) node [anchor=north west][inner sep=0.75pt]  [font=\small]  {$\mathbb{R}^{+} e{^{-3\pi i/4}}$};
\end{tikzpicture}
\caption{The jump contours of RH problem for $M^{(\textup{PC})}$.}
\label{fig:pc jump contours}
\end{figure}

The RH problem \ref{RHP: pc parametrix} admits a unique solution, whose explicit form can be found in \cite{DZAnn}.
Particularly, as $\zeta\to\infty$, we have 
\begin{equation}\label{mpcxieta0}
        M^{(\textup{PC})}=I+\frac{1}{\zeta}
            \left(\begin{array}{cc}
            0 & -i\beta^{(\eta)}\\
            i\overline{\beta^{(\eta)}} & 0
            \end{array}\right)+\mathcal{O}(\zeta^{-2}),
    \end{equation}
    where
    \begin{equation}\label{beta}
        \beta^{(\eta)}=\frac{\sqrt{2\pi}e^{\frac{i\pi}{4}}e^{-\frac{\pi \nu(\eta)}{2}}}{\kappa\Gamma(-i\nu(\eta))},\quad  |\beta^{(\eta)}|=\sqrt{\nu(\eta)} .
    \end{equation}

\section{Airy parametrix}\label{appendix:Airy model}
Fix the four rays on the complex plane
\begin{align*}
    &\Gamma^{(\textnormal{Ai})}_1=\left\{\zeta\in\mathbb{C}\vert \textnormal{arg}\zeta=2\pi/3\right\}, \quad \Gamma^{(\textnormal{Ai})}_2=\left\{\zeta\in\mathbb{C}\vert \arg\zeta=\pi\right\}, \\
    &\Gamma^{(\textnormal{Ai})}_3=\left\{\zeta\in\mathbb{C}\vert \textnormal{arg}\zeta=-2\pi/3\right\}, \quad \Gamma^{(\textnormal{Ai})}_4=\left\{\zeta\in\mathbb{C}\vert \arg\zeta=0\right\},
\end{align*}
all oriented from left-to-right. We define the sectors
\begin{align*}
    S_{1}=\left\{\zeta\in\mathbb{C}\vert \arg\zeta\in(0,2\pi/3)\right\}, \quad S_{2}=\left\{\zeta\in\mathbb{C}\vert \arg\zeta\in(2\pi/3,\pi)\right\}.
\end{align*}
Then $S_3$, $S_4$ are the conjugative sectors of $S_2$, $S_1$ respectively; see Figure \ref{AiryRays}
for an illustration.

\begin{figure}[htbp]
\begin{center}
\tikzset{every picture/.style={line width=0.75pt}} %set default line width to 0.75pt
\begin{tikzpicture}[x=0.75pt,y=0.75pt,yscale=-1,xscale=1]
%uncomment if require: \path (0,300); %set diagram left start at 0, and has height of 300
%Straight Lines [id:da31405109437202383]
\draw    (216,155) -- (311,155) ;
\draw [shift={(269.5,155)}, rotate = 180] [color={rgb, 255:red, 0; green, 0; blue, 0 }  ][line width=0.75]    (10.93,-3.29) .. controls (6.95,-1.4) and (3.31,-0.3) .. (0,0) .. controls (3.31,0.3) and (6.95,1.4) .. (10.93,3.29)   ;
%Straight Lines [id:da8223421811474445]
\draw    (220,218) -- (311,155) ;
\draw [shift={(270.43,183.08)}, rotate = 145.3] [color={rgb, 255:red, 0; green, 0; blue, 0 }  ][line width=0.75]    (10.93,-3.29) .. controls (6.95,-1.4) and (3.31,-0.3) .. (0,0) .. controls (3.31,0.3) and (6.95,1.4) .. (10.93,3.29)   ;
%Straight Lines [id:da38581882424097036]
\draw    (311,155) -- (424,155) ;
\draw [shift={(373.5,155)}, rotate = 180] [color={rgb, 255:red, 0; green, 0; blue, 0 }  ][line width=0.75]    (10.93,-3.29) .. controls (6.95,-1.4) and (3.31,-0.3) .. (0,0) .. controls (3.31,0.3) and (6.95,1.4) .. (10.93,3.29)   ;
%Straight Lines [id:da7872598738349936]
\draw    (219,93) -- (311,155) ;
\draw [shift={(269.98,127.35)}, rotate = 213.98] [color={rgb, 255:red, 0; green, 0; blue, 0 }  ][line width=0.75]    (10.93,-3.29) .. controls (6.95,-1.4) and (3.31,-0.3) .. (0,0) .. controls (3.31,0.3) and (6.95,1.4) .. (10.93,3.29)   ;
% Text Node
\draw (308,160.4) node [anchor=north west][inner sep=0.75pt]  [font=\scriptsize]  {$0$};
% Text Node
\draw (389,162.4) node [anchor=north west][inner sep=0.75pt]  [font=\scriptsize]  {$\Gamma _{4}^{(\textnormal{Ai})}$};
% Text Node
\draw (250,78.4) node [anchor=north west][inner sep=0.75pt]  [font=\scriptsize]  {$\Gamma _{1}^{(\textnormal{Ai})}$};
% Text Node
\draw (186,146.4) node [anchor=north west][inner sep=0.75pt]  [font=\scriptsize]  {$\Gamma _{2}^{(\textnormal{Ai})}$};
% Text Node
\draw (250,202.4) node [anchor=north west][inner sep=0.75pt]  [font=\scriptsize]  {$\Gamma _{3}^{(\textnormal{Ai})}$};
% Text Node
\draw (344,102.4) node [anchor=north west][inner sep=0.75pt]  [font=\scriptsize]  {$S_{1}$};
% Text Node
\draw (217,115.4) node [anchor=north west][inner sep=0.75pt]  [font=\scriptsize]  {$S_{2}$};
% Text Node
\draw (220,177.4) node [anchor=north west][inner sep=0.75pt]  [font=\scriptsize]  {$S_{3}$};
% Text Node
\draw (350,191.4) node [anchor=north west][inner sep=0.75pt]  [font=\scriptsize]  {$S_{4}$};
\end{tikzpicture}
\end{center}
\caption{The jump contours of the model Airy RH problem.}\label{AiryRays}
\end{figure}
Find a $2\times 2$ matrix-valued function $\Psi^{(\textnormal{Ai})}$ be the solution
of the following RH problem:
\begin{RHP}
    \hfill
\begin{itemize}
    \item $\Psi^{(\textnormal{Ai})}(\zeta)$ is holomorphic for $\zeta\in\mathbb{C}\backslash\Gamma^{(\textnormal{Ai})}$,where $\Gamma^{(\textnormal{Ai})}=\cup_{j=1}^{4}\Gamma^{(\textnormal{Ai})}_j$.
    \item For $\zeta\in\Gamma^{(\textnormal{Ai})}$, we have
    \begin{align*}
        \Psi^{(\textnormal{Ai})}_{+}(\zeta)=\Psi^{(\textnormal{Ai})}_{-}(\zeta)V^{(\textnormal{Ai})}(\zeta),
    \end{align*}
    where
    \begin{equation*}
        V^{(\textnormal{Ai})}(\zeta)=
            \begin{cases}
            \begin{pmatrix} 1 & 0\\ e^{\frac{4}{3}\zeta^\frac{3}{2}} & 1\end{pmatrix}, &\zeta\in\Gamma^{(\textnormal{Ai})}_1\cup\Gamma^{(\textnormal{Ai})}_3,  \\
            \begin{pmatrix} 0 & 1 \\ -1 & 0\end{pmatrix},  &\zeta\in\Gamma^{(\textnormal{Ai})}_2, \\
            \begin{pmatrix} 1 & e^{-\frac{4}{3}\zeta^\frac{3}{2}}\\ 0  & 1\end{pmatrix}, &\zeta\in\Gamma^{(\textnormal{Ai})}_4.
            \end{cases}
    \end{equation*}
    \item As $\zeta\to\infty$ in $\mathbb{C}\setminus\Gamma^{(\rm Ai)}$, we have
    % The asymptotic behavior of $\Psi^{\mathrm{(Ai)}}$ as $\zeta \rightarrow \infty$ is given by
\begin{equation}\label{asymptotics for Airy model}
    \Psi^{\mathrm{(Ai)}}(\zeta) =\zeta^{-\frac{\sigma_3}{4}} N\left(I+\sum_{j=1}^{\infty} \frac{\Psi_j^{\mathrm{(Ai)}}}{\zeta^{3 j / 2}}\right), \quad \zeta \rightarrow \infty,
\end{equation}
where
\begin{equation}\label{airy N}
    N=\frac{1}{\sqrt{2}}\left(\begin{array}{cc}
1 & i \\
i & 1
\end{array}\right), \quad 
% \Psi_j^{\mathrm{(Ai)}}=\frac{\mathrm{e}^{\frac{\pi i}{4}}}{\sqrt{2}} N^{-1}\left(\begin{array}{cc}
% 1 & 0 \\
% 0 & -i
% \end{array}\right)\left(\frac{3}{2}\right)^j\left(\begin{array}{cc}
% (-1)^j u_j & u_j \\
% (-1)^{j+1} v_j & v_j
% \end{array}\right) \mathrm{e}^{-\frac{\pi i}{4} \sigma_3},
\Psi_j^{\mathrm{(Ai)}}=-\frac{6^{-2 j}\left(j+\frac{1}{2}\right)_{2 j}}{(6 j-1) j!}\left(\begin{array}{cc}
(-1)^j & -6 i j \\
(-1)^j 6 i j & 1
\end{array}\right), \quad j \in \mathbb{N},
\end{equation}
where the Pochhammer symbol $(a)_j$ is defined by
\begin{equation*}
   (a)_j=\frac{\Gamma(a+j)}{\Gamma(a)}=a(a+1)(a+2) \cdots(a+j-1) . 
\end{equation*}
% with the coefficients $u_j, v_j$ defined by
% \begin{equation*}
%     u_j=\frac{(2 j+1)(2 j+3) \cdots(6 j-1)}{(216)^j j!}, \quad v_j=\frac{6 j+1}{1-6 j} u_j,\quad j \in \mathbb{N}.
% \end{equation*}
\end{itemize}
\end{RHP}
The explicit form of $\Psi^{(\textnormal{Ai})}$ is given by \cite{DeiKriMcVenaZhou}
\begin{equation*}
    \Psi^{(\textnormal{Ai})}(\zeta)=
            \begin{cases}
            \mathcal{A}(\zeta)e^{\frac{2}{3}\zeta^{\frac{3}{2}}\sigma_3}, &\zeta\in S_1, \\
            \mathcal{A}(\zeta)\begin{pmatrix} 1 & 0\\ -1 & 1\end{pmatrix}e^{\frac{2}{3}\zeta^{\frac{3}{2}}\sigma_3}, &\zeta\in S_2, \\
            \mathcal{A}(\zeta)\begin{pmatrix} 1 & 0\\  1 & 1\end{pmatrix}e^{\frac{2}{3}\zeta^{\frac{3}{2}}\sigma_3}, &\zeta\in S_3,\\
            \mathcal{A}(\zeta)e^{\frac{2}{3}\zeta^{\frac{3}{2}}\sigma_3}, &\zeta\in S_4,
            \end{cases}
\end{equation*}
where
\begin{equation*}
    \mathcal{A}(\zeta)=
        \begin{cases}
        \begin{pmatrix} \textnormal{Ai}(\zeta) & \textnormal{Ai}(\omega^2\zeta)\\ \textnormal{Ai}'(\zeta) & \omega^2\textnormal{Ai}'(\omega^2\zeta) \end{pmatrix}e^{-\frac{i\pi}{6}\sigma_3}, \quad \zeta\in\mathbb{C}^{+}, \\
        \begin{pmatrix} \textnormal{Ai}(\zeta) & -\omega^2\textnormal{Ai}(\omega\zeta)\\ \textnormal{Ai}'(\zeta) & -\textnormal{Ai}'(\omega\zeta) \end{pmatrix}e^{-\frac{i\pi}{6}\sigma_3}, \quad \zeta\in\mathbb{C}^{-},
        \end{cases}
\end{equation*}
with $\omega=e^{2\pi i/3}$.
% For the convenience of our use, we write $\Psi_j^{\mathrm{(Ai)}}$ as
% \begin{equation}\label{psi_AI}
%     \Psi_j^{\mathrm{(Ai)}}=-\frac{6^{-2 j}\left(j+\frac{1}{2}\right)_{2 j}}{(6 j-1) j!}\left(\begin{array}{cc}
% (-1)^j & -6 i j \\
% (-1)^j 6 i j & 1
% \end{array}\right), \quad j \in \mathbb{N},
% \end{equation}
% where the Pochhammer symbol $(a)_j$ is defined by
% \begin{equation*}
%    (a)_j=\frac{\Gamma(a+j)}{\Gamma(a)}=a(a+1)(a+2) \cdots(a+j-1) . 
% \end{equation*}

\section{Bessel parametrix}\label{appendix:bessel model}
Fix the three rays on the complex plane
\begin{align}
    \Gamma^{(\textnormal{Be})}_1=\left\{\zeta\in\mathbb{C}\vert \textnormal{arg}\zeta=2\pi/3\right\}, \quad \Gamma^{(\textnormal{Be})}_2=\left\{\zeta\in\mathbb{C}\vert \arg\zeta=\pi\right\}, \nonumber \quad
    \Gamma^{(\textnormal{Be})}_3=\left\{\zeta\in\mathbb{C}\vert \textnormal{arg}\zeta=-2\pi/3\right\},
\end{align}
all oriented from left-to-right. We define the sectors
\begin{align*}
    S_{1}=\left\{\zeta\in\mathbb{C}\vert \arg\zeta\in(-2\pi/3,2\pi/3)\right\}, \quad S_{2}=\left\{\zeta\in\mathbb{C}\vert \arg\zeta\in(2\pi/3,\pi)\right\},
\end{align*}
and $S_3$ is the conjugative sector of $S_2$; see Figure \ref{besselRays}
for an illustration. 

\begin{figure}[H]
\begin{center}
\tikzset{every picture/.style={line width=0.75pt}} %set default line width to 0.75pt
\begin{tikzpicture}[x=0.75pt,y=0.75pt,yscale=-1,xscale=1]
%uncomment if require: \path (0,300); %set diagram left start at 0, and has height of 300
%Straight Lines [id:da31405109437202383]
\draw    (216,155) -- (311,155) ;
\draw [shift={(269.5,155)}, rotate = 180] [color={rgb, 255:red, 0; green, 0; blue, 0 }  ][line width=0.75]    (10.93,-3.29) .. controls (6.95,-1.4) and (3.31,-0.3) .. (0,0) .. controls (3.31,0.3) and (6.95,1.4) .. (10.93,3.29)   ;
%Straight Lines [id:da8223421811474445]
\draw    (220,218) -- (311,155) ;
\draw [shift={(270.43,183.08)}, rotate = 145.3] [color={rgb, 255:red, 0; green, 0; blue, 0 }  ][line width=0.75]    (10.93,-3.29) .. controls (6.95,-1.4) and (3.31,-0.3) .. (0,0) .. controls (3.31,0.3) and (6.95,1.4) .. (10.93,3.29)   ;
%Straight Lines [id:da38581882424097036]

%Straight Lines [id:da7872598738349936]
\draw    (219,93) -- (311,155) ;
\draw [shift={(269.98,127.35)}, rotate = 213.98] [color={rgb, 255:red, 0; green, 0; blue, 0 }  ][line width=0.75]    (10.93,-3.29) .. controls (6.95,-1.4) and (3.31,-0.3) .. (0,0) .. controls (3.31,0.3) and (6.95,1.4) .. (10.93,3.29)   ;
% Text Node
\draw (308,160.4) node [anchor=north west][inner sep=0.75pt]  [font=\scriptsize]  {$0$};
% Text Node

% Text Node
\draw (250,78.4) node [anchor=north west][inner sep=0.75pt]  [font=\scriptsize]  {$\Gamma _{1}^{(\textnormal{Be})}$};
% Text Node
\draw (186,146.4) node [anchor=north west][inner sep=0.75pt]  [font=\scriptsize]  {$\Gamma _{2}^{(\textnormal{Be})}$};
% Text Node
\draw (250,202.4) node [anchor=north west][inner sep=0.75pt]  [font=\scriptsize]  {$\Gamma _{3}^{(\textnormal{Be})}$};
% Text Node
\draw (344,146.4) node [anchor=north west][inner sep=0.75pt]  [font=\scriptsize]  {$S_{1}$};
% Text Node
\draw (217,115.4) node [anchor=north west][inner sep=0.75pt]  [font=\scriptsize]  {$S_{2}$};
% Text Node
\draw (220,177.4) node [anchor=north west][inner sep=0.75pt]  [font=\scriptsize]  {$S_{3}$};
% Text Node

\end{tikzpicture}
\end{center}
\caption{The jump contours of the model Bessel RH problem.}\label{besselRays}
\end{figure}
Let $\alpha\in\mathbb{R}$ and $\mu=4\alpha^2$, we can find a $2\times 2$ matrix-valued function $\Psi^{(\textnormal{Be})}_{\alpha}$ be the solution
of the following RH problem:
\begin{RHP}
    \hfill
\begin{itemize}
    \item $\Psi^{(\textnormal{Be})}_{\alpha}(\zeta)$ is holomorphic for $\zeta\in\mathbb{C}\backslash\Gamma^{(\textnormal{Be})}$,where $\Gamma^{(\textnormal{Be})}=\cup_{j=1}^{3}\Gamma^{(\textnormal{Be})}_j$.
    \item For $\zeta\in\Gamma^{(\textnormal{Be})}$, we have
    \begin{align*}
        \Psi^{(\textnormal{Be})}_{\alpha,+}(\zeta)=\Psi^{(\textnormal{Be})}_{\alpha,-}(\zeta)V^{(\textnormal{Be})}(\zeta),
    \end{align*}
    where
    \begin{equation*}
        V^{(\textnormal{Be})}(\zeta)=
            \begin{cases}
            \begin{pmatrix} 1 & 0\\ e^{\alpha\pi i}e^{-4\zeta^\frac{1}{2}} & 1\end{pmatrix}, 
            &\zeta\in\Gamma^{(\textnormal{Be})}_1,  \\
             \begin{pmatrix} 1 & 0\\ e^{-\alpha\pi i}e^{-4 \zeta^\frac{1}{2}} & 1\end{pmatrix}, 
            &\zeta\in\Gamma^{(\textnormal{Be})}_3,\\
            \begin{pmatrix} 0 & 1 \\ -1 & 0\end{pmatrix},  &\zeta\in\Gamma^{(\textnormal{Be})}_2.
            \end{cases}
    \end{equation*}
   % \item $\Psi^{(\textnormal{Be})}$ has the following behavior as $\zeta\to 0$:
   % \begin{equation*}
  %      \Psi^{(\textnormal{Be})}=\begin{pmatrix}
  %          \log|\zeta| &  \log|\zeta|\\ \log|\zeta|& \log|\zeta|
  %      \end{pmatrix},\quad \zeta\to 0.
 %   \end{equation*}
\end{itemize}
\end{RHP}
The explicit form of  $\Psi^{(\textnormal{Be})}_{\alpha}$ could be found in \cite{bessel}. 
Here, we care about the large-$\zeta$ behavior of $\Psi^{(\textnormal{Be})}_{\alpha}$.
As $\zeta\to\infty$ on $S_1$, it follows that
\begin{equation}\label{bessel-asy}
    \Psi^{(\textnormal{Be})}_{\alpha}(\zeta)=\left(2\pi\zeta^{\frac{1}{2}}\right)^{-\frac{\sigma_3}{2}}N\left(I+\frac{\Psi^{(\textnormal{Be})}_{\alpha,\hspace*{0.1em} 1}}{\zeta^{\frac{1}{2}}}+\frac{\Psi^{(\textnormal{Be})}_{\alpha,\hspace*{0.1em} 2}}{\zeta}+\frac{\Psi^{(\textnormal{Be})}_{\alpha, \hspace*{0.1em} 3}}{\zeta^{\frac{3}{2}}}+\mathcal{O}\left(\zeta^{-2}\right)\right),
\end{equation}
where
\begin{equation}\label{bessel asy}
    \begin{aligned}
    & N=\frac{1}{\sqrt{2}}\left(\begin{array}{cc}
1 & i \\
i & 1
\end{array}\right),\quad \Psi^{(\textnormal{Be})}_{\alpha,\hspace*{0.1em} 1}=\frac{1}{16}\begin{pmatrix}
    -\mu-1&-2i\\-2i&\mu+1
\end{pmatrix},\\
&\Psi^{(\textnormal{Be})}_{\alpha,\hspace*{0.1em} 2}=\frac{\mu-1}{512}\begin{pmatrix}
    \mu+3&-12i\\12i&\mu+3
\end{pmatrix},\quad \Psi^{(\textnormal{Be})}_{\alpha,\hspace*{0.1em} 3}=\frac{(\mu-1)(\mu-9)}{24576}\begin{pmatrix}
    -\mu-5&-30i\\-30i&\mu+5
\end{pmatrix}.
\end{aligned}
\end{equation}

\medskip
\noindent

%%%%%%%%%%%%%%%%%%%%%%%%%%%%%%%%%%%%%%%%

\end{document}